\documentclass{article}
 \usepackage{amsmath}
  \usepackage{amsthm}

 \usepackage{amssymb}

      \usepackage[english]{babel}
         \usepackage{wrapfig}
           \usepackage{graphicx}
             \usepackage{geometry}
               \usepackage[rightcaption]{sidecap}
\usepackage{amsfonts}
\usepackage{marvosym}
\usepackage{pifont}

\usepackage{mathrsfs}
\DeclareMathAlphabet{\mathpzc}{OT1}{pzc}{m}{it}

\allowdisplaybreaks[2]

\usepackage{fancyhdr}

\title{Geodesic Convexity of Small Neighborhood in the Space of K\"ahler Potentials}
\author{Xiuxiong Chen, Mikhail Feldman, Jingchen Hu}

\newcommand{\locationofpictures}{{D:/2018_TexPaper_Hasee/Geodesic_2018_03_FinalVersion/Pictures}}

\date{}

\def\test{0}

\newtheorem{theorem}{Theorem}[section]
\newtheorem{proposition}[theorem]{Proposition}
\newtheorem{lemma}[theorem]{Lemma}
\newtheorem{lemmasubsection}[theorem]{Lemma}

\newtheorem{remark}[theorem]{Remark}
\newtheorem{definition}[theorem]{Definition}
\newtheorem{problem}[theorem]{Problem}
\newtheorem{question}[theorem]{Conjecture/Question}

\numberwithin{equation}{section}

\newcommand{\ER}{\mathbb{R}}
\newcommand{\EZ}{\mathbb{Z}}
\newcommand{\EC}{\mathbb{C}}
\newcommand{\EN}{\mathbb{N}}

\newcommand{\SD}{\mathcal{L}}

\newcommand{\ddbar}{\partial\overline\partial}

\if\test1
\newcommand{\torange}{\triangleright}
\else
\newcommand{\torange}{;\,}
\fi

\newcommand{\requirement}{{4}}
\if\test1
\newcommand{\jop}{\underline{J}}
\else
\newcommand{\jop}{J}
\fi

\if\test1
\newcommand{\fractionSpace}{{{\underline{X}}}}
\else
\newcommand{\fractionSpace}{X}
\fi

\newcommand{\alsoDependonBaseFrac}{{, \fractionSpace}}

\newcommand{\alsoDependonBaseFracornot}{{}}
\newcommand{\onlyDependonBaseFracornot}{{}}

\newcommand{\sS}{{\mathcal{S}}}

\newcommand{\SpaceofKahlerPotential}{{\mathcal{H}}}

\newcommand{\varkappaforgamma}{\varkappa}

\newcommand{\AD}{{\mathcal{A}}}
\newcommand{\ADF}{{\AD_F}}
\newcommand{\ADHF}{{\AD_{\hat F}}}

\newcommand{\LSMfdF}{{\Lambda_F}}

\newcommand{\ADFlambda}{{\AD_{F_\lambda}}}
\newcommand{\ADFnu}{{\AD_{F_\nu}}}

\newcommand{\WV}{{\mathcal{W}_V}}

\newcommand{\ir}{{\gamma}}
\newcommand{\constr}{{{\ir}}}

\newcommand{\half}{{1\over 2}}

\newcommand{\NHl}{{\mathcal{N}_{Hl}}}

\newcommand{\lf}{{\mathfrak{f}}}
\newcommand{\lh}{{\mathfrak{h}}}

\newcommand{\hlf}{{\hat{\mathfrak{f}}}}
\newcommand{\hlh}{{\hat{\mathfrak{h}}}}

\newcommand{\taubar}{{\overline{\tau}}}
\newcommand{\II}{{\mathcal{I}}}

\newcommand{\plusfractionSpace}{+\fractionSpace}
\newcommand{\fivethird}{{5{\plusfractionSpace}}}
\newcommand{\threethird}{{3{\plusfractionSpace}}}

\newcommand{\irtwothird}{{\ir-2}}
\newcommand{\irthreethird}{{\ir-1}}
\newcommand{\irfourthird}{{\ir}}
\newcommand{\irfivethird}{{\ir+1}}
\newcommand{\irseventhird}{{\ir+3}}

\newcommand{\spacecdot}{{\ \cdot\ }}

\newcommand{\dsLG}{{G}}
\newcommand{\dsG}{{\mathcal{G}}}
\newcommand{\ok}{{\kappa}}

\newcommand{\Uk}{{U_{\ok}}}

\newcommand{\Uok}{{U_\ok}}

\newcommand{\Ook}{{O_{\ok}}}

\newcommand{\Uvm}{{U_\varsigma}}
\newcommand{\vm}{{\varsigma}}

\newcommand{\oi}{{\iota}}

\newcommand{\mP}{{\mathcal{P}}}
\newcommand{\GF}{{\dsG_F}}
\newcommand{\hGF}{{\dsG_{\hat F}}}

\newcommand{\ADFi}{{\AD_F^{-1}}}

\newcommand{\ADHFi}{{\AD_{\hat F}^{-1}}}

\newcommand{\zeroptofp}{\mathfrak{z}_0}

\newcommand{\possion}{{\mathpzc{p}}}
\newcommand{\Foliation}{\mathfrak{Fl}}

\newcommand{\cD}{{\overline{D}}}
\newcommand{\ctrR}{{\overline{\trR}}}

\newcommand{\NashSmoothingOperator}{{\mathpzc{S}}}
\newcommand{\tNashSmoothingOperator}{{{\widetilde{\NashSmoothingOperator}}}}

\if\test1

\else

\fi

\if\test1
\newcommand{\trR}{{L\widetilde{\mathcal{R}}}}
\newcommand{\tcC}{{L\widetilde{\mathcal{C}}}}
\newcommand{\rR}{{s\mathcal{R}}}
\newcommand{\cC}{{s\mathcal{C}}}
\else
\newcommand{\rR}{{\widetilde{\mathcal{R}}}}
\newcommand{\cC}{{\widetilde{\mathcal{C}}}}
\newcommand{\trR}{{\mathcal{R}}}
\newcommand{\tcC}{{\mathcal{C}}}
\fi

\if\test1
\newcommand{\Rectangle}{{\APLbox}}
\newcommand{\RectFunc}{{\boxed{Func}}}
\else
\newcommand{\Rectangle}{{\mathfrak{D}}}
\newcommand{\RectFunc}{{\mathfrak{F}}}
\fi

\newcommand{\exth}{{\widetilde h}}

\newcommand{\tmF}{{\mathcal{F}}}

\newcommand{\Str}{{\mathfrak{s}}}

\newcommand{\fourthird}{{4{\plusfractionSpace}}}
\newcommand{\mB}{{\mathcal{B}}}

\newcommand{\sF}{{\mathscr{G}}}

\newcommand{\sP}{{\mathscr{P}}}

\newcommand{\Support}{\text{Supp}}

\if\test1
\newcommand{\bvarphi}{{\Xi}}
\else
\newcommand{\bvarphi}{{\boldsymbol\varphi}}
\fi

\newcommand{\varphit}{{\varphi_t}}

\newcommand{\vffnormfourthird}{{|\bvarphi,\phi|_{\fourthird}}}
\newcommand{\vffnormr}{{|\bvarphi,\phi|_{r+2}}}
\newcommand{\vffnormvarrho}{{|\bvarphi,\phi|_{\varrho+2}}}
\newcommand{\vfcf}{{{\bvarphi},\phi}}

\newcommand{\tT}{{{T}}}
\newcommand{\TtrR}{{{\mathfrak{T}}}}

\newcommand{\DtrR}{{{\mathfrak{P}}}}
\newcommand{\AtrR}{{{\mathfrak{A}}}}

\if\test1\newcommand{\vecu}{\varpi}
\else
\newcommand{\vecu}{{\boldsymbol{u}}}
\fi

\newcommand{\ut}{u_t}

\newcommand{\HtrR}{{{\mathfrak{H}}}}
\newcommand{\prodh}{{\mathpzc{h}}}

\newcommand{\twothird}{{2{\plusfractionSpace}}}

\newcommand{\rhf}{{\mathfrak{f}}}
\newcommand{\rhh}{{\mathfrak{h}}}
\newcommand{\rhb}{{\mathfrak{b}}}
\newcommand{\rhrho}{{\varrho}}
\newcommand{\rhg}{{\mathfrak{g}}}

\newcommand{\Noodle}{\mathfrak{V}}

\newcommand{\tlh}{\widetilde \lh}

\newcommand{\ApU}{\mathcal{U}}
\newcommand{\ApB}{\mathcal{B}_1}
\newcommand{\ApBdelta}{\mathcal{B}_\delta}

\newcommand{\tApU}{{\widetilde\ApU}}

\newcommand{\csdot}{{\ \cdot \ }}

\newcommand{\sN}{{\mathscr{N}}}

\begin{document}
\maketitle

\begin{abstract}
We show that, given $k> \requirement$, $0<\jop<\min\{{1\over 4},{k-\requirement\over 4}\}$, any point in the space of non-degenerate smooth K\"ahler  potentials has a small neighborhood with respect to $C^k$ norm, s.t. any two points in this neighborhood can be connected by a geodesic of at least $C^{k-\jop}$ regularity.

\end{abstract}

\tableofcontents


\section{Introduction}\label{Introduction}
\subsection{Motivations}
Let $(V, [\omega_0])$ be a K\"ahler manifold without boundary.  In 1982, E. Calabi \cite{calabi82}\cite{calabi85} proposed his now famous program of finding the
critical K\"ahler metric which is defined as critical point of the following Calabi energy functional:
\[
Ca(g) = \displaystyle \int_V\; R(g)^2 d v_g
\]
where $R(g)$ is the scalar curvature of the K\"ahler metric $g$ in the K\"ahler
class $[\omega_0].\;$ Since its inception, the problem of establishing existence
of this critical metric (constant scalar curvature K\"ahler metric,
extremal K\"ahler metric) has always been a core problem in K\"ahler geometry.  Over last few decades, many fundamental work emerged in connection with this renown
program of E. Calabi, noticeably,  the seminal work of
Calabi\cite{calabi58},  Yau \cite{Yau10} and more
recently, Chen-Donaldson-Sun\cite{cds12}.  We refer readers to the recent work
of J. Demailly \cite{Demailly} on K\"ahler Einstein metric problems
and Chen-Cheng \cite{cc17} on constant scalar curvature K\"ahler metric problems
for updated references on this program.\\

In a seminal paper \cite{Dona96}, S. K. Donaldson proposed a beautiful program to
attack the existence and uniqueness problem of constant scalar curvature K\"ahler (cscK) metrics.
Donaldson took the point of view that the space of K\"ahler potentials
\[\SpaceofKahlerPotential=\left\{\varphi\in C^\infty(V)\big|\omega_0+\sqrt{-1}\ddbar \varphi >0\right\}\]
 is formally a symmetric space of non-compact type; and
the scalar curvature function is the moment map from the space of almost complex structures compatible with a fixed symplectic form to the Lie algebra of certain infinite dimensional symplectic structure group which is exactly the space of all real valued smooth functions in the manifold.
With this in mind, the problem of finding a critical metric is reduced to finding zero of this moment map. This new point of view leads him to
define a $L^2$ type Riemannian metrics in space $\SpaceofKahlerPotential$:
\[
\|\delta \varphi\|^2_\varphi =  \int_M\; (\delta \varphi)^2 \omega_\varphi^n.
\]
It turns out that this has been introduced by Mabuchi \cite{Mabuchi}, Semmes \cite{Semmes} earlier.  Under this norm,  the equation for smooth geodesic
is
\begin{equation}\label{GeodesicEquationDependingont_201803272016}
{{\partial^2 \varphi}\over {\partial t^2}} - g_\varphi^{\alpha\bar \beta} ({{\partial \varphi}\over {\partial t}})_\alpha ({{\partial \varphi}\over {\partial t}})_{\bar \beta} = 0.
\end{equation}
This gives rise to the Levi-Civita connection in tangent space $T_\varphi \mathcal H$ for any $\varphi \in \mathcal {H}$. One can verify in a straightforward manner
that this space (under $L^2$ norm) is formally a symmetric space of non-compact type which carries a non-positive curvature. In \cite{ChenPrinceton}, Chen proved the existence
of $C^{1,1}$ geodesic segment (in the sense that the Laplacian of potential be bounded) and  used this to show that this is a metric space (Donaldson conjecture). Together with E. Calabi, Chen proved that this is a non-positively curved space
in the sense of Alexandrov.\\

Following T. Mabuchi \cite{Mabuchi}, we introduce the notion of the K-energy through its derivatives.  It is straightforward to prove that the following 1-form
\[
\alpha:  T_{\varphi} {\cal H} \rightarrow \mathbb{R},
\]
which sends a genuine ``vector" $\delta f\in T_{\varphi} {\cal H}  $  to
\[
\int_M\; (\delta f) (R_\varphi -\underline{R}) \omega_\varphi^n,
\]
is a closed $1-$form in the tangent space $T\cal H.\;$ Since $\cal H$ is linearly convex, this gives rise to an energy potential $\mathcal E\;$, whose critical points are precisely the constant scalar curvature
K\"ahler metrics (cscK) introduced by E. Calabi.  In literature, $\mathcal E$ is usually called ``K-energy functional" and one important property of $\mathcal E$ is that it is convex
over smooth geodesic segment.  In other words, we have
\[
{{d^2 {\mathcal E}}\over d t^2} = \displaystyle \int_M\;  | {\cal D}({{\partial \varphi} \over {\partial t}})|^2_\varphi \;\omega_\varphi^n \geq 0
\]
over a smooth geodesic segment $\varphi(t)$. Here $\cal D$  represents
second order pure covariant derivatives, i.e., for any function $f \in C^\infty(M)$ we have
\[
{\cal D}(f) = f_{,\alpha \beta} \;d z^\alpha \otimes d z^\beta.
\] The equality holds if and only if this geodesic path represents a path of holomorphic transformation.
It follows that if  the space of K\"ahler potentials is  convex by smooth geodesics, then cscK metrics must be unique up to holomorphic transformation. The uniqueness problem
goes back to E. Calabi where he proved it for K\"ahler Einstein metrics with negative scalar curvature. For K\"ahler Einstein metrics with positive scalar curvature, this is due to
Bando-Mabuchi\cite{BandoMabuchi}.
For cscK metric problem,  first such application comes from  \cite{ChenPrinceton} where the author proved the uniqueness of cscK
metric if the first Chern class is negative. The restriction of negative first Chern class is there to compensate the lack of higher order regularity of Chen's original solution.  This is generalized in \cite{Donaldson2001Project} by S. Donaldson to  algebraic manifolds with discrete automorphism groups.  It follows with T. Mabuchi for extremal K\"ahler metric \cite{MabuchiCPZ}  and Chen-Tian for general K\"ahler manifolds
\cite{ChenXXTianGang}.  To bypass the higher order regularity issue of geodesic segment,
the first named author conjectured that the K-energy is convex over $C^{1,1}$ geodesic segment. This is proved  in a fundamental paper by Berndtsson-Berman\cite{BerndtssonBerman} (c.f. Chen-Paun-Li\cite{ChenPaunLi}).  As a corollary, this leads to uniqueness of cscK metrics in the most general form in Berndtsson-Berman\cite{BerndtssonBerman} (c.f. Chen-Paun-Zeng\cite{ChenPaunZeng}). \\

The research on regularity of geodesic and applications has been very intense as illustrated by this important work of Berndtsson-Berman\cite{BerndtssonBerman}
(c.f. Chen-Paun-Zeng\cite{ChenPaunZeng}) and references therein. However, the attempts to improve regularity beyond Chen's $C^{1,1}$ solution has not been this successful up to now,
although important progress was made in J. Chu, B. Wenkove and V.Tossati \cite{Jianchun}, who proved that
Chen's weak geodesic has full Hessian bound.
More importantly, the examples constructed by
 Lempert-Vivas \cite{LempertLizVivas},  Lempert-Darvas\cite{LempertDarvas}  show
that 
$C^{1,1}$ regularity of geodesic is the optimal global regularity in general. Therefore, it is of great interest to understand the problem of  higher order regularity of geodesic
segments: To what extent, the original
conjecture of Donaldson, which predicts that $\cal H$ is convex by smooth geodesic segments, holds with some necessary modifications?
As a first step, we ask

\begin{question} \label{Frage:PerturbationofOnePtGeodesic20180402}
For any smooth potential $\varphi_0\in \SpaceofKahlerPotential$, does there exist a small neighborhood of $\varphi_0$ such that   any generic smooth potential in this neighborhood can be connected with $\varphi_0$ by a smooth and non-degenerate geodesic?
\end{question}

In finite dimensional Riemannian manifold, this question  corresponds to existence of normal geodesic neighborhood.
Using exponential map, one obtains that existence of  normal geodesic neighborhood is equivalent to
 the short time existence of geodesic from any given point along  any
 tangential direction at that point.
 It is therefore natural to study the initial value problem for geodesics. Unfortunately, initial value problem for geodesics in the space of K\"ahler potentials is not well-posed:
 according to Donaldson \cite{DonaldsonSymmetricSpace}, initial value problem
 might not always have a regular (at least $C^3$) solution, even for a short time.
 In fact, an explicit example is given
in \cite{DonaldsonSymmetricSpace}, and a necessary condition for the solvability
was given by Y. Rubinstein and S. Zelditch
\cite{RubinsteinZelditchAnalytisity}. Somewhat surprisingly, we prove that there
always exists a normal regular-geodesic neighborhood.

\begin{theorem}\label{thm:RegularNormalNeighborhoodExistence}
For any smooth K\"ahler potential $\psi_0\in \SpaceofKahlerPotential$, and
$0<\alpha<1$,
there exists a small  neighborhood of $\psi_0$
in $C^{4,\alpha}$  norm, such that for
any potential $\psi$ in this neighborhood, there exists a
$C^{4}$ geodesic segment connecting $\psi_0$ and $\psi.$
\end{theorem}

 Given the huge success of Donaldson's program on the space of K\"ahler potentials, it is clearly very important to continue our line of research, from Theorem 1.2, to improve the regularity of geodesic segment
 beyond  $C^{1,1}$.
 There are two natural approaches conceptually. 

The first is to view geodesic segment as the length minimizer,  and try to improve regularity of this ``minimizer"  from  $C^{1,1}$ as much as possible. Hopefully we can improve the regularity  to the extent that K\"ahler potentials along geodesic path are nearly smooth
in some appropriate geometric sense. In this approach, the crucial step is to obtain $C^2$ continuity  of geodesic segment. For this approach, the following problem will be interesting:

\begin{question}
Given a non-degenerate $C^{1,1}$ geodesic segment with smooth end points, can we show the geodesic is actually $C^2$  continuous?
\end{question}

In light of the work of J. Chu, B. Wenkove and V.Tossati \cite{Jianchun}, and L. Lempert, L. Vivas, T. Darvas \cite{LempertLizVivas}, \cite{LempertDarvas},
the new assumption is that the $C^{1,1}$ geodesic segment is non-degenerate. In fact, we believe
it is sufficient to assume that the geodesic segment is non-degenerate near its
two end points.\\

The second is to envision or conceive certain partially high order regularity statement which is slightly weaker than the original stated version
 of Donaldson's conjecture and try to prove it by
 method of continuity.  For this approach, the following problem is critically important:

\begin{question}\label{Frage:NearlySmooth201803302327} In the space of K\"ahler potentials, for two generic smooth potentials, does there exist a nearly smooth geodesic segment connecting them?  
\end{question}
While we expect the K\"ahler forms involved in geodesic segment to be smooth almost everywhere, the nature of singularities
must be part of the puzzle to be figured out together when attacking the full extent of Donaldson's conjecture.
Indeed, Theorem 1.2 or Theorem 1.8 is a first step in this approach.

\subsection{Main Result and Technique}
A path $\psi:[0,1]\rightarrow \SpaceofKahlerPotential$
 can be considered as a function on $[0,1]\times V$, and also as a function
 on $[0,1]\times \ER\times V$, with a dummy variable in $\ER$ direction.
 When $\psi$ is considered as a function on $[0,1]\times \ER\times V$ or
 $[0,1]\times V$, we denote it by $\Psi$. It was shown in \cite{DonaldsonSymmetricSpace}, that  $\psi$ satisfying geodesic equation is equivalent to that $\Psi$ satisfying a homogenous complex Monge-Ampere equation (which we abbreviate as HCMA equation)  on $[0,1]\times \ER\times V.$

Instead of studying the regularity of general geodesics, or the regularity of
solution to general HCMA equations, a technique was developed in
\cite{DonaldsonHolomorphicDiscs} to study perturbations  of
 boundary values in
Dirichlet Problem for HCMA equation on $D\times V$, where $D$ is the unit disc on complex plane. More precisely the following problem was considered:
 \begin{problem}[Dirichlet Problem of HCMA equation on $D\times V$ with boundary value $F$]
\label{prob:DirichletonD}
Given a K\"ahler manifold
$(V,\omega_0)$, with $\omega_0>0$, and  a smooth real valued function $F$
on $\partial D\times V$ satisfying
\begin{equation}\label{nondegenBdryF}
\omega_0+\sqrt{-1}F(\tau, \cdot)_{i\overline j}dz^i\wedge\overline{dz^j}>0,
\ \ \ \ \ \  \text{on } \{\tau\}\times V\;\mbox{ for all }\;\tau \in \partial D,
\end{equation}
 and denoting the trivial projection from $D\times V$ to $V$ by $\pi_V$, and
 $\Omega_0:=\pi_V^\ast(\omega_0)$, we look for
 $\Phi\in C^2(D\times V\torange \ER)\cap C^0(\overline D\times V\torange \ER)$, solving
\[\left\{
\begin{array}{cc}
(\Omega_0+\sqrt{-1}\ddbar\Phi)^{n+1}=0,&\text{in }D\times V;\\
\Phi=F		,						&\text{on } \partial D\times V;\\
\omega_0+\sqrt{-1}\Phi(\tau, \cdot)_{i\overline j}dz^i\wedge\overline{dz^j}>0, & \text{on } \{\tau\}\times V,\ \ \  \forall\tau \in D .
\end{array}
\right.\]
\end{problem}
In \cite{DonaldsonHolomorphicDiscs}, the above problem was reduced to a family of
elliptic free boundary problems, whose linearized problem is a family of
Riemann-Hilbert problems on the disc. Then using the solvability and stability of elliptic
problem, Donaldson proved that the set of smooth functions $F$, for which a smooth
solution to Problem \ref{prob:DirichletonD}
exists, is open in $C^\infty(\partial D\times V)$ with respect to $C^2$ topology.

Inspired by the above result, we try to address the following question
\begin{question}\label{Frage:ShortimpliesSmooth}
Does any point in $\SpaceofKahlerPotential$ possess a  neighborhood, in $
C^2$ norm, that is geodesically convex by $C^\infty$ non-degenerate geodesics?
\end{question}

In this paper, we provide a partial answer to the above question.

\begin{theorem}\label{thm:RegularityofNearConstantGeodesics}
Let $(V,\omega_0)$ be a compact smooth K\"ahler manifold, with  $\omega_0>0$, and let
$k> \requirement$, $0<\jop<\min\{{1\over 4},{k-\requirement\over 4}\}$. There exists
$\varepsilon=\varepsilon(V,\omega_0,k,\jop)>0$ such that if
$\varphi_0, \varphi_1\in C^k(V\torange \ER)$
 satisfy $|\varphi_0|_k+|\varphi_1|_k<\varepsilon$, then
there exists  a non-degenerate geodesic $\Psi$ connecting $\varphi_0$ and $\varphi_1$, with
\[\Psi\in C^{k-\jop}([0,1]\times V\torange \ER),\]
and
\[|\Psi|_{k-\jop;[0,1]\times V}\leq C(V,\omega_0, k,\jop)\left(|\varphi_0|_k+|\varphi_1|_k\right).\]
\end{theorem}

\begin{remark}The notation $C^k$, for non-integer $k$, is as same as that used in
\cite{HormanderPhysicalGeodesy}, and it's also explained in Section \ref{NotationandConvention} of
the current paper.
\end{remark}
\begin{remark}
In \cite{RegularityofGeodesics}, J. Hu shows that answer to
Question \ref{Frage:ShortimpliesSmooth} is negative. Indeed, with flat torus
being the background manifold, he constructs a sequence of analytic functions $\psi_k$, for $k=1,2,...$, so that
\[|\psi_k|_B\rightarrow 0, \text{ as }k\rightarrow \infty, \text{ for any fixed }B\in \EZ^+,  \]
while none of $\psi_k$ can be connected with $0$ by $C^\infty$ non-degenerate geodesic. It suggests that  in this sense our result Theorem \ref{thm:RegularityofNearConstantGeodesics} is optimal.

\end{remark}

\begin{remark} In theorem above,
 constant $\varepsilon(V,\omega_0,k,\jop)$ may go to zero and $C(V,\omega_0,k,\jop)$ may go to $\infty$, as $\jop\rightarrow 0$ or $k\rightarrow 4\ (\text{or } \infty) $.
\end{remark}

\begin{remark}
Theorem \ref{thm:RegularityofNearConstantGeodesics} implies Theorem \ref{thm:RegularNormalNeighborhoodExistence}, by simply replacing background metric $\omega_0$ by $\omega_0+\sqrt{-1}\ddbar\psi_0$, for any $\psi_0$.
\end{remark}
Now we explain the ideas behind Theorem \ref{thm:RegularityofNearConstantGeodesics}. This is proved by an iteration, in which on each step we solve Problem \ref{prob:DirichletonD}
with boundary data determined by the previous step of the iteration.
More precisely,
in order to construct a geodesic segment between two K\"{a}hler potentials
in a small neighborhood of a given potential,
we perturb the explicit solution of HCMA equation which corresponds to the
one-point geodesic
segment, in the domain $\trR\times V$, where
$\trR\subset [0,1]\times \ER$ is a long strip of the finite length, as illustrated in Figure \ref{fig:ConstructionofLongStrip}.
We use complex coordinates $\tau=t+\sqrt{-1}\theta$ on $[0,1]\times \ER$.
Endpoints of geodesic are the prescribed K\"{a}hler potentials at
$t=0$ and $t=1$ respectively.
Then
Dirichlet data in Problem \ref{prob:DirichletonD}
are now given
only on the
part  of $\partial\trR$ which lies on $t=0$ and $t=1$, and we seek a
solution which does not depend
on  the imaginary part, $\theta$, of the complex ``time" $\tau$. We construct such
solution
by iteration: at each step we
prescribe the data on the remaining part of the boundary, i.e.
on $\partial \trR\setminus(\{t=0\}\cup\{t=1\})$, and then solve
Problem \ref{prob:DirichletonD} with these Dirichlet data
using a version of the method of \cite{DonaldsonHolomorphicDiscs}
for H\"older spaces $C^{k,\alpha}$, which we develop in Section \ref{DirichletProblem}.
We use  this solution to update the boundary data on
$\partial \trR\setminus(\{t=0\}\cup\{t=1\})$
in such way
that the fixed point of this process does not depend on $\theta$-variable.
Then the fixed point
is the solution of HCMA which corresponds to the geodesic segment.
Existence of a fixed point is obtained
by Nash-Moser type theorem, since
 the estimates of solution  involve a loss of regularity,
due to the degeneracy of the HCMA equation. This comprise the bulk of our Section 3.

In Section \ref{FamilyofProblems}, we prove three lemmas. One
studies a family of Dirichlet problems for Poisson equation, and
 is used in several parts of the paper, to show the regularity of functions on
 the product space.
 Another shows the solvability of the
Riemann-Hilbert Problem in H\"older spaces, and is used in Section \ref{DirichletProblem} in the proof of stable existence of holomorphic disc families.
The third discusses a family of harmonic functions, and is used in Section \ref{Iteration}, to show the invertibility of tangential map of the iteration map.

In Section \ref{MoserTheorem}, we prove a version of Moser-type inverse function theorem.
\subsection{Notation and Convention}\label{NotationandConvention}
\begin{itemize}
\item We denote $[0,1]\times \ER$ by $\sS$, it is considered as a Riemann  surface with complex coordinate $\tau=t+\sqrt{-1}\theta$, where we used $t$ as $[0,1]$ direction variable, and $\theta$ as $\ER$ direction variable.
\item Given closed metric spaces $M, N$, space $C^r(M)$, for $r>0$, should be understood
as $C^{[r],\{r\}}(M)$, and $|\spacecdot|_r=|\spacecdot|_{[r],\{r\}}$.
When there is no ambiguity, we will not indicate domain of definition,
so $|f|_{r;M}$ may be abbreviated as $|f|_r$ sometimes. Also, given $(f_1, f_2)\in C^r(M)\times C^r(N)$, we may denote
\[|f_1, f_2|_r=|f_1|_{r;M}+|f_2|_{r;N}.\]
\item When there is no ambiguity, when doing estimate, we will not indicate the dependence of constants on $V$ and $\omega_0.$
\end{itemize}


\section{Disc Problem}\label{DirichletProblem}
In this section  we study the solvability of Problem \ref{prob:DirichletonD}, i.e. Dirichlet Problem of the HCMA equation on $D\times V$. Our aim is to prove
Proposition \ref{prop:DonaldsonDiscDirichletProblem}.

We first briefly describe the approach in \cite{DonaldsonHolomorphicDiscs}, where solvability of the HCMA equation is related to the existence of a family of holomorphic discs, with boundaries attached to a totally real submanifold defined by the Dirichlet data.

From the argument of \cite{DonaldsonHolomorphicDiscs} and \cite{Semmes}, we know if
$\Phi\in C^{3}(\overline D\times V)$ is a solution to Problem \ref{prob:DirichletonD}, with boundary
value $F$, then kernels  of $\Omega_0+\sqrt{-1}\,\ddbar\,\Phi$ form a foliation on $D\times V$,
which we denote by $\Foliation(F)$ and there is a map $\ADF$ from $\cD\times V$ to $\cD\times V$, satisfying
\begin{equation}\label{EquationofADFAug13}
\left\{
\begin{array}{l}
\pi_D\circ \ADF=\pi_D;\\
\vspace{1ex}
\ADF\big|_{\{\tau=-\sqrt{-1}\}\times V}=Id;\\
\vspace{1ex}
\text{for each }  z\in V,
\text{ restriction of }    \Omega_0+\sqrt{-1}\,\ddbar\,\Phi   \text{  to }  \ADF(D\times \{z\})
 \text{ vanishes};\\
\text{for each }  z\in V,
\text{ restriction of }    \ADF   \text{  to }  D\times \{z\}
 \text{  is holomorphic.}
 \end{array}
\right.
\end{equation}
where $\pi_D$ is the trivial projection from $\cD\times V$ to $\cD$.

To introduce the family of holomorphic discs with boundaries attached to a totally real manifold, a holomorphic fiber bundle $\WV$ was constructed in \cite{DonaldsonHolomorphicDiscs}, in the following way. Suppose $V=\cup \Uk$, where $\Uk$'s are open sets. On $\Uk$, denote the complex coordinates
 by $z_\ok^i$, and suppose $\omega_0$ is locally given by
 $\sqrt{-1}\ddbar \rho_0^\ok$. We glue up $\{T^\ast \Uk\}$ in the following way.
 Let $p\in \Uk \cap \Uvm$, then $(p,\xi) \in T^\ast\Uk$ and $(p,\eta) \in T^\ast\Uvm$ are identical, if and only if
\[(\xi_i-\partial_i\rho_0^\ok)dz_\ok^i=(\eta_i-\partial_i\rho_0^\vm)dz_\vm^i.\]
Then $\cup T^\ast\Uk$ modulo this equivalence relation is denoted by $\WV$ in \cite{DonaldsonHolomorphicDiscs}.

It was shown in \cite{DonaldsonHolomorphicDiscs}, that  in $D\times \WV$,
\[\Lambda_F=\bigcup_{\ok}\bigcup_{\overset{\tau\in\partial D}{ p\in U_\ok}}(\tau, p, [\partial_{z^i_\ok}\rho_0^\ok(p)+\partial_{z^i_\ok} F(\tau,p)])\]
is a totally real submanifold, more precisely an LS-submanifold.
Mapping $\mathcal{G}$ from $D\times V $ to $D\times\WV$, given by
\[(\tau,p)\mapsto\left(\tau,\ \ADF(\tau,p),\  (\partial_{z^i_\ok}\rho^\ok_0(\spacecdot)+\partial_{z^i_\ok}\Phi(\tau,\spacecdot))\circ\ADF(\tau,p)\right),\text{ for }(\tau,p) \in \ADFi(D\times \Uok),\]
is holomorphic with respect to $\tau$ variable, and
\[\mathcal{G}(\partial D\times V)\subset\Lambda_F.\]

The converse is also true, roughly speaking, given a family of holomorphic discs, with boundaries attached to $\Lambda_F$, we can construct solution to HCMA equations on $D\times V$.

Then, using the theory of  elliptic PDEs and free boundary problems, it was shown that existence of holomorphic disc families with boundaries attached to a totally real manifold is stable under the $C^1$ perturbation of boundary manifold.

Through this approach, Donaldson showed in \cite{DonaldsonHolomorphicDiscs}, that the set of smooth functions, $F$, for which a smooth solution to Problem \ref{prob:DirichletonD} exists is open in $C^\infty(\partial D\times V)$ with respect to $C^2$ topology.

Comparing to Theorem 1 of \cite{DonaldsonHolomorphicDiscs},  the following Proposition \ref{prop:DonaldsonDiscDirichletProblem} is weaker and less general, but it contains the estimates we need.

\begin{proposition}\label{prop:DonaldsonDiscDirichletProblem}
Given $(V,\omega_0)$,  with  $\omega_0>0$,
 for any $0<\fractionSpace<1$ and $\ir\geq \fourthird$, $\ir\notin \EZ$,
there exists $\delta_D(\ir\alsoDependonBaseFrac)>0$, such that if $F\in C^\ir(\partial D\times V\torange \ER)$
 satisfies (\ref{nondegenBdryF}) and $|F|_{\fourthird}<\delta_D(\ir\alsoDependonBaseFrac)$

, then there exists $\Phi\in C^{\ir-2}(\cD\times V\torange \ER)$ solving Problem \ref{prob:DirichletonD}
 with boundary value $F$, and
\begin{equation}
|\Phi|_{\ir-2}\leq C(\constr\alsoDependonBaseFrac)|F|_\ir.					\label{CrEstimateofPhiDisc}
\end{equation}
The corresponding map $\ADF$ satisfying (\ref{EquationofADFAug13}) is in $C^{\ir -2}(\cD\times V\torange \cD\times V)$, with
\begin{align}
&|\ADF-Id|_{\ir-2}\leq C(\constr\alsoDependonBaseFrac)|F|_\ir,	\label{CrEstimateofADFDisc}
\\
&|\ADFi-Id|_{\ir-2}\leq C(\constr\alsoDependonBaseFrac)|F|_\ir.	\label{CrEstimateofADFiDisc}
\end{align}
Moreover, for any two boundary functions $F, \hat F$, both satisfying  requirements of this Proposition,
the corresponding
$\Phi, \hat \Phi$ and $\ADF, \ADHF$ satisfy
\begin{align}
&|\Phi-\hat\Phi|_{\ir-2}\leq C(\constr\alsoDependonBaseFrac)(|F-\hat F|_{\ir-2}+(1+|F|_\ir+|\hat F|_\ir)|F-\hat F|_{2}),	\label{CrLipPotentialDisc}
\\
&|\ADF-\ADHF|_{\ir-2}\leq C(\constr\alsoDependonBaseFrac)(|F-\hat F|_{\ir-1}+(1+|F|_\ir+|\hat F|_\ir)|F-\hat F|_{\threethird}).	\label{CrLipADFDisc}
\end{align}
\end{proposition}

\begin{remark}Condition that  $|F|_{\fourthird}<\delta_D$ implies
condition  (\ref{nondegenBdryF}), providing $\delta_D$ small enough.
\end{remark}
\begin{remark}
Below, when using Proposition \ref{prop:DonaldsonDiscDirichletProblem}, we
will fix $\fractionSpace=\min\{\frac{k-4}{3}, {1\over 3}\}$, while we need to
make $\ir$ large. Then, for convenience, when there is no ambiguity we will
refer constants  in Proposition \ref{prop:DonaldsonDiscDirichletProblem}
as $\delta_D(\ir)$ and $C(\ir)$, without displaying the dependence on $\fractionSpace$. And $\delta_D(\ir)$ may go to zero and $C(\ir)$ may go to infinity as the fractional  part of $\ir$ goes to zero, so, in section \ref{ApplyingMoserTheorem} when using Proposition \ref{prop:DonaldsonDiscDirichletProblem}, we only allow fractional part of $\ir$ equal to $\frac{1}{3}$.
\end{remark}
 To prove Proposition \ref{prop:DonaldsonDiscDirichletProblem},
\begin{itemize}
\item  in Section \ref{LocalTheory}, we prove Lemma \ref{lem:LocalStableExistenceofHolomorphicDiscFamilies}, regarding the local existence of family of holomorphic discs, which is a local version of Proposition 2 of \cite{DonaldsonHolomorphicDiscs};
\item  in Section \ref{GlobalTheory}, we show these locally constructed disc families agree with each other, when and where  their domains of definition overlap;
\item in  Section \ref{ProcedureofDonaldsonSemmes}, we construct potential function $\Phi$ from family of holomorphic discs, following argument of \cite{DonaldsonHolomorphicDiscs} and \cite{Semmes};
\item in Section \ref{ImprovingRegularity}, we improve estimate regarding comparison of potential functions with an integration method
and complete the proof of Proposition \ref{prop:DonaldsonDiscDirichletProblem}.
\end{itemize}
\subsection{Stability and  Existence of Holomorphic Disc Families}\label{StableExistenceofHolomorphicDiscFamilies}
In this section, we prove the existence of holomorphic disc families. We first prove a local version, Lemma \ref{lem:LocalStableExistenceofHolomorphicDiscFamilies}, in Section \ref{LocalTheory}. Then get global theory Lemma  \ref{lem:GlobalStableExistenceofHolomorphicDiscFamilies}, \ref{lem:InverseofGlobalMonodromy} in Section \ref{GlobalTheory}.
\subsubsection{Local Theory}\label{LocalTheory}
\begin{lemma}[Local Existence of Holomorphic Disc Families]
\label{lem:LocalStableExistenceofHolomorphicDiscFamilies}
We denote $B_1=\overline{B_1(0)}$ in $\EC^n$, $B_{\half}=\overline{B_\half(0)}$ in
$\EC^n$, and let $z^i$, $i=1,..., n$ be the complex coordinates on $\EC^n$.
Given $\rho_0\in C^\infty(B_1\torange \ER)$, with $\sqrt{-1}\ddbar\rho_0>0$,
and any $0<\fractionSpace<1$, $\ \ir\geq \fourthird$, $\ir\notin \EZ$,
there exists $\delta_L=\delta_L(\rho_0,\constr\alsoDependonBaseFrac)>0$
such that, if $F\in C^{\ir}(\partial D\times B_1\torange \ER)$ satisfies
$|F|_{\fourthird}\leq \delta_L$,
there exists
\[(f,h)\in C^{\ir-2}(\cD\times B_\half\torange \EC^n\times \EC^n),\]
solving the boundary value problem
\begin{align}
\overline{\partial_\tau}f=\overline{\partial_\tau}h=0,\ \ \ \ \ \ \ \ \ \ \ \ \ & \text{ in } D\times B_\half;\label{LocalDiscEquationLemma}\\
\partial_i(\rho_0+F)(\tau,z+f(\tau,z))=\partial_i\rho_0(z)+h_i(\tau,z),\ \ \ \ \ \ \ \ \ & \text{ on } \partial D\times B_\half,\text{ for } i=1,...,n;
												\label{LocalDiscBoundaryCondition}\\
f(-\sqrt{-1},z)=0,  \ \ \ \ \ \ \ \ \ \ \ \ &\text{for  all } z\in B_\half,								\label{LocalDiscOnePointBdConditionLemma}
\end{align}
and satisfying
\begin{align}
|f|_{\ir-2}+|h|_{\ir-2}\leq C(\rho_0, \constr\alsoDependonBaseFrac) |F|_\ir.
   					\label{Cr-2BoundMonodromyDiscLocalTheoryNovember30}\\
|f|_{\twothird}+|h|_{\twothird}\leq C(\rho_0\alsoDependonBaseFrac) |F|_\threethird.			
\label{C23BoundMonodromyDiscLocalTheory}
\end{align}

Also, given two boundary perturbations $F$, $\hat F$, the corresponding solutions
$(f,h)$, $( \hat f, \hat h)$ satisfy
\begin{align}
&|f-\hat f|_{\twothird}+|h-\hat h|_{\twothird}\leq C(\rho_0\alsoDependonBaseFrac)|F-\hat F|_\threethird;\label{C23LipMonodromyDiscLocalTheory}
\\
&|f-\hat f|_{\ir-2}+|h-\hat h|_{\ir-2}\leq C(\rho_0,\constr\alsoDependonBaseFrac)(|F-\hat F|_{\ir-1}+(1+|F|_\ir+|\hat F|_\ir)|F-\hat F|_{\threethird}). \label{CrLipMonodromyDiscLocalTheory}
\end{align}

Moreover, if $|F|_{\fourthird}\leq \delta_L(4+X)$, then a solution of
problem (\ref{LocalDiscEquationLemma})--(\ref{LocalDiscOnePointBdConditionLemma})
satisfying (\ref{C23BoundMonodromyDiscLocalTheory}) is unique.
Here and bellow we write $\delta_L(\ir)$ for $\delta_L(\rho_0,\constr\alsoDependonBaseFrac)$
because $\rho_0$ and $X$ are fixed.

\end{lemma}

\begin{remark}\label{lowerIndexesRmk}
For a given $X$, we can assume without loss of generality that $\delta_L(\gamma)\le\delta_L(4+X)$
for all $\gamma\ge 4+X$. Then uniqueness above
 implies solutions obtained by the application of
Lemma  \ref{lem:LocalStableExistenceofHolomorphicDiscFamilies} does not depend on $\gamma$, so
for the given $F, \hat F\in C^\gamma$,
application of Lemma \ref{lem:LocalStableExistenceofHolomorphicDiscFamilies} with any
$\gamma_1\in [4+X, \gamma]$ will give estimates (\ref{Cr-2BoundMonodromyDiscLocalTheoryNovember30}),
(\ref{CrLipMonodromyDiscLocalTheory})
with $\gamma_1$ instead of $\gamma$.
\end{remark}

\begin{remark}
The following proof is based on a modification of the standard implicit function theorem.
 Application of the standard implicit function theorem would give the existence of the
solution $(f, h)$ in $C^{\ir-2}$ under the condition that
 $|F|_\ir$ is small. But here we assumed only that $|F|_{4+\fractionSpace}$ is small.
\end{remark}

{\flushleft {\em Proof of Lemma \ref {lem:LocalStableExistenceofHolomorphicDiscFamilies}:}}

For $F\in C^{\ir}(\partial D\times B_1\torange \ER)$, define
\begin{equation}\label{setNHlDef}
\begin{split}\NHl=
\Big\{
(f,h)\in C^{\ir-2}(\cD\times B_\half\torange \EC^n\times \EC^n)&\;
\Big|\; \overline\partial_\tau f=\overline\partial_\tau h=0,\\
&f(-\sqrt{-1},\cdot)=0, \;|f|_{\twothird}\leq l<{1\over 4}, \;|f|_{\ir-2}\leq H
\Big\},
\end{split}
\end{equation}
with $l, H$ to be determined,  $l$ depending only on $\rho_0, \gamma, \fractionSpace$,
while
$H$  depending on $\rho_o$, $\gamma$ and $|F|_\ir$, and we can assume $H>1$.
We will perform an iteration on $\NHl$.

Given $(f, h)\in \NHl$, we can find $(\lf,\lh)\in C^{\ir-2}(\cD\times B_\half\torange \EC^n\times \EC^n)$, satisfying
\begin{align}
\partial_{\taubar}&\lf=\partial_{\taubar}\lh=0, \ \ \ \ \ \ \ \ \ \ \ \ \ \ \ \ \ \ \ \ \ \ \ \
&\text{ in } D\times B_\half;
\label{DiscMonodromyInverseEquationLocalTheory-holom}\\
\partial_{ij}\rho_0(z)\lf^j(\tau,z)&+\partial_{i\overline j}\rho_0(z)\overline{\lf^j}(\tau,z)-
\lh_i(\tau,z)\nonumber\\
&=-(\partial_i(\rho_0+F))(\tau, f(\tau, z)+z)+\partial_i\rho_0(z)+h_i(\tau,z),
&\text{  on }\partial D\times B_\half;\label{DiscMonodromyInverseBoundaryEquationLocalTheory}\\
\lf(-&\sqrt{-1},\cdot)=0.
\label{DiscMonodromyInverseEquationLocalTheory-ptcond}
\end{align}
The existence of $\lf, \lh $ follows from Lemma \ref{lem:FamilyofConstantCoefficientRHProblem}.
Then we define a map $\II$ from $\NHl$ to $C^{\ir-2}(\partial D\times B_\half\torange \EC^n\times \EC^n)$ by
\[\II(f,h)=(f,h)+(\lf,\lh).\]

Our argument will consist of four steps:
\begin{description}
\item[Step 1.] We  will show that, when $l$ and 
$\delta_L$
are small enough, and  $H$ is large enough,
\[\II(\NHl)\subset \NHl.\]
\item[Step 2.] We show that if $l$ and 
$\delta_L$ are small enough, then $\II$ is a contraction map, with respect to some
    weighted norm.
\item[Step 3.] We show that, by  making  $l$ and $\delta_L$
even smaller, our solution will satisfy Lipschitz type estimate, i.e.
(\ref{C23LipMonodromyDiscLocalTheory}) and (\ref{CrLipMonodromyDiscLocalTheory}).
\item[Step 4.] We obtain the uniqueness assertion of the Lemma by  further decreasing $\delta_L$.
\end{description}

In each step, we will  precisely estimate $|\II(f,h)|_\twothird$ and
$|\II(f,h)|_{\ir-2}$ with (\ref{DiscMonodromyInverseBoundaryEquationLocalTheory}).
And, in the following, we will use index $\alpha,\beta,\varkappaforgamma$, each
running through $1,...,n, \overline 1,..., \overline n$, and we will also use
notation $f^{\overline i}=\overline{f^i}$,
 $h_{\overline i}=\overline{h_i}$.

{\flushleft{\bf Step 1.}}
On $\partial D\times B_\half$, $\II(f,h)$ satisfies the following boundary condition
\begin{align}
  &\partial_{ij}\rho_0(z)(\lf^j+f^j)(\tau,z)+
  \partial_{i\overline j}\rho_0(z)\overline{(\lf^j+f^j)}(\tau,z)-(\lh_i(\tau,z)+
  h_i(\tau,z))\nonumber\\
=&-\int_0^1\left[\partial_{i\alpha}\rho_0(z+uf(\tau,z))-
\partial_{i\alpha}\rho_0(z)\right]\,du\, f^\alpha(\tau,z)
-(\partial_i F)(\tau, f(\tau, z)+z)\nonumber \\
=&-\int_{0}^1\int_0^1\partial_{i\alpha\beta}\rho_0(z+
uvf(\tau,z))u\,dv\,du\, f^{\alpha}(\tau,z)f^\beta(\tau,z)-(\partial_i F)(\tau, f(\tau, z)+z).
											\label{BoundaryConditionofIIfhwithIntegration}
\end{align}
By Theorem A.7, A.8 of \cite{HormanderPhysicalGeodesy}, right hand side of (\ref{BoundaryConditionofIIfhwithIntegration}) can be estimated as
\begin{align*}
|\text{RHS of } (\ref{BoundaryConditionofIIfhwithIntegration})|_{\twothird}
\leq C(\rho_0)&(
|\rho_0|_{\fivethird}(1+|f|_1)^{\twothird}|f|_0^2+
|\rho_0|_4(1+|f|_\twothird)|f|_0^2				
	\\
&+|\rho_0|_3|f|_{\twothird}|f|_0
+|F|_\threethird(1+|f|_1)^{\twothird}+|F|_2(1+|f|_\twothird))	
	\\
\leq C(\rho_0)&(l^2+ |F|_\threethird),
\end{align*}
\begin{align*}
|\text{RHS of } (\ref{BoundaryConditionofIIfhwithIntegration})|_{\irtwothird} \leq C(\rho_0,\constr)&(
|\rho_0|_{\irfivethird}(1+|f|_1)^{\irtwothird}|f|_0^2+|\rho_0|_4(1+|f|_\irtwothird)|f|_0^2		 \nonumber		
	\\
&+|\rho_0|_3|f|_{\irtwothird}|f|_0
+|F|_\irthreethird(1+|f|_1)^{\irtwothird}+|F|_2(1+|f|_\irtwothird))					 \nonumber
	\\
\leq C(\rho_0,\constr)&(l^2+ |F|_\irthreethird+(1+H)(l+|F|_2)),
\end{align*}
where we used the bounds of $(f,h)$ which follow from the inclusion $(f,h)\in \NHl$.
Then by Lemma \ref{lem:FamilyofConstantCoefficientRHProblem}, we have
\begin{equation*}
|\II(f,h)|_\twothird\leq C_1(\rho_0\alsoDependonBaseFrac)(l^2+|F|_{\threethird});
\end{equation*}
\begin{equation*}
|\II(f,h)|_\irtwothird\leq C_2(\rho_0,\constr)(|F|_{\irthreethird}+2H(l+|F|_2)).
\end{equation*}
Constants $C_1, C_2, ... , C_{10} $ here and in the following are all assumed
to be greater than $1.$ If we choose  $\delta_L$, $l$, $H$ satisfying
\begin{align}l<{1\over 4C_1(\rho_0\alsoDependonBaseFrac)},\ \ \
\delta_L<{l\over 4C_1(\rho_0\alsoDependonBaseFrac)},\ \ \
\delta_L+l\leq {1\over 8C_2(\rho_0,\constr)},
\label{Step1_l_ConditionCollection_Dec10}
\end{align}
\begin{align}
H> 4C_2(\rho_0,\constr)(|F|_{\ir-1}+1),
\label{Step1F_HighOrderNorm_and_H_Dec10}
\end{align}
we have
\[|\II(f,h)|_\twothird\leq l, \text{ and } |\II(f,h)|_\irtwothird\leq H,\]
i.e. $\II$ maps $\NHl$ into $\NHl$.
{\flushleft{\bf Step 2.}} Now, we assume $l,\ H,\ F$ satisfy conditions in
(\ref{Step1_l_ConditionCollection_Dec10}), (\ref{Step1F_HighOrderNorm_and_H_Dec10}), and
 show, when $l$ and 
 $\delta_L$
 are made even smaller, $\II$ is a
 contraction map with respect to some weighted norm. The norm is, for some $A$ large enough,
\begin{equation}\|\spacecdot\|=|\spacecdot|_{\ir-2}+A|\spacecdot|_\twothird.\label{WeightedNormDefinitionLocalTheoryofDiscFamily}
\end{equation}

 Then given $(f,h)$, $(\hat f,\hat h)\in \NHl$, we need to estimate $|\II(f,h)-\II(\hat f,\hat h)|_{\twothird}$ and $|\II(f,h)-\II(\hat f,\hat h)|_{\irtwothird}$.

Plug $\hat f,\hat h$ into (\ref{BoundaryConditionofIIfhwithIntegration}) then subtract  original (\ref{BoundaryConditionofIIfhwithIntegration}) we get, on  $\partial D\times B_\half$,
\begin{align*}
  &\partial_{ij}\rho_0(z)(\lf^j+f^j-\hlf^j-\hat f^j)
  +\partial_{i\overline j}\rho_0(z)\overline{(\lf^j+f^j-\hlf^j-\hat f^j)}-(\lh_i+h_i-\hlh_i-\hat h_i)\nonumber\\
&\quad
=\int_{0}^1\int_0^1\int_0^1\partial_{i\alpha\beta\varkappaforgamma}\rho_0(z+wuv\hat f
+(1-w)uvf)u^2v\,dv\,du\,dw
  	 f^{\alpha}f^\beta(\hat f-f)^\varkappaforgamma\nonumber\\
  &\quad\quad-\int_{0}^1\int_0^1\partial_{i\alpha\beta}\rho_0(z+uvf(\tau,z))u\,dv\,du \,(\hat f-f)^\alpha (\hat f+f)^\beta\nonumber\\
  &\quad\quad+\int_0^1\partial_{i\alpha}F(\tau, z+u\hat f+(1-u)f)du\,(\hat f-f)^\alpha.
\end{align*}
Again, with Theorem A.7, A.8 of \cite{HormanderPhysicalGeodesy} and Lemma \ref{lem:FamilyofConstantCoefficientRHProblem}, we have
\begin{align*}
&|\II(f,h)-\II(\hat f,\hat h)|_\twothird\leq C_3(\rho_0\alsoDependonBaseFrac)
(l+|F|_\fourthird)|\hat f-f|_\twothird,
\\
&
|\II(f,h)-\II(\hat f,\hat h)|_\irtwothird\leq C_4(\rho_0,\constr)
\left((l+|F|_2)|\hat f-f|_{\irtwothird}+(H+|F|_\irfourthird)|\hat f-f|_\twothird\right).
\end{align*}
Then, choosing 
$\delta_L$ and $l$
 small so that
\begin{align}
&
\delta_L\leq{1\over 4C_3 (\rho_0\alsoDependonBaseFrac)+4C_4(\rho_0,\constr )},
 \label{F_LowOrderNormUpperBoundDec10}
 \\
&l\leq {1\over 4C_3(\rho_0\alsoDependonBaseFrac)+4 C_4(\rho_0,\constr)},	\label{UpperBoundof_l_Dec10}
\end{align}
and using $|F|_{\fourthird}\leq
\delta_L$, we have
\begin{align}
&|\II(f,h)-\II(\hat f,\hat h)|_{\twothird}\leq {1\over 2}|(f,h)-(\hat f,\hat h)|_{\twothird},
\label{Contractionwrt23NormLocalHolomorphicDiscFamily}
\\
&|\II(f,h)-\II(\hat f,\hat h)|_{\irtwothird}\leq {1\over 2}|(f,h)-(\hat f,\hat h)|_{\irtwothird}
+C_4(\rho_0,\constr)(H+|F|_\ir)|(f,h)-(\hat f,\hat h)|_\twothird.
\label{ContractionwrtrNormLocalHolomorphicDiscFamily}
\end{align}

Plugging (\ref{ContractionwrtrNormLocalHolomorphicDiscFamily}) and (\ref{Contractionwrt23NormLocalHolomorphicDiscFamily}) into
(\ref{WeightedNormDefinitionLocalTheoryofDiscFamily}), gives
\begin{align}
	&\|\II(f,h)-\II(\hat f,\hat h)\|\nonumber\\
&
\qquad\qquad
\leq	
{1\over 2}|(f,h)-(\hat f,\hat h)|_{\ir-2}+\left({1\over 2}A+C_4(\rho_0,\constr)
(H+|F|_\ir)\right)|(f,h)-(\hat f,\hat h)|_{\twothird}.
\label{DiscContractionTrytoDeterA}
\end{align}
We found that if we choose
\begin{equation}\label{defAforContract}
A=6C_4(\rho_0,\constr)(H+|F|_\ir),
\end{equation}
then
\[
\text{RHS of }(\ref{DiscContractionTrytoDeterA})\leq\frac{2}{3}|(f,h)-(\hat f,\hat h)|_{\ir-2}+\frac{2}{3}A|(f,h)-(\hat f,\hat h)|_{\twothird}
= \frac{2}{3}
\|(f,h)-(\hat f,\hat h)\|.\]
This makes $\II$ a contraction map w.r.t the weighted norm, more precisely
 \begin{align}\label{contractImplM}
\|\II(f,h)-\II(\hat f,\hat h)\|\leq {2\over 3} \|(f,h)-(\hat f,\hat h)\|.
\end{align}

We note that the above estimates are justified if
$l$, $\delta_L$ and $H$ satisfy (\ref{Step1_l_ConditionCollection_Dec10})
(\ref{Step1F_HighOrderNorm_and_H_Dec10})
(\ref{F_LowOrderNormUpperBoundDec10})
(\ref{UpperBoundof_l_Dec10}), which can be achieved by choosing
\begin{align*}
&\delta_L=\frac{1}{32^2(1+C_1(\rho_0\alsoDependonBaseFrac)+C_2(\rho_0,\constr)
+C_3(\rho_0\alsoDependonBaseFrac)+C_4(\rho_0,\constr))^2},\\
& 	l		  = \frac{1}{32(1+C_1(\rho_0\alsoDependonBaseFrac)+C_2(\rho_0,\constr)
+C_3(\rho_0\alsoDependonBaseFrac)+C_4(\rho_0,\constr))},\\
&H=16C_2(\rho,\constr)(|F|_{\ir-1}+1) .
\end{align*}
From (\ref{Contractionwrt23NormLocalHolomorphicDiscFamily}) and (\ref{contractImplM}),
 $\II$ is contraction map in $\|\spacecdot\|$ and $|\spacecdot|_\twothird$ norms.
Then the sequence $\{\II^i(0,0)\}_{i=1}^{\infty}$ converges with
respect to $|\spacecdot|_{\ir-2}$ and $|\spacecdot|_\twothird$ norms, and it is easy to see
the limit is a solution to 	(\ref{LocalDiscEquationLemma}), (\ref{LocalDiscBoundaryCondition}),
(\ref{LocalDiscOnePointBdConditionLemma}). For convenience, we still denote solution by $(f,h)$.
Using (\ref{Contractionwrt23NormLocalHolomorphicDiscFamily})(\ref{contractImplM}), we  estimate solution $(f,h)$ as:
\begin{equation}|(f,h)|_{\twothird}\leq 2|\II(0,0)-(0,0)|_{\twothird}\leq C_5(\rho_0\alsoDependonBaseFrac)|F|_\threethird,  \label{DiscC23NormLemmaLocal}
\end{equation}
\begin{align}|(f,h)|_{\irtwothird}\leq &3\|\II(0,0)-(0,0)\|
\leq C(\rho_0,\ir)(|F|_\irthreethird+(H+|F|_\ir)|F|_\threethird)\leq
C(\rho_0,\constr\alsoDependonBaseFrac)|F|_\ir,  \label{preDiscC23NormLemmaLocal}
\end{align}
where the second inequality follows from Lemma \ref{lem:FamilyofConstantCoefficientRHProblem}
applied to problem (\ref{DiscMonodromyInverseEquationLocalTheory-holom})--(\ref{DiscMonodromyInverseEquationLocalTheory-ptcond}),
and we used (\ref{defAforContract}) in the definition of the norm $\|\cdot\|$.
Estimate (\ref{DiscC23NormLemmaLocal})(\ref{preDiscC23NormLemmaLocal}) confirms (\ref{C23BoundMonodromyDiscLocalTheory}) (\ref{Cr-2BoundMonodromyDiscLocalTheoryNovember30}).

\begin{remark}
Actually without using weighted norm, by more carefully manipulating $\ir-2$ and $\fractionSpace+2$ norms, we can also show for any ${\boldsymbol{p}}\in \NHl$,
the sequence $\{\II^k({\boldsymbol{p}})\}$ is a Cauchy sequence with respect to
$|\spacecdot|_{\ir-2}$.  The use of weighted norm is only for conciseness of presentation.
\end{remark}

{\flushleft{\bf Step 3.}}
Now, given $F, \hat F\in C^\ir(\partial D\times B_1)$, with
$|F|_{\fourthird}, |\hat F|_{\fourthird}<\delta_L$, we can
find $(f,h), (\hat f,\hat h)\in
C^{\ir-2}(\cD\times B_\half\torange \EC^n\times \EC^n)$, which solve
(\ref{LocalDiscEquationLemma}),
(\ref{LocalDiscBoundaryCondition}), (\ref{LocalDiscOnePointBdConditionLemma}),
corresponding to $ F$ and $\hat F$ respectively. Then
 $(f-\hat f,h-\hat h)$ satisfies
 on
$\partial D\times B_\half$:
\begin{align}
	\partial_{i\overline j}
&\rho_0(z)(\hat f-f)^{\overline j}+\partial_{i j}\rho_0(z)(\hat f-f)^j-(\hat h-h)_i\nonumber\\
&  = -\int_{0}^1\int_0^1\partial_{i\alpha\beta}\rho_0(z+vu\hat f+v(1-u)f)
(u \hat f+(1-u)f)^\beta dv\,du\, (\hat f-f)^\alpha\nonumber\\
  & \quad\;-[\partial_i (\hat F-F)](\tau,z+\hat f(\tau,z))-
  \int_0^1[\partial_{i\alpha} F](\tau,z+u\hat f+(1-u)f)\,du\, (\hat f-f)^\alpha.
\label{LipConditionIntegralFormDiscLemma}
\end{align}
Then, using again Theorems A.7,  A.8 of \cite{HormanderPhysicalGeodesy} and Lemma \ref{lem:FamilyofConstantCoefficientRHProblem}, we have
\[
|f-\hat f|_\twothird+|h-\hat h|_\twothird
\leq C_6(\rho_0\alsoDependonBaseFrac)(|f|_\twothird+|\hat f|_\twothird+|F|_\fourthird)|f-\hat f|_\twothird+C(\rho_0\alsoDependonBaseFrac)|F-\hat F|_\threethird.\]
Also, by (\ref{DiscC23NormLemmaLocal}), we have
\[|f|_\twothird+|\hat f|_\twothird\leq C_7(\rho_0\alsoDependonBaseFrac)
(|F|_{\fourthird}+|\hat F|_{\fourthird}),\]
so, reducing $\delta_L$ so that
\[|F|_{\fourthird}+|\hat F|_{\fourthird}\leq 2\delta_L
\leq \frac{1}{2(1+C_6(\rho_0\alsoDependonBaseFrac))(1+C_7(\rho_0\alsoDependonBaseFrac))},\]
we have
\begin{equation}
|f-\hat f|_\twothird+|h-\hat h|_\twothird\leq C_8(\rho_0\alsoDependonBaseFrac)|F-\hat F|_{\threethird},
\end{equation}
which shows (\ref{C23LipMonodromyDiscLocalTheory}).
Now, applying  Theorem A.8 of \cite{HormanderPhysicalGeodesy} and Lemma \ref{lem:FamilyofConstantCoefficientRHProblem} to (\ref{LipConditionIntegralFormDiscLemma}) again, to estimate the $C_{\ir-2}$ norm, we get
\begin{align}|f-\hat f|_{\ir-2}+|h-\hat h|_{\ir-2}\leq &C_9(\rho_0,\constr\alsoDependonBaseFrac)(|F|_2+|f|_\twothird+|\hat f|_\twothird)|f-\hat f|_{\ir-2}\nonumber\\
&+C(\rho_0,\constr\alsoDependonBaseFrac)(|F-\hat F|_{\ir-1}+(1+|F|_\ir+|\hat F|_\ir)|F-\hat F|_{\threethird}).
\end{align}
Then, further reducing $\delta_L$ so that
\[|F|_{\fourthird}+|\hat F|_\fourthird \leq 2\delta_L\leq
\frac{1}{8(1+C_9(\rho_0,\constr\alsoDependonBaseFrac))
(1+C_7(\rho_0\alsoDependonBaseFrac))},\]
we have
\begin{align*}
|f-\hat f|_{\ir-2}+|h-\hat h|_{\ir-2}\leq &C_{10}(\rho_0,\constr\alsoDependonBaseFrac)\left(|F-\hat F|_{\ir-1}+(1+|F|_\ir+|\hat F|_\ir)|F-\hat F|_{\threethird}\right).
\end{align*}
This confirms (\ref{CrLipMonodromyDiscLocalTheory}).

{\flushleft{\bf Step 4.}} Now, after further reducing $\delta_L(4+X)$,
 we prove uniqueness asserted in this Lemma. That is, we show that
  if
$|F|_{\fourthird}\leq \delta_L(4+X)$, then
solution of
problem (\ref{LocalDiscEquationLemma})--(\ref{LocalDiscOnePointBdConditionLemma})
satisfying (\ref{C23BoundMonodromyDiscLocalTheory}) is unique.

In the argument bellow, we fix $\ir=4+X$,
and we use the constant $l$ determined in Steps 1--3 for $\ir=4+X$.
Note that this $l$ depends only on $\rho_0$ and $X$.
 Noting that $\ir-2=2+X$ and $H>1$, we obtain that the iteration set
(\ref{setNHlDef}) for $\ir=4+X$ becomes:
\begin{equation}\label{setNHlDef-4pX}
\begin{split}\NHl=
\Big\{
(f,h)\in C^{\ir-2}(\cD\times B_\half\torange \EC^n\times \EC^n)\;
\Big|\; &\overline\partial_\tau f=\overline\partial_\tau h=0,\\
&f(-\sqrt{-1},\cdot)=0, \;|f|_{\twothird}\leq l<{1\over 4}
\Big\}.
\end{split}
\end{equation}

If $\delta_L(4+X)$ is small as determined in Steps 1--3 for  $\gamma=4+X$,
and  $|F|_{\fourthird}\leq \delta_L(4+X)$, then
$\II(\NHl)\subset \NHl$. This, combined with
(\ref{DiscMonodromyInverseEquationLocalTheory-holom})--(\ref{DiscMonodromyInverseEquationLocalTheory-ptcond})
implies that for such $F$, every solution $(f,g)\in \NHl$ of
problem (\ref{LocalDiscEquationLemma})--(\ref{LocalDiscOnePointBdConditionLemma})
is a fixed point of the iteration map $\II(\cdot)$. Since the map $\II(\cdot)$ is a
contraction as we showed in Step 2, we obtain the uniqueness of a solution
$(f,g)\in \NHl$ of
problem (\ref{LocalDiscEquationLemma})--(\ref{LocalDiscOnePointBdConditionLemma}).

Now we reduce $\delta_L(4+X)$ so that $\delta_L(4+X) \le l/C(\rho_0, X)$,
where $C(\rho_0, X)$ is from
(\ref{C23BoundMonodromyDiscLocalTheory})
and $l=l(\rho_0, X)$ is fixed above. Now, using (\ref{C23BoundMonodromyDiscLocalTheory}) (\ref{setNHlDef-4pX}),
we obtain that if $|F|_{\fourthird}\leq \delta_L(4+X)$,
then any solution $(f,g)$ satisfying
(\ref{C23BoundMonodromyDiscLocalTheory}) is in the set $\NHl$. This proves
the uniqueness asserted in Lemma.

Lemma \ref{lem:LocalStableExistenceofHolomorphicDiscFamilies} is now proved.

\subsubsection{Global Theory}\label{GlobalTheory}
In the argument bellow we use  the holomorphic fibre bundle  $\WV$ and the
submanifold $\Lambda_F\subset D\times \WV$ defined earlier, after (\ref{EquationofADFAug13}).

 We now prove the following global version of Lemma \ref{lem:LocalStableExistenceofHolomorphicDiscFamilies}.

\begin{lemma}[Global Existence and Stability of  Families of Holomorphic Disc]
\label{lem:GlobalStableExistenceofHolomorphicDiscFamilies}
Given K\"ahler manifold $(V,\omega_0)$, $\omega_0>0$ and any $0<\fractionSpace<1$,
$\ir\geq \fourthird$, $\ir\notin \EZ$, there exists
$\delta_G=\delta_G(V, \omega_0,\constr\alsoDependonBaseFrac)>0$
such for each $F\in C^{\ir}(\partial D\times V\torange \ER)$ with $|F|_{\fourthird}\leq \delta_G$,
there exists
\[\GF\in C^{\ir-2}(\cD\times V\torange \cD\times \WV),\]
satisfying
\begin{align}
\partial_{\overline\tau} \GF=0,\text{ in }D\times V;
\label{GlobDiscEquationLemma}\\
\GF(\partial D\times V)\subset \Lambda_F;
\label{GlobDiscBoundaryCondition}\\
\Pi_{\cD\times V}\left(\GF\big|_{\{\tau=-\sqrt{-1}\}\times V}\right)=Id;
\label{GlobDiscOnePointBdConditionLemma}\\
\Pi_D\circ\GF=\pi_D,
\label{GlobDiscProjConditionLemma}
\end{align}
 where $\tau$ denotes variable in $D$,  $\Pi_{\cD\times V}$
 is the projection from $\cD\times \WV$ to $\cD\times V$ determined by the fiber bundle projection
$\mP_V:\WV\to V$, and $\Pi_{D}: \cD\times \WV\to \cD$ and $\pi_D: \cD\times V\to \cD$ are the trivial
projections.

Moreover, let $\ADF := \Pi_{\cD\times V}\circ \GF: \,\cD\times V \to \cD\times V$. It is
 shown in \cite{DonaldsonHolomorphicDiscs} that the map $\ADF$ satisfies
 (\ref{EquationofADFAug13}).
 We have the following estimate for $\ADF$:
\begin{align}
&|\ADF-Id|_{\ir-2;\cD\times V}\leq C(\fractionSpace,\constr\alsoDependonBaseFracornot) |F|_{\ir},				
\label{Cr-2BoundMonodromyDiscGlobalTheory}\\
&|\ADF-Id|_{\twothird;\cD\times V}\leq C(\fractionSpace)\onlyDependonBaseFracornot |F|_{3+X}.			
\label{C23BoundMonodromyDiscGlobalTheory}
\end{align}

Also, for two boundary perturbations $F,\hat F$,  the corresponding maps $\GF, \dsG_{\hat F}$ satisfy
\begin{align}
&|\GF-\dsG_{\hat F}|_{\twothird;\cD\times V}\leq C\onlyDependonBaseFracornot|F-\hat F|_\threethird;\label{C23LipMonodromyDiscGlobalTheory}
\\
&|\GF-\dsG_{\hat F}|_{\ir-2;\cD\times V}\leq
C(\constr\alsoDependonBaseFracornot)\left(|F-\hat F|_{\ir-1}+
(1+|F|_\ir+|\hat F|_\ir)|F-\hat F|_{\threethird}\right).
\label{CrLipMonodromyDiscGlobalTheory}
\end{align}

Moreover, if $|F|_{\fourthird}\leq \delta_G(4+X)$, then a solution of
problem (\ref{GlobDiscEquationLemma})--(\ref{GlobDiscProjConditionLemma})
satisfying (\ref{C23BoundMonodromyDiscGlobalTheory}) is unique.
Here and bellow we write $\delta_G(\ir)$ for $\delta_G(\omega_0,\constr\alsoDependonBaseFrac)$
because $\omega_0$ and $X$ are fixed.
\end{lemma}
\begin{proof}
Since  $V$ is a  compact K\"ahler manifold, $V$
can be covered by a finite number of open sets $\{\Uok\}_{\ok=1}^M$ such that for each $\ok$ there is
an open set $\Ook\supset \overline{\Uok}$ and   a biholomorphic map
$\chi_\ok:{\Ook}\rightarrow \EC^n$ satisfying:
\begin{enumerate}
\item $\chi_{\ok}\in C^\infty( \Ook);$
\item $\chi_{\ok}(\Ook)=B_1\subset \EC^n$;
\item $\chi_{\ok}(U_\ok)=B_\half\subset \EC^n$.
\end{enumerate}
Then for any $0<\fractionSpace<1$ and $\ir\geq \fourthird$, $\ir\notin \EZ$, there exists a
$\tilde{\delta}_G(\ir\alsoDependonBaseFrac, V,\omega_0)>0$ such that if
$F\in C^{\ir}(\partial D\times V\torange \ER)$ with $|F|_{\fourthird}\leq \tilde{\delta}_G$,
then
\begin{equation}\label{boundsInLocalizations}
|F\circ \chi_\ok^{-1}|_{\fourthird}\leq \delta_L(\ir,\rho_0^\ok\alsoDependonBaseFrac),
\end{equation}
 for any $\ok=1, ... , M$,
where $\chi_\ok$ is considered as a map from $D\times \Ook$ to $D\times B_1$ by
$\chi_\ok(\tau, x):=(\tau, \chi_\ok(x))$, and $\omega_0$ is
locally given by $\rho_0^\ok$, and $\delta_L(\ir,\rho_0^\ok\alsoDependonBaseFrac)$ is from
 Lemma \ref{lem:LocalStableExistenceofHolomorphicDiscFamilies}.
By Lemma \ref{lem:LocalStableExistenceofHolomorphicDiscFamilies} there exists
$$\dsLG_{\ok}\in C^{\ir-2}(\cD\times U_\ok\torange \cD\times \mP_V^{-1}(\Ook)) ,$$
 where $\mP_V$
 is the projection from $\WV$ to $V$, with $\dsLG_\ok$ satisfying
\begin{equation}
\left\{
\begin{array}{c}
\partial_{\overline\tau} \dsLG_\ok=0,\text{\ \ \ \ \ in $D\times U_\ok$};\\
\dsLG_{\ok}({\partial D\times U_\ok})\subset \Lambda_F;\\
\Pi_{D\times V}\circ \dsLG_\ok\big|_{\{-\sqrt{-1}\}\times U_\ok}\equiv Id;\\
\Pi_D\circ \dsLG_{\ok}=\pi_D,
\end{array}
\right.
\end{equation}
where
$\Pi_{D\times V}$, $\Pi_D$, $\pi_D$ are projections defined in the formulation of the Lemma.

Now, assuming that $U_\ok\cap U_{\oi}\neq \O$, we want to show
\[\dsLG_{\ok}\big|_{D\times(U_\ok\cap U_\oi)}=\dsLG_\oi\big|_{D\times(U_\ok\cap U_\oi)}.\]
So, let $p\in {U_\ok\cap U_\oi}$, then $\dsLG_{\ok}\big|_{D\times \{p\}}$ and $\dsLG_{\oi}\big|_{D\times \{p\}}$ are both holomorphic discs with boundaries attached to $\Lambda_F$. And if $\tilde\delta_G(\ir)$ is small enough, without loss of generality, we can assume
\[\dsLG_\oi(D\times \{p\})\subset D\times\mP_V^{-1}(\Ook),\]
and so, we can denote,
\[(\dsLG_{\ok}\big|_{D\times \{p\}})(\tau)=(\tau,p+f(\tau),k(\tau)),\]
\[(\dsLG_{\oi}\big|_{D\times \{p\}})(\tau)=(\tau,p+\hat f(\tau),\hat k(\tau)),\]
where $f,\hat f, k,\hat k$ are complex vector valued holomorphic functions on $D$, and they satisfy, with coordinates  $\{z^i\}_{i=1}^n$ on $\Ook$:
\begin{align}
\left[\partial_i(\rho_0+F)\right](\tau,p+f(\tau))=k_i(\tau), \text{ on }\partial D, \text{ for }i=1,\ ...\ M;\label{GlobalTheoryBoundaryConditionfh}\\
\left[\partial_i(\rho_0+F)\right](\tau,p+\hat f(\tau))=\hat k_i(\tau), \text{ on }\partial D, \text{ for }i=1,\ ...\ M;\label{GlobalTheoryBoundaryConditionhatfh}\\
f(-\sqrt{-1})=\hat f(-\sqrt{-1})=0.\ \ \ \ \ \ \ \ \ \ \ \ \ \ \ \ \ \
\end{align}
Taking difference of (\ref{GlobalTheoryBoundaryConditionfh}) (\ref{GlobalTheoryBoundaryConditionhatfh}) gives
\begin{align}
	&[\partial_{ij}\rho_0^\ok(p)](f-\hat f)^j+[\partial_{i\overline j}\rho_0^\ok(p)](f-\hat f)^{\overline j}-(k-\hat k)_i\nonumber\\
=&	-\int_0^1\int_0^1[\partial_{i\alpha\beta}\rho_0^\ok](p+vuf+v(1-u)\hat f)(uf+(1-u)\hat f)^\beta \,du\,dv\;(f-\hat f)^\alpha\nonumber\\
&\ \ \ \ -\int_0^1[\partial_{i\alpha}F](\tau,p+uf+(1-u)\hat f)\,du\, (f-\hat f)^\alpha.     \label{DifferenceofLocalDefinedDiscsGlobalTheory}
\end{align}

Applying Theorem A.7, A.8 of \cite{HormanderPhysicalGeodesy}, Lemma \ref{lem:FamilyofConstantCoefficientRHProblem} and (\ref{C23BoundMonodromyDiscLocalTheory}) of Lemma \ref{lem:LocalStableExistenceofHolomorphicDiscFamilies} to (\ref{DifferenceofLocalDefinedDiscsGlobalTheory}) gives
\[|f-\hat f|_{\twothird}\leq C(\omega_0, V)|f-\hat f|_{\twothird}(|f|_\twothird+|\hat f|_\twothird+|F|_{\fourthird})\leq C(\omega_0)|f-\hat f|_\twothird|F|_\fourthird.\]

When
\[|F|_{\fourthird}\leq \min \left\{\tilde\delta_G(\fourthird\alsoDependonBaseFrac,V,\omega_0),{1\over 2C(\omega_0, V\alsoDependonBaseFrac)}\right\},\]
we have $f=\hat f$. So those locally constructed holomorphic discs over different coordinate patches
match, so  (\ref{Cr-2BoundMonodromyDiscGlobalTheory}),
(\ref{C23BoundMonodromyDiscGlobalTheory})
follow from (\ref{Cr-2BoundMonodromyDiscLocalTheoryNovember30}),
(\ref{C23BoundMonodromyDiscLocalTheory}). Similarly,
(\ref{C23LipMonodromyDiscGlobalTheory}), (\ref{CrLipMonodromyDiscGlobalTheory})
 follow from (\ref{C23LipMonodromyDiscLocalTheory}), (\ref{CrLipMonodromyDiscLocalTheory}) directly.
 This proves Lemma \ref{lem:GlobalStableExistenceofHolomorphicDiscFamilies}.
 Also, uniqueness statement follows from the bounds (\ref{boundsInLocalizations})
 in localizations and the uniqueness in Lemma \ref{lem:LocalStableExistenceofHolomorphicDiscFamilies}.
\end{proof}

Next we discuss the  invertibility of the map $\ADF$.

\begin{lemma}\label{lem:InverseofGlobalMonodromy}
In addition to the statement of Lemma \ref{lem:GlobalStableExistenceofHolomorphicDiscFamilies},
 there exists a $\delta_I(\ir\alsoDependonBaseFracornot,V,\omega_0)$, with $0<\delta_I\leq \delta_G$
 such that if $F\in C^\ir(\partial D\times V)$, satisfies $|F|_{\fourthird}\leq \delta_I$, then $\ADF$ is invertible and satisfies
\begin{equation}|\ADFi-Id|_{\ir-2;\cD\times V}\leq
C(\fractionSpace,\ir\alsoDependonBaseFracornot)|F|_{\ir}.\label{DifferenceofADFiandId}
\end{equation}

Moreover,  if $\ir \geq \fivethird$, then for any $F, \hat F$ with $|F|_{\fourthird},|\hat F|_{\fourthird}$ both small enough,  the corresponding $\ADF$, $\ADHF$ satisfy
\begin{equation}|\ADFi-\ADHFi|_{\ir-3;\cD\times V}\leq C(\ir\alsoDependonBaseFracornot,
|F|_\ir,|\hat F|_\ir )|F-\hat F|_{\ir}.\label{DifferenceofADFiandADHFi}
\end{equation}

\end{lemma}

\begin{proof}

For the invertibility of $\ADF$,  since
\[|\ADF-Id|_\twothird\leq C|F|_\fourthird,\]
when $|F|_{\fourthird}$ small enough, we have $\ADF$ is one-to-one
 with non-degenerate Jacobian, so $\ADFi$ is well defined, and has same
 differentiability as $\ADF$.

 Now we estimate $\ADFi-Id$.
Locally,
we can write
\[\ADFi-Id=(-\ADF+Id)\circ \ADFi.\]
Applying Theorem A.8 of \cite{HormanderPhysicalGeodesy}  to above expression, gives
\[|\ADFi-Id|_1\leq C|\ADF-Id|_1(1+|\ADFi-Id|_1).\]
 When $|F|_\fourthird$, and so $|\ADF-Id|_1$ is small enough, we have
\begin{equation}\label{DifferenceofADFiandId-prelim}
|\ADFi-Id|_1\leq C|\ADF-Id|_1.
\end{equation}

Then using Theorem A.8 of \cite{HormanderPhysicalGeodesy} again with $a=
\ir-2$, we get
\begin{align*}
|\ADFi-Id|_\irtwothird&\leq C(\ir)\left(|\ADF-Id|_1(1+|\ADFi-Id|_\irtwothird)+|\ADF-Id|_\irtwothird(1+|\ADFi-Id|_1)^{\irtwothird}\right)\\
						&\leq C(\ir) |\ADF-Id|_\irtwothird+C(\ir)|\ADF-Id|_1|\ADFi-Id|_\irtwothird.
\end{align*}
So, when $|F|_\fourthird$, and so $|\ADF-Id|_1$ small enough, we have, using
also (\ref{DifferenceofADFiandId-prelim}):
\[|\ADFi-Id|_\irtwothird\leq C(\ir)|\ADF-Id|_\irtwothird\leq C(\ir) |F|_\irfourthird.\]
This proves (\ref{DifferenceofADFiandId}).

Now, given $F$ and $\hat F$, with $|F|_\fourthird$, $|\hat F|_\fourthird$ both small enough, we want to compare $\ADFi$ with $\ADHFi$.

Locally, we can write $\ADFi-\ADHFi$ as
\begin{align*}
	&\ADFi-\ADHFi\\
=	&\ADFi(Id)-\ADFi(\ADF\ADHFi)\\
=	&\int_0^1\frac{d}{du}\left[\ADFi\left(uId+(1-u)\ADF\ADHFi\right)\right]du\\
=	&\int_0^1\frac{\partial\ADFi}{\partial z^\alpha}\left(uId+(1-u)\ADF\ADHFi\right)du\left(Id-\ADF\ADHFi\right)^\alpha\\
=	&\int_0^1\frac{\partial\ADFi}{\partial z^\alpha}\left(uId+(1-u)\ADF\ADHFi\right)du\left((\ADHF-\ADF)\ADHFi\right)^\alpha,
\end{align*}
where the index $\alpha$ runs over  $1, ..., n$ and $\overline 1,..., \overline n.$

Then applying A.7, A.8 of \cite{HormanderPhysicalGeodesy} and
Lemma \ref{lem:GlobalStableExistenceofHolomorphicDiscFamilies}, provided $\ir\geq \fivethird$,  and $|F|_\fourthird, |\hat F|_\fourthird$ both small enough, we get
\[|\ADFi-\ADHFi|_{\ir-3}\leq C(\ir, |F|_\ir,|\hat F|_\ir )|F-\hat F|_{\ir}.\]
\end{proof}

\begin{remark}\label{rem:VaryingIndexJan82018}
Using uniqueness in Lemma \ref{lem:GlobalStableExistenceofHolomorphicDiscFamilies}, we obtain
a global version of Remark \ref{lowerIndexesRmk}. Namely, assuming without loss
of generality that  $\delta_G(\gamma)\le\delta_G(4+X)$
for all $\gamma\ge 4+X$, we have that uniqueness in
Lemma  \ref{lem:GlobalStableExistenceofHolomorphicDiscFamilies} implies solutions obtained by the application of
Lemma \ref{lem:GlobalStableExistenceofHolomorphicDiscFamilies} does not depend on $\gamma$, so
for the given $F, \hat F\in C^\gamma$,
application of Lemma \ref{lem:GlobalStableExistenceofHolomorphicDiscFamilies} with any
$\gamma_1\in [4+X(\text{or }5+X), \gamma]$ will give estimates
(\ref{Cr-2BoundMonodromyDiscGlobalTheory})--(\ref{CrLipMonodromyDiscGlobalTheory}) and
(\ref{DifferenceofADFiandId})--(\ref{DifferenceofADFiandADHFi})
with $\gamma_1$ instead of $\gamma$.
\end{remark}

\subsection{Potential Function}
Now, with Lemma \ref{lem:GlobalStableExistenceofHolomorphicDiscFamilies}, Lemma \ref{lem:InverseofGlobalMonodromy} and the analysis of \cite{DonaldsonHolomorphicDiscs} and \cite{Semmes}, we can construct potential function solving Problem \ref{prop:DonaldsonDiscDirichletProblem}.
 In Section \ref{ProcedureofDonaldsonSemmes}, we will first follow method in \cite{DonaldsonHolomorphicDiscs} and \cite{Semmes} to construct the potential and perform
its estimates in H\"older spaces. However, the  comparison of potential functions leads to a loss of three orders of  regularity. Then, in Section \ref{ImprovingRegularity}, we improve the comparison result
 (\ref{DonaldsonComparisonAug15}) of Section \ref{ProcedureofDonaldsonSemmes}, using the integration of leafwise harmonic function, to obtain (\ref{CrLipPotentialDisc}),
and complete the proof of Proposition \ref{prop:DonaldsonDiscDirichletProblem}.

\subsubsection{Construction of Donaldson and Semmes}
\label{ProcedureofDonaldsonSemmes}
With Lemma \ref{lem:GlobalStableExistenceofHolomorphicDiscFamilies},
for the given $0<\fractionSpace<1,$ $F\in C^{\ir}(\partial D\times V)$, with $\ir\geq \fourthird$, and
\[|F|_\fourthird\leq \delta_I(\ir),\]
where $\delta_I(\cdot)$ is from Lemma \ref{lem:InverseofGlobalMonodromy},  we can find the corresponding
\[\GF\in C^{\ir-2}(\cD\times V\torange \cD\times \WV),\] 
and locally $\GF$ is given by
\[(\tau,z)\rightarrow(\tau, z+f(\tau,z),\partial\rho_0(z)+h(\tau,z)).\]
where $f$ and $h$ are holomorphic with respect to
 $\tau$. Moreover, from (\ref{GlobDiscBoundaryCondition}),
 (\ref{GlobDiscProjConditionLemma}) and Lemma  \ref{lem:InverseofGlobalMonodromy}
we obtain that for each $\tau\in\partial D$, the map $\GF(\tau, \cdot)$ maps
$V$ diffeomorphically to the LS-graph $\Lambda_F\cap(\{\tau\}\times \WV)$.
Then according to  Lemma 3 and Proposition 1 of \cite{DonaldsonHolomorphicDiscs},
 for any $\tau\in \overline D$,  there exists $p^\tau$, a real valued function on $V$, satisfying
\begin{equation}\partial\rho_0+h|_{\{\tau\}\times V}=[\partial(\rho_0+ p^\tau)](z+f),    \label{hisdifferentialLocalExpression_2017_1221}
\end{equation} locally.
Since adding a constant  to $p^\tau$ does not change (\ref{hisdifferentialLocalExpression_2017_1221}), we can choose a point $\zeroptofp$ on $V$, independent of $\tau$, and let
\begin{equation}\label{OnePointConditionofP_Dec_21_DS}
p^\tau(\zeroptofp)=0, \ \ \ \ \forall \tau\in \overline D.
\end{equation}
Denote $$P(\tau,z)=p^\tau(z),$$
we will then construct solution to Problem \ref{prob:DirichletonD}, based on $P$.

Since $F\in C^{\ir}(\partial D\times V)$,  $\GF\in C^{\ir-2}(\cD\times V)$,
using (\ref{hisdifferentialLocalExpression_2017_1221}) (\ref{OnePointConditionofP_Dec_21_DS}) (\ref{Cr-2BoundMonodromyDiscLocalTheoryNovember30})
(\ref{C23BoundMonodromyDiscLocalTheory}), we get
\[|P|_{\ir-2}\leq C(\ir)|\partial p^\tau|_{\ir-2}\leq C({\ir})(|h(\AD_F^{-1})|_{\ir-2}+|\ADFi-Id|_{\ir-2})\leq C(\ir)|F|_{\ir},\]
\[|P|_{2}\leq C(\ir)|F|_{\fourthird}.\]

Next, we use the fact that
$\GF$ is holomorphic with respect to $\tau$, i.e. 
for a fixed $z$,
\begin{equation}
(\tau,z)\rightarrow (\tau, z+f(\tau,z), [\partial_V (\rho_0+P)](\tau,z+f)),
\end{equation}
is holomorphic with respect to $\tau$.
Taking $\partial_{\overline \tau}$ derivative of the third component of the  expression above gives
\begin{equation}\label{dtaubarderivativeofhComponent_2017_Dec_21}
[\partial_{i\overline\tau}(\rho_0+P)](\ADF(\tau,z))+[\partial_{i\overline j}(\rho_0+P)](\ADF(\tau,z))
\overline{\left(
\frac{\partial f^j(\tau, z)}{\partial \tau}
\right)}=0.
\end{equation}

Then using $\partial_{i\overline j}(\rho_0+P)=g_{i\overline j}$ in (\ref{dtaubarderivativeofhComponent_2017_Dec_21}), we have
\begin{equation}\label{derivativeofffrom_hholomorphic_2017_Dec_21}
\frac{\partial f^j}{\partial \tau}=-[\partial_{\overline i \tau}(\rho_0+P)g^{\overline i j}](\ADF(\tau,z)),
\end{equation}
where $[g_{i\overline j}]$ is invertible because $[\rho_{0,i\overline j}]$ is invertible, and $P_{i\overline j}$ is close to zero.
Then taking $\partial_{\overline \tau}$ derivative of (\ref{derivativeofffrom_hholomorphic_2017_Dec_21}),  and using again
\[\frac{\partial f^j}{\partial \overline \tau}=0, \]
gives
\begin{align}
0=
 g_{j\overline k}&(\ADF(\tau,z))
\left[\frac{\partial^2 f^j}{\partial \tau\partial \overline\tau}\right]
=\left\{
-\partial_{\overline k\tau\overline\tau}(\rho_0+P)
+
\partial_{\overline i\tau}(\rho_0+P)g^{\overline i l}\partial_{l\overline k\overline\tau}(\rho_0+P)
\right\}\circ(\ADF(\tau,z))
\nonumber\\
&+\left\{-\partial_{\overline k\tau\overline l}(\rho_0+P)
+\partial_{\overline i\tau}(\rho_0+P)g^{\overline i \mu}\partial_{\mu\overline k\overline l}(\rho_0+P)\right\}\circ(\ADF(\tau,z))\cdot\overline{\left(\frac{\partial f^l}{\partial \tau}\right)}
.
\end{align}

Then substituting  (\ref{dtaubarderivativeofhComponent_2017_Dec_21}) into
the above equation
we get
\begin{align}
0=\left\{\left[-\partial_{\tau\overline \tau}(\rho_0+P)+\partial_{\tau\overline i}(\rho_0+P)g^{\overline i j}\partial_{\overline\tau j}(\rho_0+P)\right]_{\overline k}\right\}(\ADF(\tau,z)).
\label{kbarDerivativeofLaplaceQ}
\end{align}
So we know
\[
\partial_{\tau\overline \tau}(\rho_0+P)-\partial_{\tau\overline i}(\rho_0+P)g^{\overline i j}\partial_{\overline\tau j}(\rho_0+P)
\]
depends only on $\tau$. And since $\rho_0$ does not depend on $\tau$, then
\[P_{\tau \overline\tau}-P_{\tau \overline i} g^{\overline i j}P_{j\overline \tau}
\]
is a globally defined function on $D\times V$, depending only on $\tau$. We have
\[P_{\tau \overline\tau}-P_{\tau \overline i} g^{\overline i j}P_{j\overline \tau}\in C^{\ir-4}(\cD),\]
since $P\in C^{\ir-2}(\cD\times V)$. Solving Dirichlet problem for Laplacian equation in $\tau$
variables in $\cD$, we can find a $Q\in C^{\ir-2}(\cD)$ satisfying
\begin{align}
\label{equationofQ_2017_Dec_21}
&Q_{\tau\overline\tau}=-P_{\tau \overline\tau}+P_{\tau \overline i} g^{\overline i j}P_{j\overline \tau} \ \ \ \ \text{ in }D,\\
&Q\big|_{\partial D}=F\big|_{\partial D\times \{\zeroptofp\}};
\label{BdryCondofQ_2017_Dec_21}
\end{align}
where $\zeroptofp\in V$ is the point fixed above, see (\ref{OnePointConditionofP_Dec_21_DS}).
If we consider $Q$ as a function defined on $\cD\times V$, then (\ref{equationofQ_2017_Dec_21}) is satisfied in $D\times V$ and, since $Q$ only depends on $\tau$, we can transform (\ref{equationofQ_2017_Dec_21}) into
\[(P+Q)_{\tau\overline\tau}=(P+Q)_{\tau \overline i} g^{\overline i j}_{}(P+Q)_{j\overline \tau}.\]
Then $\Phi=P+Q$ is a solution to Problem \ref{prob:DirichletonD} with boundary value $F$,
as explained in the first two pages of \cite{Semmes}. And for the boundary condition,
we use
that $\GF(\partial D\times V)$ lies in $\LSMfdF$, so
$P-F$  on $\partial D\times V$ depends only on $\tau$,
and then from (\ref{OnePointConditionofP_Dec_21_DS}) and (\ref{BdryCondofQ_2017_Dec_21}) we have
$\Phi=F$ on $\partial D\times V$.

Since  $Q$ does not depend on $z$, estimates for the Dirichlet problem (\ref{equationofQ_2017_Dec_21})--(\ref{BdryCondofQ_2017_Dec_21}) imply the following estimates
on $\cD\times V$
\begin{align*}
&|Q|_{\ir-2}\leq C(\ir)(|Q_{\tau\overline\tau}|_{\ir-4}+|F|_{\ir-2})\leq  C(\ir,|P|_2)(|P|_{\ir-2}+|F|_{\ir-2})\leq  C(\ir)|F|_{\ir},
\end{align*}
and from this we get
\begin{align*}
&|\Phi|_{\ir-2}\leq |P|_{\ir-2}+|Q|_{\ir-2}\leq C(\ir)|F|_{\ir}.
\end{align*}

Then we will do  a comparison of two solutions. Suppose we have  $ F,\ \hat F\in C^{\ir+1}(\partial D\times V)$, with
\[|F|_\fourthird, \ |\hat F|_\fourthird\leq \min\{\delta_I(\ir+1), \delta_I(\fivethird)\},\]
we will have corresponding
 $\GF, p^\tau,  P, Q$, $\Phi$ and
$\hGF,\  \hat p^\tau, \hat P, \hat Q$, $\hat \Phi$.  We then compare these quantities.

Taking difference of (\ref{hisdifferentialLocalExpression_2017_1221}) and the corresponding
equation for $\hat h, \hat P$, gives
\[\partial p^{\tau}-\partial \hat p^\tau=h(\AD_F^{-1})-\hat h({\ADHF}^{-1})+\partial \rho_0(\ADFi)-\partial\rho_0(\ADHFi)\]
 then using Theorem A.7, A.8 of  \cite{HormanderPhysicalGeodesy} and (\ref{DifferenceofADFiandADHFi}) of Lemma \ref{lem:InverseofGlobalMonodromy} we get
\begin{align*}
|P-\hat P|_{\ir-2}&\leq |(h-\hat h)(\ADF^{-1})|_{\ir-2}+\left|[\int_{0}^1 \frac{\partial\hat h}{\partial z^\alpha}(u \ADF^{-1}+(1-u)\ADHF^{-1})du](\ADF^{-1}-\ADHF^{-1})^\alpha\right|_{\ir-2}\\
                   &\ \ \ \ \ +\left|[\int_{0}^1 \partial^2 \rho_0(u \ADF^{-1}+(1-u)\ADHF^{-1})du](\ADF^{-1}-\ADHF^{-1})\right|_{\ir-2}\\
				&\leq C(\ir, |F|_{\ir+1}, |\hat F|_{\ir+1})|F-\hat F|_{\ir+1}.
\end{align*}
Subsequently, we have
\[|\Phi-\hat \Phi|_{\ir-2}\leq C(\ir, |F|_{\ir+1},|\hat F|_{\ir+1})|F-\hat F|_{\ir+1}.\]

So, we proved the following: 
\begin{lemma}\label{lem:DonaldsonSemmesConstructionofPotentialFunction}
Given $0<\fractionSpace<1,$ $F\in C^{\ir}(\partial D\times V\torange \ER)$, $\ir\geq \fourthird$, $\ir \notin \EZ$, with
\[|F|_\fourthird\leq \delta_I(\ir),\]
there exists a solution $\Phi\in C^{\ir-2}(\cD\times V\torange \ER)$ to Problem \ref{prob:DirichletonD} with boundary value $F$, satisfying
\begin{equation}|\Phi|_{\ir-2}\leq  C(\ir)|F|_{\ir}.						\label{DonaldsonControlofPotentialFunctionAug15}
\end{equation}

Moreover, given two boundary values $F, \hat F$ satisfying
\[|F|_\fourthird, \ |\hat F|_\fourthird\leq \min\{\delta_I(\ir+1), \delta_I(\fivethird)\},\]
 the corresponding solution $\Phi$ and $\hat \Phi$ satisfy
\begin{equation}
|\Phi-\hat \Phi|_{\ir-2}\leq C(\ir, |F|_{\ir+1}, |\hat F|_{\ir+1})|F-\hat F|_{\ir+1}.				\label{DonaldsonComparisonAug15}
\end{equation}
\end{lemma}
We note that the estimate (\ref{DonaldsonControlofPotentialFunctionAug15}) confirms
(\ref{CrEstimateofPhiDisc}). Moreover,
since $\Phi=P+Q$, where $Q$ depends only on $\tau$, we can replace $P$ by $\Phi$ in
(\ref{dtaubarderivativeofhComponent_2017_Dec_21}). Then it follows that
\[\partial_{\tau}+\frac{\partial f^i}{\partial \tau}\partial_{z^i}\]
lies in kernels of $\Omega_0+\sqrt{-1}\partial\overline\partial \Phi$.
That is, $\ADF$ satisfies (\ref{EquationofADFAug13}) with $\Phi$ constructed above.
And, estimates (\ref{CrEstimateofADFDisc}) (\ref{CrEstimateofADFiDisc}) (\ref{CrLipADFDisc})
for $\ADF$ follow from (\ref{Cr-2BoundMonodromyDiscGlobalTheory}), (\ref{DifferenceofADFiandId}), (\ref{CrLipMonodromyDiscGlobalTheory}), respectively.

\subsubsection{Improving Comparison Result}				\label{ImprovingRegularity}
Now suppose we are given $F_0, F_1\in C^{\ir+3}$, $\ir\geq \fourthird$, $\ir \notin \EZ$, satisfying
\[| F_0|_\fourthird,|F_1|_\fourthird\leq \min\{\delta_I(\ir+3), \delta_I(\ir), \delta_I(\fivethird)\}.\]
 Denote
\[F_\lambda=\lambda F_1+(1-\lambda)F_0,\]
then  $F_\lambda$'s satisfy
\[|F_\lambda|_\fourthird\leq\min\{\delta_G(\ir+3), \delta_G(\ir), \delta_G(\fivethird)\}.\]
By Lemma \ref{lem:DonaldsonSemmesConstructionofPotentialFunction}, for any $\lambda\in [0,1]$
there exists $\Phi_\lambda\in C^\ir(\cD\times V)$ which solves Problem \ref{prob:DirichletonD} with boundary value $F_\lambda$.

We want to show that this family of solutions $\Phi_\lambda$ can be linearly approximated at any $\lambda\in [0,1]$, and the linearization is integrable as a family of $C^{\ir-2}$ functions.
We will construct a family of functions $H_{\lambda}$ which can be considered as the differential of $\Phi_\lambda$ with respect to $\lambda$, show the higher regularity of $H_\lambda$,
use it to obtain the higher order estimate of $\Phi_1-\Phi_0$.

Our argument will consist of the following 3 steps:
\begin{description}
\item[Step 1.] For any $\lambda\in[0,1]$, we construct an $H_\lambda\in C^{\ir-2}(\cD\times V\torange \ER)$ satisfying
\begin{equation}|H_\lambda|_{\ir-2}\leq C(\ir)\left(|F_0-F_1|_{\ir-2}+(1+|F_0|_\ir+|F_1|_\ir)|F_0-F_1|_2\right). \label{EstimateofLinearizationofFamilyofSolution}
\end{equation}
\item[Step 2.] We show that, as a function of $\lambda$,  $H_\lambda$ is  continuous with respect to $C^{\ir-2}$ norm.
\item[Step 3.] For any $\lambda,\nu \in [0,1],$\[|\Phi_{\lambda}-\Phi_{\nu}-(\lambda-\nu)H_{\nu}|_0\leq C(|F_0|_\fivethird, |F_1|_\fivethird)(\lambda-\nu)^2.\]
\end{description}

With the result of Steps 2 and 3, we can express $\Phi_1-\Phi_0$ as
\[\Phi_1-\Phi_0=\int_0^1H_\lambda d\lambda,\]
and with  Steps 1 and 2, we obtain that $\Phi_1-\Phi_0\in C^{\ir-2}(\cD\times V)$ and
\[|\Phi_1-\Phi_0|_{\ir-2}\leq \int_0^1|H_\lambda|_{\ir-2}d\lambda\leq C(\ir)\left(|F_0-F_1|_{\ir-2}+(1+|F_0|_\ir+|F_1|_\ir)|F_0-F_1|_2\right) .\]

{\flushleft{\bf Step 1.}}

We define the following operator $\prodh:C^{\ir-2}(\partial D\times V)\rightarrow
C^{\ir-2}(\cD\times V)$: given $u\in C^{\ir-2}(\partial D\times V\torange \ER)$, $\prodh(u)$
is the unique solution of
\begin{equation}\label{DirichletWithParam}
\left\{
\begin{array}{cc}
\triangle_\tau \prodh(u)=0 , & \text{ in }D\times V;\\
\prodh(u)=u,&   \text{ on }\partial D\times V.
\end{array}
\right.
\end{equation}
Then  $\prodh(u)\in C^{\ir-2}(\cD\times V\torange \ER)$ and
\begin{equation}\label{EstDirichletWithParam}
\begin{split}
&|\prodh(u)|_{\ir-2}\leq C(\ir) |u|_{\ir-2} ,\\
&|\prodh(u)|_1\leq C|u|_2,
\end{split}
\end{equation}
where we used Lemma \ref{lem:FamilyofPossionEquationwithBernsteinandExtension20180315}.

We define $H_\lambda$ as the leafwise harmonic function on $\Foliation(F_\lambda)$, with boundary value $F_1-F_0$.
More precisely, $H_\lambda$ is the unique solution of:
\begin{align}
\ddbar H_\lambda\wedge(\Omega_0+\sqrt{-1}\ddbar\Phi_\lambda)^n=0, &\text{ in \  } D\times V;\label{LeafwizeHarmonicH_2017_Dec_23}\\
H_\lambda=F_1-F_0, 			\ \ \ \ \ \ \ \ \ 						       &\text{ on } \partial D\times V.\label{BoundaryCondition201805062039}
\end{align}
Then $H_\lambda$ satisfies:
\[H_\lambda=(\ADFlambda)_\ast\prodh\left[(\ADFlambda)^\ast(F_1-F_0)\right]
\equiv \bigg(\prodh[(F_1-F_0)\circ \ADFlambda]\bigg) \circ \ADFlambda^{-1}.\]
Applying Theorem A.8 of \cite{HormanderPhysicalGeodesy} to the expression above, we  have the estimate
\begin{align}
&|H_\lambda|_{\ir-2}\nonumber\\
\leq &C(\ir)\left(|\prodh[(\ADFlambda)^\ast(F_1-F_0)]|_1(1+|\ADFlambda^{-1}|_{\ir-2})\right.\nonumber\\
								  	&\ \ \ \ \ \ \ \left.+|\prodh[(\ADFlambda)^\ast(F_1-F_0)]|_{\ir-2}(1+|\ADFlambda^{-1}|_{1})^{\ir-2}\right)\nonumber\\
\leq &C(\ir)\left(|(\ADFlambda)^\ast(F_1-F_0)|_2(1+|F_\lambda|_{\ir})
								  	+|(\ADFlambda)^\ast(F_1-F_0)|_{\ir-2}(1+|F_\lambda|_{\fourthird})^{\ir-2}\right)\nonumber\\
\leq &C(\ir)\left(|F_0-F_1|_2(1+|\ADFlambda|_2)^2(1+|F_\lambda|_{\ir})+|F_0-F_1|_1(1+|\ADFlambda|_{\ir-2})(1+|F_\lambda|_{\fourthird})^{\ir-2}\right.\nonumber\\
         &\ \ \ \ \ \ \ \ \ \left.+|F_0-F_1|_{\ir-2}(1+|\ADFlambda|_1)^{\ir-2}(1+|F_\lambda|_{\fourthird})^{\ir-2}\right)\nonumber\\
\leq &C(\ir)(|F_0-F_1|_{\ir-2}+(1+|F_\lambda|_{\ir})|F_0-F_1|_2).\label{BetterEstimateofLinearizationofFamilyofSolution}
\end{align}

{\flushleft{\bf Step 2.}}
Now we compare $H_\lambda$ and $H_\nu$ for $\lambda,\nu\in [0,1]$. Taking difference of  (\ref{LeafwizeHarmonicH_2017_Dec_23})(\ref{BoundaryCondition201805062039}) and the corresponding expressions for $H_\nu$, we find that
 $H_\lambda-H_\nu$ satisfies:
\[
\left
\{
\begin{array}{cc}
\ddbar (H_\lambda-H_\nu)\wedge(\Omega_0+\sqrt{-1}\ddbar\Phi_\nu)^n\ \ \ \ \ \ \ \ \ \ \ \ \ \ \ \ \ \ \ \ \ \ \ \ \ \ \ \ \ \ \ \ \ \ \ \ \ \ \ \ \ \ \ \ \\
  \ \ \ \ \ \    +\ddbar H_\lambda\wedge\sum_{k=1}^n\left(\begin{matrix}n\\k\end{matrix}\right) [\sqrt{-1}\ddbar(\Phi_\lambda-\Phi_\nu)]^k(\Omega_0+\sqrt{-1}\ddbar\Phi_\nu)^{n-k}=0 ,&\text{ in }D\times V;\\
H_\lambda-H_\nu=0 ,														&\text{ on }\partial D\times V.
\end{array}
\right.
\]
Denote
\begin{equation}\label{LeafwizePossionRHS_2017_1223_1034}
B=\sum_{k=1}^n\left(\begin{matrix}n\\k\end{matrix}\right)
\frac
{\ddbar H_\lambda\wedge[\sqrt{-1}\ddbar(\Phi_\lambda-\Phi_\nu)]^k(\Omega_0+\sqrt{-1}\ddbar\Phi_\nu)^{n-k}}
{d\tau\wedge d\overline \tau\wedge(\Omega_0+\sqrt{-1}\ddbar\Phi_\nu)^n}.
\end{equation}
To estimate $H_\lambda-H_\nu$, we define operator $\possion$ as following.
Given $v\in C^{\ir-2}(\cD\times V)$, let  $\possion(v)(\cdot, z)$ be the unique solution of
\[
\left
\{
\begin{array}{cc}
\triangle_{\tau} \possion(v)=v,								&\text{ in }D\times V;\\
\possion(v)=0 ,														&\text{ on }\partial D\times V.
\end{array}
\right.
\]
Then $\possion(v)\in C^{\ir-2}(\cD\times V)$ and satisfies
\[|\possion(v)|_{\ir-2}\leq C(\ir)|v|_{\ir-2},\ \ \ \ |\possion(v)|_1\leq C|v|_2,\]
where we used Lemma \ref{lem:FamilyofPossionEquationwithBernsteinandExtension20180315}.

With operator $\possion$ we can express $H_\lambda-H_\nu$ as
\[H_\lambda-H_\nu=(\ADFnu)_\ast\possion[\left(\ADFnu\right)^\ast B].\]

Applying Theorem A.8 of \cite{HormanderPhysicalGeodesy} to the above expression we can get
an estimate similar to the one we have for  $H_\lambda$, i.e. (\ref{BetterEstimateofLinearizationofFamilyofSolution}). But  we do not need estimate that precise here, and the following simpler estimate is sufficient:
\begin{align}
	|H_\lambda-H_\nu|_{\ir-2}
\leq C(\ir, |F_0|_{\irseventhird}, |F_1|_\irseventhird)|\possion(\ADFlambda^\ast B)|_{\ir-2}
\leq C(\ir,|F_0|_\irseventhird, |F_1|_\irseventhird)| B|_{\ir-2}.
\end{align}
For $B$, applying Theorem A.7 of \cite{HormanderPhysicalGeodesy} gives:
\[|B|_{\ir-2}\leq C(\ir,|\Phi_\lambda|_\ir, |\Phi_\nu|_\ir, |H_\lambda|_\ir)|\Phi_\lambda-\Phi_\nu|_\ir. \]
From this inequality, using (\ref{DonaldsonControlofPotentialFunctionAug15}), (\ref{DonaldsonComparisonAug15}) and
\[|H_\lambda|_{\ir}\leq C(\ir)(|F_0|_{\irseventhird}, |F_1|_{\irseventhird}),\]
which follows from replacing $\ir$ by $\ir+2$ in (\ref{BetterEstimateofLinearizationofFamilyofSolution}),
we get
\[|B|_{\ir-2}\leq C(\ir,|F_0|_\irseventhird, |F_1|_\irseventhird)|F_\lambda-F_\nu|_\irseventhird\]
So we have
\[|H_\lambda-H_\nu|_{\ir-2}\leq C(\ir,|F_0|_\irseventhird, |F_1|_\irseventhird) |\lambda-\nu|.\]

{\flushleft{\bf Step 3.}}
Denote
\[Z:=\Phi_\lambda-\Phi_\nu-(\lambda-\nu)H_\nu.\]
Then $Z $ satisfies

\[
\left
\{
\begin{array}{cc}
\ddbar Z\wedge(\Omega_0+\sqrt{-1}\ddbar\Phi_\nu)^n\ \ \ \ \ \ \ \ \ \ \ \ \ \ \ \ \ \ \ \ \ \ \ \ \ \ \ \ \ \ \ \ \ \ \ \ \ \ \ \ \ \ \ \ \\
=\frac{\sqrt{-1}}{n+1}\sum_{k=2}^{n+1}\left(\begin{matrix}n+1\\k\end{matrix}\right) [\sqrt{-1}\ddbar(\Phi_\lambda-\Phi_\nu)]^k(\Omega_0+\sqrt{-1}\ddbar\Phi_\nu)^{n+1-k} ,&\text{ in }D\times V;\\
Z=0 ,														&\text{ on }\partial D\times V.
\end{array}
\right.
\]
Then from the  maximum estimate for Laplace equation on the leaves
 $\ADF(D\times \{z\})$, $z\in V$, we get
\[|Z|_0\leq C(|\Phi_\nu|_2)|\Phi_\lambda-\Phi_\nu|_2^2\leq C(|F_0|_\fivethird, |F_1|_\fivethird)|F_\lambda-F_\nu|_\fivethird^2
\leq C(|F_0|_\fivethird, |F_1|_\fivethird)|\lambda-\nu|^2,\]
which concludes the proof of Step 3.

Thus we proved the following:
\begin{lemma}\label{ContinuityLemma}
Given $0<\fractionSpace<1,$ $F_0, F_1\in C^{\irseventhird}$, $\ir\geq \fourthird$, $\ir \notin \EZ$, satisfying
\begin{equation}| F_0|_\fourthird,|F_1|_\fourthird\leq \min\{\delta_G(\irseventhird), \delta_G(\ir), \delta_G(\fivethird)\},\label{FinalFourthirdNormSmallnessRequirement}
\end{equation}
solutions $\Phi_0$, $\Phi_1$  to Problem \ref{prob:DirichletonD} with boundary values $F_0,$ $ F_1$ satisfy
\[|\Phi_0-\Phi_1|_{\ir-2}\leq C(\ir)(|F_0-F_1|_{\ir-2}+(1+|F_0|_\ir+|F_1|_\ir)|F_0-F_1|_2).\]
\end{lemma}

 Given $F_0, F_1\in  C^{\ir}(\partial D\times V)$,  satisfying
\[
| F_0|_\fourthird,|F_1|_\fourthird\leq \frac{1}{2}\min\{\delta_G(\irseventhird), \delta_G(\ir), \delta_G(\fivethird)\},
\]
we can approximate them by $\{F_0^N\}_{N\in \EZ}$ and $\{F_1^N\}_{N\in \EZ}$, such that $F_0^N, F_1^N\in C^{\irseventhird}(\partial D\times V)$, and
\[|F_0^N-F_0|_\ir\rightarrow0,\ \ \ |F_1^N-F_1|_\ir\rightarrow0.\]
Without loss of generality, we can assume
\[
| F^N_0|_\fourthird,|F^N_1|_\fourthird\leq \min\{\delta_G(\irseventhird), \delta_G(\ir), \delta_G(\fivethird)\},
\]
\[|F_0^N|_\ir\leq 2|F_0|_\ir,\ \ |F_1^N|_\ir\leq 2|F_1|_\ir.\]
Then we can apply Lemma \ref{ContinuityLemma} to $F_0^N$ and $F_1^N$. Let $\Phi_0^N$ and $\Phi_1^N$ be solutions to Problem \ref{prob:DirichletonD} with boundary values $F_0^N$ and $F_1^N$. We have for $i=0,1$:
\[|\Phi_i^N-\Phi_i^M|_{\ir-2}\leq C(\ir)(|F_i^N-F_i^M|_{\ir-2}+(1+|F_i|_\ir)|F_i^N-F_i^M|_2).\]
Since $\{F_0^N\}$ and $\{F_1^N\}$ are both Cauchy sequences in $C^\ir(\partial D\times V)$, we have $\{\Phi_0^N\}$ and $\{\Phi_1^N\}$ are Cauchy sequence in $C^{\ir-2}(\cD\times V)$.  Let  $\Phi_0, \Phi_1\in C^{\ir-2}(\cD\times V)$ be the limits of $\{\Phi_0^N\}$ and $\{\Phi_1^N\}$
respectively, i.e.
\[\Phi_0^N\rightarrow \Phi_0 ,\ \ \ \Phi_1^N\rightarrow \Phi_1,\]
as $N\rightarrow \infty$, with respect to $C^{\ir-2}$ norm on $\cD\times V$.
Then $\Phi_0$ and $\Phi_1$ are solutions to Problem \ref{prob:DirichletonD} with boundary values $F_0, F_1,$ respectively.
We have, using Lemma \ref{ContinuityLemma}:
\begin{align*}
		&|\Phi_0-\Phi_1|_{\ir-2}\\
\leq	&|\Phi_0-\Phi_0^N|_{\ir-2}+|\Phi_0^N-\Phi_1^N|_{\ir-2}+|\Phi_1^N-\Phi_1|_{\ir-2}\\
\leq 	& o(1)+ C(\ir)(|F_0^N-F_1^N|_{\ir-2}+(1+|F_0|_\ir+|F_1|_\ir)|F_0^N-F_1^N|_2)\\
\leq 	&o(1)+C(\ir)(|F_0-F_1|_{\ir-2}+(1+|F_0|_\ir+|F_1|_\ir)|F_0-F_1|_2).
\end{align*}
Estimate (\ref{CrLipPotentialDisc}) is proved.

Now Proposition \ref{prop:DonaldsonDiscDirichletProblem} is proved with $\delta_D(\ir)=\frac{1}{2}\min\{\delta_G(\irseventhird), \delta_G(\ir), \delta_G(\fivethird)\}$.

\section{Iteration}\label{Iteration}
In this section we prove our main theorem: Theorem \ref{thm:RegularityofNearConstantGeodesics}. As stated
 in Section \ref{Introduction},  we will need to find a solution to HCMA equation on $\sS\times V$,
 which does not depend on $\theta$. To do this, we repeatedly solve Dirichlet problem on $\trR\times V$,
 with Proposition \ref{prop:DonaldsonDiscDirichletProblem}, where $\trR$ is a long but finite strip in
 complex plane. Our iteration sequence will converge to a solution to HCMA equation, which does not
 depend on $\theta$, and so it is a geodesic in the space of K\"ahler potentials.

In Section \ref{IterationFramework}, we introduce the iteration framework.

In Section \ref{ApplyingMoserTheorem}, we provide the precise estimates for our construction and show
that it satisfies conditions required by Lemma \ref{lem:MoserTheorem}.

In Section \ref{Conclusion}, we reach conclusion.
\subsection{Iteration Framework}\label{IterationFramework}
As shown in Figure \ref{fig:ConstructionofLongStrip},
 we first construct a smooth curve
$\cC$ on complex plane $\{\tau=t+\sqrt{-1}\,\theta\}$, which is the
boundary of an open convex domain $\rR$ satisfying
$$
\{0< t < 1\}\cap \{|\theta|< 1\}\subset\rR\subset\{0< t < 1\}\cap \{|\theta|<{7\over 4}\}.\qquad
$$
 Then we stretch $\cC$ to $\tcC$, so  that
 $\tcC\cap\{\Theta+1\geq\theta\geq\Theta\}$ is the leftward translation
 of $\cC\cap\{2\geq \theta\geq 1\}$ by $\Theta-1$, and
 $\tcC\cap\{-\Theta\geq\theta\geq-\Theta-1\}$ is the rightward translation
 of $\cC\cap\{-1\geq \theta\geq -2\}$ by $\Theta-1$, and $\tcC$ coincides with the
 lines $\{t=0,1\}$ on $\{|\theta|\leq \Theta\}$, where $\Theta>4$ is a large constant to
 be determined.

\begin{figure}
\centering
\includegraphics[height=4.6cm]{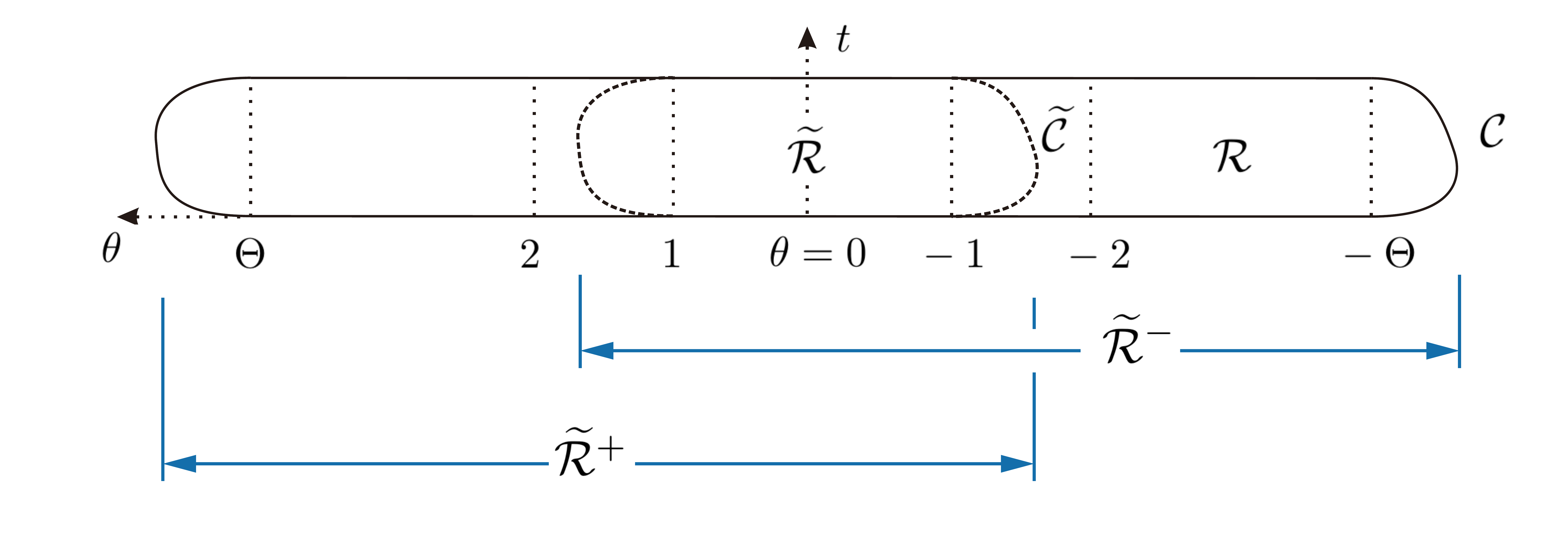}
\caption{Construction of Strips}
\label{fig:ConstructionofLongStrip}
\end{figure}

We denote by  $\trR$ the region surrounded by the curve $\tcC$, and
\[\Rectangle:=\{2\geq \theta\geq -2\}\cap \sS,\]
where $\sS$ is defined in Section \ref{NotationandConvention}. Then for any
$\varrho\geq0$, we define
\begin{equation}\RectFunc^\varrho=\{\phi\in C^\varrho(\Rectangle\times V)
\mid\phi|_{\{t=0,1\}\times V}=0 \},		
\label{DefinitionofRectFunc}
\end{equation}
\begin{equation}	\tmF^\varrho=\{\phi\in C^\varrho(\tcC\times V)\mid\phi|_{\{t=0,1\}\times V}=0 \}.	\label{DefinitionoftmF}
\end{equation}
When we don't need to specify $\varrho$, we will simply drop index $\varrho$.

For convenience, we define $\Str: \RectFunc^\varrho\rightarrow \tmF^\varrho$
by
\begin{equation}(\Str\phi)(\theta,t,z)=\left\{
\begin{array}{ccc}
\phi(\theta-\Theta+1, t,z),& \text{if} &\theta>\Theta,\\
0					,&\text{if}  &-\Theta\leq \theta\leq \Theta,\\
\phi(\theta+\Theta-1, t,z),& \text{if}& \theta<-\Theta.
\end{array}\right.																\label{DefinitionofStrSep28}
\end{equation}
\begin{remark}
We will define norms of functions on $\tcC\times V$
in such way that for any  $\phi\in \RectFunc^\varrho$ and $0\leq\nu\leq \varrho$,
\begin{equation}
|\Str\phi|_{\nu;\tcC\times V}\leq C |\phi|_{\nu; \Rectangle\times V},				
\label{ExtensionisVeryGood}
\end{equation}
for some $C$ independent of $\Theta$ and $\nu(\text{or }\varrho).$

To define these norms on $C^\varrho(\tcC\times V)$, we need to specify an open
covering of $\tcC\times V$ and partitions of unity subordinate to the covering.
Note that functions in $\tmF$ vanish on
$(\tcC\cap\{|\theta|\le\Theta-\half\})\times V$. Then, in order
to make (\ref{ExtensionisVeryGood}) valid, we need to make the open covering and
partition of unity of $\tcC\times V$ over
$(\tcC\cap\{|\theta|>\Theta-\half\})\times V$ independent of $\Theta$.
This can be done because by definition $\tcC\cap\{|\theta|>\Theta-\half\}$ is
the same as $\cC\cap\{|\theta|>\half\}$.
\end{remark}
For
 $(\varphi_0, \varphi_1)\in C^\varrho(V)\times C^\varrho(V)$, we denote
\begin{align*}
&\bvarphi=(\varphi_0,\varphi_1),\\
&\varphit=(1-t)\varphi_0+t\varphi_1, \quad\mbox{ for } t\in[0,1],
\end{align*}
and for any $q\geq 0$
\[
|\bvarphi|_q=|\varphi_0|_q+|\varphi_1|_q.
\]
We consider $\varphit$ introduced above as a function on $\sS\times V$,
 defined by
$$
(\theta, t, z) \mapsto (1-t)\varphi_0(z)+t\varphi_1(z).
$$
Then $\varphi_t\in C^\varrho(\sS\times V)$.

Now given $0<\fractionSpace<1,$ $\bvarphi\in C^\varrho(V)\times C^\varrho(V)$,
 $\phi\in \RectFunc^\varrho$, with $\varrho\geq \fourthird$,
 $|\bvarphi|_{\fourthird}$ and $|\phi|_{\fourthird}$ sufficiently small,
 according to Proposition  \ref{prop:DonaldsonDiscDirichletProblem},
 we can find $\Phi\in C^{\varrho-2}(\ctrR\times V)$ solving the  following problem,
\begin{equation}\label{DirichletProblemDefIterMap}
\left\{
\begin{array}{ccc}
(\Omega_0+\sqrt{-1}\ddbar\Phi)^{n+1}=0,& \text{in} &\trR\times V,\\
\Phi=\varphit+\Str(\phi) 				,&\text{on}& \tcC\times V.
\end{array}\right.
\end{equation}
We define $\mB_\bvarphi(\phi)$ by
\[\mB_\bvarphi(\phi)=\Phi|_{\Rectangle\times V}-\varphit.\]
This defines $\mB_\bvarphi$ as a map from a neighborhood of zero in
$\RectFunc^\varrho$ into $\RectFunc^{\varrho-2}$.

Further requirements on the size of $|\bvarphi|_{\fourthird}$ and
$|\phi|_{\fourthird}$, and the
regularity estimate of $\mB_\bvarphi(\phi)$ will be discussed more carefully in
Section \ref{ApplyingMoserTheorem}.

Given $\bvarphi\in C^\varrho(V)\times C^\varrho(V)$,
assume that $\phi\in\RectFunc^\varrho$ is a fixed point
of $\mB_\bvarphi$. Then the corresponding $\Phi$ solving
(\ref{DirichletProblemDefIterMap})
is a real-valued function on
$\ctrR\times V$, satisfying:
\[
\left\{
\begin{array}{cc}
(\Omega_0+\sqrt{-1}\ddbar\Phi)^{n+1}=0, &\text{in} \  \trR\times V,\\
\Phi|_{\{t=0\}\times V}=\varphi_0, &\Phi|_{\{t=1\}\times V}=\varphi_1,\\
\Str(\Phi-\varphi_t)+\varphi_t=\Phi|_{\tcC\times V}.
\end{array}
\right.
\]

 Denote
\begin{align*}
& \rR^-:=(\rR\cap\{2>\theta\geq1\})\cup(\trR\cap\{1>\theta\}),\\
&\rR^+:=(\trR\cap\{\theta>-1\})\cup(\rR\cap\{-1\geq\theta>-2\}).
\end{align*}
So, $\rR^-$ is a rightward translation of $\rR^+$ by  $\Theta-1$.
Noting that $\varphi_t$ does not depend on $\theta$, we conclude
that condition $\Str(\Phi-\varphi_t)+\varphi_t=\Phi|_{\tcC\times V}$ implies
that the restriction of $\Phi$ on $\partial \rR^+\times V$ is identical to the restriction
of $\Phi$ on $\partial \rR^-\times V$ under the translation in $\theta$-direction.
Combining  this with
the uniqueness  of Dirichlet Problem
(Corollary 7 of \cite{DonaldsonSymmetricSpace}), we conclude
that  the restriction of $\Phi$ on $\rR^+\times V$ is identical to the restriction of
$\Phi$ on $ \rR^-\times V$, under the same translation. So, $\Phi$ can be extended as a periodic function of $\theta$ on $\sS\times V$, with period $\Theta-1$.

Then $\Phi$  can be viewed as a function defined on the product of a cylinder and the K\"ahler manifold:
\[(\sS\slash \thicksim)\times V,  \text{ where $(t,\theta+\Theta-1,z)\thicksim (t, \theta, z)$}.\] And because $\Phi$ is independent of $\theta$ on $\{t=0,1\}$, for any $\Lambda\in \ER$, $\Phi(t,\theta+\Lambda, \ast)$ also satisfies the HCMA equation and has the same boundary values as $\Phi$. Using the uniqueness theorem (Corollary 7) of \cite{DonaldsonSymmetricSpace}, we can conclude that $\Phi(t,\theta+\Lambda, \ast)=\Phi(t,\theta,\ast)$, i.e.  $\Phi$ is independent of $\theta$, so it's a geodesic.

Thus it remains to show the existence of fixed points of
$\mB_\bvarphi$, if  $|\bvarphi|_{\fourthird}$ is sufficiently small.
The outline of our argument is following.
We will
apply a Moser-type inverse function theorem in the following setting. Define
\begin{align}
&\sF^\varrho=C^\varrho(V)\times C^\varrho(V)\times\RectFunc^\varrho,
\label{DefinitionofsF}
\\
&\sN^\varrho_\delta=\{(\vfcf)\in \sF^\varrho\mid |
\bvarphi|_\fourthird+|\phi|_\fourthird\leq \delta\},
\end{align}
and for some $\delta=\delta(\varrho)$ small enough, depending on $\varrho$,
\begin{equation}\label{defOfIterMap-P}
\begin{split}
&\sP:\sN^\varrho_{\delta} \rightarrow \sF^{\varrho-2}
\mbox{,\ \ \ \ \ \ \  by}\\
&\sP(\vfcf)=(\bvarphi,\mB_\bvarphi(\phi)-\phi).
\end{split}
\end{equation}
Note that $\sP(0,0)=0$. Then if a neighborhood, in some norm,
of zero in $\sF^{\varrho-2}$ is contained in $\sP(\sN^\varrho_{\delta})$, in particular
we get
\[\{(\bvarphi,0)\mid\bvarphi\text{ small enough in some norm}\}
\subset \sP(\sN^\varrho_{\delta}),\]
and we obtain a fixed point of $\mB_\bvarphi$, for $\bvarphi$ small with
respect to some norm.
This implies, ignoring regularity, that if
$(\varphi_0,\varphi_1,0)=\sP(\varphi_0,\varphi_1,\phi)$,
then we can connect $\varphi_0$ and $\varphi_1$ by the curve $\varphit+\phi$, $t\in[0,1]$
 in the space of K\"ahler potentials. This construction will be discussed
 more precisely  in the following sections.

\subsection{Applying Moser's Theorem}\label{ApplyingMoserTheorem}

Now we show how to apply Lemma \ref{lem:MoserTheorem} to above construction. Our argument will be in 4 steps.

{\flushleft{\bf Step 1}.  {\em Specifying Indexes and Function Spaces}}

Let $k$ and $J$ be the numbers in Theorem \ref{thm:RegularityofNearConstantGeodesics}.
We will use Proposition \ref{prop:DonaldsonDiscDirichletProblem} with $$\fractionSpace=\min\{{1\over 3}, {k-4\over 3}\}.$$

Then in Lemma \ref{lem:MoserTheorem}, let
\[ B=k,\ \alpha={\jop}, \ b=4+\fractionSpace,\  l=2,\ \chi=4,\]
and
\[r=\zeta+{1\over 3},\] for some $\zeta\in \EZ^+$ big enough, such that
\begin{equation}r>k, \label{r>k}
\end{equation}
\begin{equation}
\frac{k}{4}>\frac{(1+{4\over r})^3(r-k)}{r-8},   \label{Pluggedink>chi}
\end{equation}
\begin{equation}
\frac{k}{r-k}<1+{4\over r},								\label{PluggedinIntervalMn}
\end{equation}
\begin{equation}
\frac{k(r-(\fourthird))}{(\fourthird)(r-k)}>(1+{4\over r})^2,
\end{equation}
\begin{equation}
\frac{r^3(r-k+{2+\jop})}{(r+4)^3(r-k)}>\frac{k-\jop}{k}.  \label{PluggedinConvergewrtBalpha}
\end{equation}
Note that (\ref{r>k})-(\ref{PluggedinConvergewrtBalpha}) can be satisfied by
simply letting $\zeta\rightarrow \infty$.

Requirements
(\ref{r>k})-(\ref{PluggedinConvergewrtBalpha}) imply (\ref{IndexOrderMoserTheorem})-(\ref{ConvergeWRTBalphaNorm})
for $B,\ b,\ l,\ \chi,\ \alpha$ specified above: in fact, (\ref{r>k})-(\ref{PluggedinConvergewrtBalpha})
 come from plugging these values of $B,\ b,\ l,\ \chi,\ \alpha$ into
 (\ref{IndexOrderMoserTheorem})-(\ref{ConvergeWRTBalphaNorm}), respectively.

Note that $k>4$ and $\jop<\min\{{1\over 4}, \frac{k-4}{4}\}$ by the conditions of
Theorem \ref{thm:RegularityofNearConstantGeodesics}. Then we have
\[k-\alpha=k-\jop>k-\min\{{1\over 4}, \frac{k-4}{4}\}>4+\min\{{1\over 3}, \frac{k-4}{3}\}=b.\]

From now on, we will fix $r$, and when there is no ambiguity, constants only depending on
$r,\ B,\ b,\ l,\ \chi,\ \alpha$ and manifold $(V,\omega_0)$, will be denoted by $C$.
In other cases, e.g. when we refer to estimates in Section {\ref{DirichletProblem}}, we will use
notation,
like $\delta_D(r)$,  for clarity. Also, before we fix $\Theta$ in Step 4, constants depending on $\Theta$ will  be denoted by $\delta(\Theta)$ or $C(\Theta)$ etc..

We will work in function space $\sF^{\varrho}=C^{\varrho}(V)\times C^{\varrho}(V)\times
\RectFunc^\varrho$, for some $\varrho$, where we used the space $\RectFunc^\varrho$ defined in (\ref{DefinitionofRectFunc}).

For $(\varphi_0,\varphi_1,\phi)\in C^{\varrho}(V)\times C^{\varrho}(V)\times \RectFunc^\varrho$,
with $\varrho\geq 0$ and $N\geq 1$, we define
\[S_N(\varphi_0,\varphi_1,\phi)=(\NashSmoothingOperator_N\varphi_0,\NashSmoothingOperator_N\varphi_1,\tNashSmoothingOperator_N\phi),\]
where $\NashSmoothingOperator_N$ is the smoothing operator constructed in
\cite[Part A]{NashImbedding}, and  $\tNashSmoothingOperator_N$ is a modification of $\NashSmoothingOperator_N$ which we define in the following.

First we extend $\phi$ to be a $C^\varrho$ function in $\{-3<\theta<3\}\times\{-1<t<2\}\times V$, with compact support, and we denote this new function by $E(\phi)$. By using Whitney's extension theorem we can make
\[
|E(\phi)|_\varrho\leq C(\varrho)|\phi|_\varrho
,\]
and
\[
E(\phi)=\phi   \text{, in } \{-2<\theta<2\}\times\{0<t<1\}\times V
.\]
Then we mollify $E(\phi)$ with smoothing operator $\NashSmoothingOperator_N$, and define
\[\renewcommand\arraystretch{1.5}
\tNashSmoothingOperator_N(\phi)(\theta,t,z)=
				\left\{
						\begin{matrix}
								(\NashSmoothingOperator_NE(\phi))(\theta,t,z)-\eta(t)\cdot(\NashSmoothingOperator_NE(\phi))(\theta,0,z),	\ \ 	&\text{ for } t\leq {1\over 2};\\
								(\NashSmoothingOperator_NE(\phi))(\theta,t,z)-\eta(1-t)\cdot(\NashSmoothingOperator_NE(\phi))(\theta,1,z),&\text{ for } t\geq {1\over 2},
						\end{matrix}
				\right.
\]
where $\eta$ is a function in $C^\infty(\ER, \ER)$, satisfying
\[0\leq \eta\leq 1,\]and
\[\renewcommand\arraystretch{1.3}
\eta(t)=
				\left\{
						\begin{matrix}
								1,	\ \ 	&\text{ for } |t|\leq {1\over 6};\\
								0,	\ \ 	&\text{ for } |t|\geq {1\over 3}.
						\end{matrix}
				\right.
\]
We note that $\tNashSmoothingOperator_N(\phi)=0$ on $\{t=0\}$ and $\{t=1\}$, so it maps $\RectFunc^\varrho$ into $\RectFunc^\infty$,
which is the reason for this modification.

It is easy to check that $S_N$ satisfies conditions (\ref{MoserAbstractSettingBlowUpofSmoothing})
and (\ref{MoserAbstractSettingConvergenceofSmoothing}), just note that,
\[
|\NashSmoothingOperator_N E(\phi)(\spacecdot,t_0, \spacecdot)|_\nu\leq C(\varrho)\cdot N^{\nu-\varrho}|\phi|_\varrho, \text{ for } t_0=0,1, \text{ and any }0\leq\nu<\varrho,
\] because $E(\phi)=0$ on $\{t=0,1\}$,
 the remaining argument is similar to the proof in \cite[Part A]{NashImbedding}.

Mapping $\sP$  in (\ref{defOfIterMap-P}) will be defined  in a neighborhood of zero in  $\sF^\varrho$:
\[\sN^{\varrho}_{\delta(\varrho)}=\{(\vfcf)\in \sF^{\varrho}\mid \vffnormfourthird \leq \delta\},\]
for $\delta$ small, and $\varrho=r+2, 4+\fractionSpace$.\\

 In Steps 2 and 3, we
show that for any $\Theta>2$
there is a $\delta_1(\Theta)$ such that
 $\sP$ is well-defined
and differentiable in $\sN_{\delta_1(\Theta)}^{\varrho}$. In Step 4 we find that, if we  choose $\Theta$ big enough, then there is a
$\tilde\delta_2(\Theta)>0$ such that for $(\vfcf)\in\sN^{r+2}_{\min\{\tilde\delta_2(\Theta),\delta_1(\Theta)\}}$
the linear map $D\sP_{\vfcf}$ is invertible. We then fix $\Theta$, and let the $\epsilon$ in
Lemma \ref{lem:MoserTheorem} be equal to
  $$\delta_2=\min\{\tilde\delta_2(\Theta),\delta_1(\Theta)\}.$$

{\flushleft {\bf Step 2.} {\em Iteration Map}}

$\trR$ is holomorphically equivalent to the disc $D$ by Riemann Mapping Theorem, and the biholomorphic map from $\trR$ to $D$ is smooth up to boundary by Theorem 1 of \cite{Kerzman}. Here $D$ is the unit disc on complex plane, with center at zero.

Let $\tT$ be the unique holomorphic map from $\trR$ to $D$ satisfying:
\[\tT(0)=-\sqrt{-1}, \ \ \ \tT(\frac{1}{2})=0.\]
Then Theorem 1 of \cite{Kerzman} implies for any $\varrho>0$, there exists a $C(\varrho,\Theta)$, such that
\begin{equation}
|\tT|_{\varrho;\overline{\trR}}\leq C(\varrho,\Theta).								\label{RegularityoftT}
\end{equation}

Then $\trR\times V$ is holomorphically equivalent to $D\times V$ through $\TtrR$, where
\[\TtrR(\tau,z)=(\tT(\tau),z).\]
Following (\ref{RegularityoftT}), we have
\begin{equation}
|\TtrR|_{\varrho;\overline{\trR}\times V}\leq C\cdot C(\varrho,\Theta),  \ \ \forall \varrho>0.   \label{RegularityofTtrR}
\end{equation}

Since $\trR\times V$ and $D\times V$ are holomorphically equivalent to each
other through a biholomorphic map smooth up to boundary,
 Problem \ref{prob:DirichletonD} is  equivalent to the following  problem:
\begin{problem}
[Dirichlet Problem of HCMA equation on $\trR\times V$ with boundary value $M$]
								\label{prob:DirichletontrR}
Given
a K\"ahler manifold
$(V,\omega_0)$, with $\omega_0>0$, and  a smooth real valued function $M$ on
 $\partial \trR\times V$ satisfying
\[
\omega_0+\sqrt{-1}M(\tau, \cdot)_{i\overline j}dz^i\wedge\overline{dz^j}>0,
\ \ \ \ \ \ \ \ \ \  \text{on } \{\tau\}\times V
\;\text{ for all }\;\tau \in \partial \trR,
\]
and denoting by
 $\pi_V$ the trivial projection from $\trR\times V$ to $V$, and setting
 $\Omega_0=\pi_V^\ast(\omega_0)$, we look for $\Phi\in C^2(\trR\times V\torange \ER)\cap C^0(\ctrR\times V\torange \ER)$, satisfying
\[\left\{
\begin{array}{cc}
(\Omega_0+\sqrt{-1}\ddbar\Phi)^{n+1}=0,&\text{in }\trR\times V;\\
\Phi=M		,						&\text{on } \partial \trR\times V;\\
\omega_0+\sqrt{-1}\Phi(\tau, \cdot)_{i\overline j}dz^i\wedge\overline{dz^j}>0, & \text{on } \{\tau\}\times V, \forall\tau \in \trR.
\end{array}
\right.\]
\end{problem}

Given $(\vfcf)\in\sF^{r+2}$, suppose  $\Psi$ is a solution to
Problem \ref{prob:DirichletontrR} with boundary data
$M=\varphit+\Str(\phi)$.
Then $\TtrR_\ast(\Psi)$ is a
solution to Problem \ref{prob:DirichletonD} with boundary data $\TtrR_\ast(\varphit+\Str(\phi))$.

By Theorem A.8 of \cite{HormanderPhysicalGeodesy},
\[|\TtrR_\ast(\varphit+\Str(\phi))|_{\fourthird}\leq
 C(\Theta)(|\varphit+\Str(\phi)|_\fourthird).\]
Then by Proposition \ref{prop:DonaldsonDiscDirichletProblem}, if
\[\vffnormfourthird\leq \frac{\delta_D(r+2)}{C(\Theta)}\triangleq \delta_1(\Theta),
\]
there exists a solution $\Phi$ to Problem \ref{prob:DirichletonD}
with boundary data $\TtrR_\ast(\varphit+\Str(\phi))$.
It follows that $\TtrR^{\ast}(\Phi)$ is
a solution to Problem \ref{prob:DirichletontrR}, with boundary data $\varphit+\Str(\phi)$.

For convenience,
\begin{equation}\label{notationSolProbl-2}
\text{\em
solution to Problem \ref{prob:DirichletontrR} with boundary
data $\varphit+\Str(\phi)$ will be denoted by $\DtrR(\vfcf)$.}
\end{equation}
Then for $(\vfcf)\in \sN_{\delta_1(\Theta)}^{r+2}$ we have
$\DtrR(\vfcf)\in C^r(\ctrR\times V)$ by Proposition \ref{prop:DonaldsonDiscDirichletProblem}, and
\begin{equation}\label{regularitySolProbl-2}
|\DtrR(\vfcf)|_r\leq C(\Theta)(|\bvarphi|_{r+2}+|\phi|_{r+2}).
\end{equation}

As suggested in last subsection (\ref{defOfIterMap-P}), we define
\begin{equation}\label{IterMapBdef}
\begin{split}
&\mB_\bvarphi(\phi)=\DtrR(\vfcf)|_{\Rectangle\times V}-\varphit,\\
&\sP(\vfcf)=(\bvarphi,\mB_{\bvarphi}(\phi)-\phi).
\end{split}
\end{equation}

Let $(\vfcf), (\hat\bvarphi,\hat\phi)\in\sN^{r+2}_{\delta_1(\Theta)}$ be given.
Denote $\Psi:=\DtrR(\vfcf)$, $\hat\Psi:=\DtrR(\hat\bvarphi,\hat\phi)$,
where we use (\ref{notationSolProbl-2}).
Using estimate (\ref{CrLipPotentialDisc}) of Proposition {\ref{prop:DonaldsonDiscDirichletProblem}},
we obtain on $D\times V$
\begin{align*}|\TtrR_\ast(\Psi-\hat\Psi)|_r&\leq C(|\TtrR_\ast(\varphit-\hat\varphit+\Str(\phi-\hat\phi))|_{r+2}\nonumber\\
&+(1+|\TtrR_\ast(\varphit+\Str(\phi))|_{r+2}+|\TtrR_\ast(\hat\varphit+\Str(\hat\phi))|_{r+2})(|\TtrR_\ast(\varphit-\hat\varphit+\Str(\phi-\hat\phi))|_{\fourthird})
\end{align*}
so we have for  $\Psi-\hat\Psi$  on $\trR\times V$
\begin{align}|\Psi-\hat\Psi|_r\leq C(\Theta)&
\left(|(\hat\bvarphi-\bvarphi,\hat\phi-\phi)|_{r+2}\right.\nonumber\\
&   \ \ \left.+
(1+|(\hat\bvarphi,\hat\phi)|_{r+2}+|\vfcf|_{r+2})
(|(\hat\bvarphi-\bvarphi,\hat\phi-\phi)|_{\fourthird})
\right), 		\label{RefinedLipontrR}
\end{align}
where the norms in the right-hand side are taken on the domain
$\Rectangle\times V$ or $V$.
Then restricting $\Psi-\hat\Psi$ to $\Rectangle\times V$, we have
\begin{align}|\sP(\vfcf)-\sP(\hat\bvarphi,\hat\phi)|_r\leq C(\Theta)&((|\hat\bvarphi-\bvarphi,\hat\phi-\phi|_{r+2})\nonumber\\
&+
(1+|\hat\bvarphi,\hat\phi|_{r+2}+|\vfcf|_{r+2})
(|\hat\bvarphi-\bvarphi,\hat\phi-\phi|_{\fourthird})
).		\label{CrLipofsP}
\end{align}
Also, we obviously have
\[\sP(0,0)=0.\]
Thus Condition 1 of Lemma \ref{lem:MoserTheorem} holds.  Estimate
(\ref{CrLipofsP}) is stronger than (\ref{CrEstimateofsPMoserTheorem}), but we
will need to use (\ref{CrLipofsP}) later in this section.

Condition 2 also holds: indeed to check (\ref{LipEstimateofsPMoserTheorem}),
we note that a stronger estimate:
\begin{equation}\label{maxPrincipleIteration22}
|\sP(f_1)-\sP(f_2)|_0\leq C|f_1-f_2|_0
\end{equation}
 can be proved with maximum principle (Lemma 6 of \cite{DonaldsonSymmetricSpace}),
by the argument similar to the proof of Corollary 7 of \cite{DonaldsonSymmetricSpace}.
Precisely, let $\Phi$ and $\tilde\Phi$ be two solutions of Problem \ref{prob:DirichletontrR}
with Dirichlet data $M$ and $\tilde M$ respectively. Then taking
$\hat\Omega=\Omega_0+\sqrt{-1}\ddbar\Phi$,
we have $\hat\Omega+\sqrt{-1}\ddbar(\tilde\Phi-\Phi)
=\Omega_0+\sqrt{-1}\ddbar\tilde\Phi>0$ on every slice
$\{\tau\}\times V$, $\tau\in \trR$ we deduce from
Lemma 6 of \cite{DonaldsonSymmetricSpace} that maximum value of
$\Phi-\tilde\Phi$ is attained on the boundary, i.e.
$$
|\Phi-\tilde\Phi|_0\le |M-\tilde M|_0.
$$
This implies (\ref{maxPrincipleIteration22}).

{\flushleft{{\bf Step 3.} {\em  Differentiability}}}

To describe the structure of the tangential map of $\sP$, we need to use two operators, $\AtrR$ and $\HtrR$, which we define in the following.

{\flushleft{\em 1. Definition and Estimate of $\AtrR$}}

As we discussed in Section \ref{DirichletProblem}, if $\Phi\in C^3{(D\times V)}$ satisfies
\[
\left\{\begin{array}{cc}
\left(\Omega_0+\sqrt{-1}\ddbar\Phi\right)^{n+1}=0, &\text{ in } D\times V;\\
\omega_0+\sqrt{-1} \Phi_{i\overline j}(\tau,\cdot)dz^i\wedge\overline{dz^j}>0,&\text{ on }\{\tau\}\times V, \ \forall\ \tau\in D,
\end{array}\right.			
\]
then kernels of $\Omega_0+\sqrt{-1}\ddbar \Phi$ form a foliation on $D\times V$ and we can find a map $\ADF$, satisfying properties (\ref{EquationofADFAug13}), which maps the trivial production foliation to this foliation.
We have same result for $\trR\times V$. Precisely, given $(\vfcf)\in \sN^{r+2}_{\delta_1(\Theta)}$,
kernels of $\Omega_0+\sqrt{-1}\ddbar\DtrR(\vfcf)$ form a foliation on $\trR\times V$ which we denote by $\Foliation(\vfcf).$ As in the case of $D\times V$, we have a map $\AtrR\in C^r(\ctrR\times V\torange\ctrR\times V)$, which satisfies
\begin{equation*}
\left\{
\begin{array}{l}
\pi_{\trR}\circ \AtrR=\pi_{\trR};\\
\vspace{1ex}
\AtrR=Id ,  \text{\ \ \ \ \ \  on } \{\tau=0\}\times V;\\
\vspace{1ex}
\text{for each }  z\in V,
\text{ restriction of }   \Omega_0+\sqrt{-1}\ddbar\DtrR(\vfcf)   \text{  to }
\AtrR(\trR\times \{z\})
 \text{ vanishes};\\
\text{for each }  z\in V,
\text{ restriction of }    \AtrR   \text{  to }  \trR\times \{z\}
 \text{  is holomorphic.}
 \end{array}
\right.
\end{equation*}
where $\pi_{\trR}$ is the trivial projection from $\trR\times V$ to $\trR$.
We denote this map $\AtrR$ by $\AtrR_{\vfcf}$.

Recall that in Section \ref{DirichletProblem},  existence, regularity and estimates
of the map $\AD_F$ with properties
(\ref{EquationofADFAug13}) were shown in
Proposition \ref{prop:DonaldsonDiscDirichletProblem},
for $F\in C^{\varrho+2}(\partial D\times V)$, $|F|_\fourthird\leq \delta_D(\varrho+2)$,
provided $\varrho+2\geq \fourthird$. Also, it's easy to see
\[\AtrR_\vfcf=\TtrR^{-1}\circ\AD_{\TtrR_{\ast}(\varphit+\Str(\phi))}\circ\TtrR.\]
Now, from (\ref{CrEstimateofADFDisc})--(\ref{CrEstimateofADFiDisc}),
using Theorem A.8 of \cite{HormanderPhysicalGeodesy}, 
we get, for any $r\geq\varrho\ge 2+X$,
\begin{equation}|\AtrR_\vfcf-Id|_\varrho=\left|\left(\TtrR^{-1}\circ\AD_{\TtrR_{\ast}(\varphit+\Str(\phi))}-\TtrR^{-1}\circ Id\right)\circ\TtrR\right|_{\varrho}\leq C(\Theta)(\vffnormvarrho).
\label{EstimateofAtrRIdr}
\end{equation}
Similarly we have:
\begin{equation}|\AtrR^{-1}_\vfcf-Id|_\varrho=\left|\left(\TtrR^{-1}\circ\AD^{-1}_{\TtrR_{\ast}(\varphit+\Str(\phi))}-\TtrR^{-1}\circ Id\right)\circ\TtrR\right|_{\varrho}\leq C(\Theta)(\vffnormvarrho).
\label{EstimateofAtrRiIdr}
\end{equation}

{\flushleft{\em 2. Definition and Estimate of $\HtrR$}}

Given $F\in C^{r}(\tcC\times V),$ we can solve the following Dirichlet Problem:
\[
\left\{
\begin{array}{cc}
\triangle_\tau H=0 , & \text{ in }\trR\times V;\\
H=F,&   \text{ on }\tcC\times V,
\end{array}
\right.
\]
and get $H\in C^{r}(\ctrR\times V)$. We denote,
\[H=\prodh(F).\]
Then $\prodh$ is a linear map from $C^r(\tcC\times V)$ to $C^r(\ctrR\times V)$, with
\begin{equation}\label{EstDirichletWithParamInStrip}
|\prodh(F)|_r\leq C(\Theta)|F|_r,
\end{equation}
where this estimate is obtained similarly to the estimate
(\ref{EstDirichletWithParam}) for problem
(\ref{DirichletWithParam}), using Lemma \ref{lem:FamilyofPossionEquationwithBernsteinandExtension20180315}.
Note that in (\ref{EstDirichletWithParamInStrip}),
 constant $C(\Theta)$ may depend on $\Theta$, and may go to $\infty$ as $\Theta$ goes to $\infty$. But if $F$ vanishes on $\{t=0,1\}\times V$ and we restrict $\prodh(F)$ to $\Rectangle\times V$, we have better estimate, which is the main result of Appendix \ref{FamilyofDirichletProblemofHarmonicFunctiononLongStrip}.

Given $(\vfcf)\in \sN^{r+2}_{\delta_1(\Theta)}$, and $F\in C^r(\tcC\times V)$, define
\begin{equation}
\HtrR_{\vfcf}(F)={(\AtrR_{\vfcf})}_\ast(\prodh(\AtrR_\vfcf^\ast(F))),
\label{ExpressionofHtrRAug28}
\end{equation}
which means that $\HtrR_{\vfcf}(F)$ is leafwise harmonic function on the foliation
$\Foliation(\vfcf)$:
\begin{equation}
\left(\Omega_0+\sqrt{-1}\ddbar\DtrR(\vfcf)\right)^n\wedge\ddbar\HtrR_{\vfcf}(F)=0, \ \ \ \ \text{ in } \trR\times V,   		\label{LeafwiseHarmonicEquationofHtrR}
\end{equation}
and
\begin{equation}\HtrR_{\vfcf}(F)=F,\text{ on } \partial \trR\times V.
\label{HtrRBoundaryCondition}
\end{equation}
Also, from (\ref{EstimateofAtrRIdr})--(\ref{EstDirichletWithParamInStrip}), we have
 $\HtrR_\vfcf(F)\in C^r(\ctrR\times V)$.

We expect that  the tangential map (in the sense of Condition 3 of Lemma \ref{lem:MoserTheorem}) of $\sP$ at $(\vfcf)\in \sF^{r+2}$ is
\begin{equation}
\label{DerivativeIterMap}
D\sP_{\vfcf}: (u_0, u_1,v)\rightarrow(u_0, u_1,\HtrR_{\vfcf}((1-t)u_0+tu_1+
\Str(v))|_{\Rectangle\times V}-(1-t)u_0-tu_1-v),
\end{equation}
for $(u_0, u_1,v)\in \sF^{r+2}$. We check the validity in the following.

We continue to use notations introduced in (\ref{notationSolProbl-2}) of Step 2:
for $(\vfcf), (\hat\bvarphi,\hat\phi)\in\sN_{\delta_1(\Theta)}^{r+2}$ denote
\[ \Psi:=\DtrR(\vfcf),\ \ \ \ \hat\Psi:=\DtrR(\hat\bvarphi,\hat\phi).\]

To get second order estimate (\ref{SecondOrderEstimateMoserTheorem}), we
subtract the expected linear part from the difference of $\sP(\vfcf)$ and $\sP(\hat\bvarphi,\hat\phi)$:
\begin{align*}
	&\sP(\hat\bvarphi,\hat\phi)-\sP(\vfcf)-(\hat\bvarphi-\bvarphi,\HtrR_{\vfcf}((\hat\varphit-\varphit)+\Str(\hat\phi-\phi))|_{\Rectangle\times V}-\hat\phi+\phi-\hat\varphit+\varphit)\\
=	&(0, \mB_{\hat\bvarphi}(\hat\phi)+\hat\varphit-\mB_\bvarphi(\phi)-\varphit-\HtrR_{\vfcf}((\hat\varphit-\varphit)+\Str(\hat\phi-\phi))|_{\Rectangle\times V})\\
=	&(0, \underbrace{\left(\hat\Psi-\Psi-\HtrR_{\vfcf}((\hat\varphit-\varphit)+\Str(\hat\phi-\phi))\right)}_{\triangleq Z}\Big|_{\Rectangle\times V}),
\end{align*}
where we used (\ref{IterMapBdef}).
Now we need to estimate $L^{\infty}$ norm of $Z$.

Taking difference of the equations and boundary conditions  satisfied by
$\hat\Psi$ and $\Psi$, as given in Problem \ref{prob:DirichletontrR},  we get
\[
\left\{
\begin{array}{cc}
{\sqrt{-1}\ddbar(\hat\Psi-\Psi)\wedge(\Omega_0+\sqrt{-1}\ddbar\Psi)^n}\cdot (n+1) \ \ \  \ \ \ \ \ \ \ \ \ \  \ \ \ \ \ \ \ \ \  \ \ \ \ \ \ \ \ \ \ \ \ \ \ \  \ \ \ \ \ \ \\
 \ \ \ \ \ \ \ \ \  \ \ \ \ \ \ +\sum_{i=2}^{n+1}\left(\begin{matrix}n+1\\i\end{matrix}\right)(\sqrt{-1}\ddbar(\hat\Psi-\Psi))^i\wedge(\Omega_0+\sqrt{-1}\ddbar\Psi)^{n+1-i}=0,&\text{ in }\trR\times V;\\
\hat\Psi-\Psi=\hat\varphit-\varphit+\Str(\hat\phi-\phi),&\text{ on }{\tcC\times V},
\end{array}
\right.
\]
then taking difference of
(\ref{LeafwiseHarmonicEquationofHtrR})--(\ref{HtrRBoundaryCondition})
with
$F=\hat \varphi_t-\varphi_t+\Str(\hat\phi-\phi)$
and the above equations, we find $Z$ satisfies
\[
\left\{
\begin{array}{cc}
{\sqrt{-1}\ddbar Z\wedge(\Omega_0+\sqrt{-1}\ddbar\Psi)^n}\cdot (n+1)  \ \ \  \ \ \ \ \ \ \ \ \ \  \ \ \ \ \ \ \ \ \  \ \ \ \ \ \ \ \ \ \ \ \ \ \ \  \ \ \ \ \ \ \\
 \ \ \ \ \ \ \ \ \  \ \ \ \ \ \ +\sum_{i=2}^{n+1}
 \left(\begin{matrix}n+1\\i\end{matrix}\right)(\sqrt{-1}\ddbar(\hat\Psi-\Psi))^i
 \wedge(\Omega_0+\sqrt{-1}\ddbar\Psi)^{n+1-i}=0,&\text{ in }\trR\times V;\\
Z=0,&\text{ on }{\tcC\times V}.
\end{array}
\right.
\]
From the maximum estimate for Laplace equation on the leaves
$\AtrR(\trR\times \{z\})$, $z\in V$, we have
\begin{align}
|Z|_0\leq C|\hat\Psi-\Psi|_2^2		.				\label{EstimateofZbeforInterpolation}
\end{align}
We can use interpolation to estimate the right hand side of above inequality
to get
\begin{equation}
|Z|_0\leq	C|\hat\Psi-\Psi|_2^2\leq C(\Theta)|\hat\Psi-\Psi|_0^{2-{4\over r}}|\hat\Psi-\Psi|_r^{4\over r}\label{EstimateofZafterInterpolation}
\end{equation}
Then applying maximum principle Lemma 6 of \cite{DonaldsonSymmetricSpace}
 as in the proof of (\ref{maxPrincipleIteration22}), we have the estimate
\begin{equation}
|\hat\Psi-\Psi|_0\leq |\hat\bvarphi-\bvarphi|_0+|\hat\phi-\phi|_0.
\end{equation}

Now we combine  (\ref{RefinedLipontrR}) and the last three inequalities, to obtain
\begin{align}
|Z|_0	\leq& C(\Theta)\left(|\hat\bvarphi-\bvarphi|_0+|\hat\phi-\phi|_0^{}\right)^{2-{4\over r}}\nonumber\\
   \cdot&\left((|\hat\bvarphi-\bvarphi|_{r+2}+|\hat\phi-\phi|_{r+2})+
(1+|\hat\bvarphi,\hat\phi|_{r+2}+|\vfcf|_{r+2})
(|\hat\bvarphi-\bvarphi,\hat\phi-\phi|_{\fourthird})
\right)^{{4\over r}}.
						\label{EstimateofZafterUsedLip}
\end{align}

This confirms (\ref{SecondOrderEstimateMoserTheorem}), and thus
 Condition 3 of Lemma \ref{lem:MoserTheorem}.

{\flushleft{{\bf Step 4.} {\em  Invertibility}}.}    

Here, we check Condition 4 of Lemma \ref{lem:MoserTheorem}.

Recall that we set $l=2$ in our present application of Lemma \ref{lem:MoserTheorem}.
Thus in this step, $(\vfcf)$ will be in $\sN_{\delta_1(\Theta)}^{r+2}$, and so
$\AtrR_{\vfcf}$ is in $ C^r(\ctrR\times V\torange \ctrR\times V)$ by (\ref{EstimateofAtrRIdr}).
For $u_0, u_1\in C^{r}(V)$,  we denote
\[\vecu=(u_0, u_1), \]
\[u_t=(1-t)u_0+tu_1,\]
and for any $q\geq 0$
\[|\vecu|_q=|u_0|_q+|u_1|_q.\]

In Step 3 we showed that, in the weak sense (as in Condition 3 of Lemma \ref{lem:MoserTheorem}),
for $(\vecu, v)\in \sF^{r+2}$,
\[
D\sP_{\vfcf}(\vecu,v)=(\vecu,\HtrR_\vfcf(\ut+\Str(v))\big|_{\Rectangle\times V}-\ut-v).
\]
From (\ref{EstimateofAtrRIdr})--(\ref{ExpressionofHtrRAug28})
we see that for $(\vfcf)\in\sN_{\delta_1(\Theta)}^{r+2}$, the map
 $D\sP_{\vfcf}$ can be extended to a map from $\sF^r$ to $\sF^r$, and the
 $L^{\infty}$ estimate (\ref{DsPC0EstimateMoserTheorem}) follows
 from Maximum Principle for leafwise harmonic function. So it remains to show
 the invertibility of $D\sP_{\vfcf}$ and estimate $D\sP^{-1}_{\vfcf}$.

 From now on, $(\vecu,v)$ will be in $\sF^r$.

We decompose $D\sP_{\vfcf}$ as
\begin{align}
D\sP_{\vfcf}(\vecu,v)&=D_1\sP_{\vfcf}(\vecu)+D_2\sP_{\vfcf}(v)\nonumber\\
				&=(\vecu,\HtrR_{\vfcf}(\ut)\big|_{\Rectangle\times V}-\ut)+(0, \HtrR_{\vfcf}(\Str(v))\big|_{\Rectangle\times V}-v).
\label{RealisticDecompositionofTangentialMap_201803151849}
\end{align}

Since $\sP(\vfcf)=(\bvarphi,\mB_{\bvarphi}(\phi)-\phi)$, we can also write $D\sP$ as
\begin{align}
D\sP_{\vfcf}(\vecu,v)=&(\vecu, D_1\mB_{\vfcf}(\vecu)+D_2\mB_{\vfcf}(v)-v)\nonumber\\
				=&(\vecu,v)\left(\begin{matrix}Id&D_1\mB_\vfcf\\ 0& D_2\mB_\vfcf-Id\end{matrix}\right). 		\label{MatrixFormofDsP}
\end{align}

By  (\ref{EstimateofAtrRIdr})--(\ref{ExpressionofHtrRAug28}),
$$
D_1\mB_\vfcf: (u_0, u_1)\rightarrow\HtrR_\vfcf(u_t)\big|_{\Rectangle\times V}-
u_t
$$
is bounded as a map from $C^r(V)\times C^r(V)$ to $\RectFunc^r$.   It follows
that $D\sP_\vfcf$, as map from $\sF^r$ to $\sF^r$, has bounded inverse if
and only if
\[D_2\mB_\vfcf-Id=\HtrR_{\vfcf}(\Str(\ \cdot\ ))\big|_{\Rectangle\times V}-Id,\]
as a map from $\RectFunc^r$ to $\RectFunc^r$, has a bounded inverse. And then
\[D\sP^{-1}_\vfcf(\vecu,v)=(\vecu, v)\left(
\begin{array}{cc}
Id&-D_1\mB_\vfcf(D_2\mB_\vfcf-Id)^{-1}\\
0&(D_2\mB_\vfcf-Id)^{-1}
\end{array}
\right).\]
Note that in above matrix, composition is taken from left to right.

In the following, we will show $D_2\mB_\vfcf$ is contraction map from $\RectFunc^r$ to $\RectFunc^r$ with respect to some weighted norm, so
\begin{equation}
(Id-D_2\mB_\vfcf)^{-1}=Id+D_2\mB_\vfcf+(D_2\mB_\vfcf)^2+(D_2\mB_\vfcf)^3+(D_2\mB_\vfcf)^4+......\label{ConvergentExpansionofInverse}
\end{equation}
The weighted norm will be
\begin{equation}\|\ \cdot\ \|=|\ \cdot\ |_r+A|\ \cdot\ |_{\twothird},					\label{DefinitionofWeightedNorm}
\end{equation}
with $A$ a large number to be determined. To determine $A$ we need to
provide an estimate of $D_2\mB_\vfcf$
with respect to $|\ \cdot\ |_{\twothird}$ and $|\ \cdot\ |_r$ norms, then
we will find that $A$ depending on $\vffnormr$.

Using (\ref{ExpressionofHtrRAug28}) and Theorem A.8 of \cite{HormanderPhysicalGeodesy}, we have
\begin{align}
	&|D_2\mB_{\vfcf}(v)|_{\twothird}\nonumber\\
=	&\left|\HtrR_{\vfcf}(\Str(v))\big|_{\Rectangle\times V}\right|_\twothird\nonumber\\
=	&\left|{(\AtrR_{\vfcf})}_\ast(\prodh(\AtrR_\vfcf^\ast(\Str(v)))\big|_{\Rectangle\times V})\right|_{\twothird}					\label{I0}\\
\leq &C\left|\prodh(\AtrR_{\vfcf}^\ast(\Str(v)))
\big|_{\Rectangle\times V}\right|_\twothird(1+\left|\AtrR_\vfcf^{-1}|_{\Rectangle\times V}\right|_\twothird)^\twothird
				\label{I1}.
\end{align}
In the estimate above, $C$ does not  depend on $\Theta$,
because, from (\ref{I0}) to (\ref{I1}),  functions involved in the composition are
functions on $\Rectangle\times V$.

Using (\ref{EstimateofAtrRiIdr}), we get:
\begin{equation}\left|\AtrR_\vfcf^{-1}|_{\Rectangle\times V}\right|_{\twothird}
\leq C(1+\left|\AtrR_\vfcf^{-1}-Id\right|_\twothird)
\leq C(1+C(\Theta)|\vfcf|_\fourthird)
													\label{AddEstimateofAtrRbyDiffwithIdSep27}.
\end{equation}
From Lemma A.2 we have
\begin{align}
\left|\prodh(\AtrR_{\vfcf}^\ast(\Str(v)))\big|_{\Rectangle\times V}\right|_\twothird
\leq \delta(\Theta) |\AtrR_\vfcf^\ast(\Str(v))|_{\twothird}						\label{I3ChangedSep27}	,
\end{align}
where $\delta(\Theta)$ is a number that tends to zero as $\Theta$ goes to $\infty$.

Then, using Theorem A.8 of \cite{HormanderPhysicalGeodesy}, we get
\begin{align}
 |\AtrR_\vfcf^\ast(\Str(v))|_{\twothird}	\leq C|\Str(v)|_\twothird\left(1+\left|\AtrR_\vfcf|_{(\tcC\times V)\cap\{|\theta|>\Theta-{1\over 2}\}}\right|_\twothird\right)^3.
\end{align}
In above, $C$ does not depend on $\Theta$, because $\overline{\Support(\Str(v))}$ is contained in $(\tcC\times V)\cap\{|\theta|>\Theta-{1\over 2}\}$, and $(\tcC\cap\{|\theta|>\Theta-{1\over 2}\})\times V$ is same as $(\cC\cap\left\{|\theta|>{1\over 2}\right\})\times V.$

Similar to (\ref{AddEstimateofAtrRbyDiffwithIdSep27}),
using (\ref{EstimateofAtrRIdr}) we obtain
\begin{align}
\left|\AtrR_\vfcf|_{(\tcC\times V)\cap\{|\theta|>\Theta-{1\over 2}\}}\right|_\twothird\leq C(1+|\AtrR_\vfcf-Id|_\twothird)\leq C(1+C(\Theta)|\vfcf|_\fourthird).\label{EstimateofAtrRbyDifferencewithIdSep28}
\end{align}

Combining the above estimates, we obtain
\begin{align}
\left|D_2\mB_\vfcf(v)\right|_\twothird\leq C\delta(\Theta)\left(1+C(\Theta)|\vfcf|_\fourthird\right)^6|v|_\twothird.\label{I4}
\end{align}

We can adopt a similar argument,
using(\ref{EstimateofAtrRIdr}) and(\ref{EstimateofAtrRiIdr}),
to get:
\begin{equation}
|D_2\mB_{\vfcf}(v)|_r\leq C\delta(\Theta)
\left(
1+C(\Theta)\vffnormfourthird
\right)^{2r}|v|_r
+
C(\Theta)(1+\vffnormr)|v|_{\twothird}
.\label{I5}
\end{equation}

Combining $(\ref{I4})$ and $(\ref{I5})$, and
 choosing $\Theta$ big enough so that
\[C\cdot \delta(\Theta)\cdot 2^{2r+6}<{1\over 2},\]
and then choosing $\tilde\delta_2$ small enough so that
\[C(\Theta)\cdot \tilde\delta_2<{1\over 2},\]
we have for $(\vfcf)\in \sN^{r+2}_{\min(\delta_1(\Theta),\tilde\delta_2)}$
\begin{equation}
|D_2\mB_\vfcf(v)|_\twothird\leq {1\over 2} |v|_\twothird,   \label{Contractionwrttwothord}
\end{equation}
\begin{equation}
|D_2\mB_{\vfcf}(v)|_r\leq {1\over 2}|v|_r+C(\Theta)(1+\vffnormr)|v|_{\twothird}.\label{Contractionwrtr}
\end{equation}

From now on, we will fix $\Theta$, so when there is no ambiguity, $C(\Theta)$ will be denoted by $C$. And we denote $\min(\delta_1,\tilde\delta_2)=\delta_2.$

Plugging (\ref{Contractionwrttwothord}) and (\ref{Contractionwrtr}) into the definition of weighted norm (\ref{DefinitionofWeightedNorm}) gives:
\begin{align}
   \|D_2\mB_\vfcf(v)\|
=	&|D_2\mB_\vfcf(v)|_r+A|D_2\mB_\vfcf(v)|_\twothird\nonumber\\
\leq &{1\over 2}|v|_r+\left(C(1+\vffnormr)+{A\over 2}\right)|v|_{\twothird}	.	
	\label{PseudoContractionwithWeightedNorm}
\end{align}

Choosing
\[A=6C(1+\vffnormr),\]
we get from (\ref{PseudoContractionwithWeightedNorm})
\[\text{Right Hand Side of (\ref{PseudoContractionwithWeightedNorm})}\leq {2\over 3}(|v|_r+A|v|_\twothird)=\frac{2}{3}\|v\|,\]

With this choice of  $A$, the norm $\|\ \cdot\ \|$ is determined,
so we can estimate $(Id-D_2\mB_\vfcf)^{-1}$:

\begin{align}
\left|(Id-D_2\mB_\vfcf)^{-1}(v)\right|_r&\leq \|(Id-D_2\mB_\vfcf)^{-1}(v)\|\nonumber\\
								&\leq 3\|v\|\nonumber\\
								&= 3|v|_r+3A|v|_\twothird\nonumber\\
								&= 3|v|_r+18C(1+\vffnormr)|v|_\twothird.
\end{align}

Then by direct computation,
\[|D\sP_{\vfcf}^{-1}(v)|_r\leq C(|v|_r+(1+\vffnormr)|v|_\twothird).\]

By Maximum Principle for leafwise harmonic function, we have
\[|D_1\mB_\vfcf(\vecu)|_0\leq C|\vecu|_0.\]

Moreover, $D_2\mB_\vfcf(v)$ is the restriction of a leafwise harmonic function,
$\HtrR_{\vfcf}(\Str( v ))$, to $\Rectangle\times V$. For any point
$p\in \Rectangle\times V$, there is  a unique leaf passing $p$.
We can then restrict $\HtrR_{\vfcf}(\Str( v ))$ to this leaf and project the
value to $\trR$. The resulting function, which we denote by  $h$ in the following,  is a harmonic function on $\trR$ that  equals to zero on $\trR\cap\{t=0,1\}$.
 Then, note that we have  $\Theta>4$, so
 we can extend $h$ by the odd
 reflection across the lines $\{t=0\}$ and $\{t=1\}$,
 to be a harmonic function in $\{-1<t<2\}\times \{-3<\theta<3\}$:
\[\exth(t,\theta)=\left\{
\begin{array}{ccc}
-h(2-t,\theta), &\text{ \ for }& 1<t;\\
h(t,\theta), &\text{ \ for }&0 \leq t\leq 1;\\
-h(-t,\theta), &\text{ \ for }& t<0.\\
\end{array}
\right.\]

Denote
\[M=\max\left\{\exth(t,\theta)\Big|-1\leq t\leq 2, -3\leq \theta\leq 3\right\},\]
\[m=\max\left\{\exth(t,\theta)\Big|-{1\over 2}\leq t\leq {3\over 2}, -2\leq \theta\leq 2\right\}.\]
The reflection construction implies that
\[-M=\min\left\{\exth(t,\theta)\Big|-1\leq t\leq 2, -3\leq \theta\leq 3\right\},\]
\[-m=\min\left\{\exth(t,\theta)\Big|-{1\over 2}\leq t\leq {3\over 2}, -2\leq \theta\leq 2\right\}.\]
Then we can apply Harnack inequality to $\exth+M$ and get there exists $\delta>0$, s.t.
\[M-m\geq \delta\cdot(M+m),\]
so,
\[|D_2\mB_\vfcf(v)(p)|\leq m\leq {1-\delta\over 1+\delta} M\leq {1-\delta\over 1+\delta} |v|_0.\]

Applying the above argument to every leaf, and
noting that $\delta$ does not depend on the leaf, 
we get
\[|D_2\mB_\vfcf(v)|_0\leq {1-\delta\over 1+\delta}|v|_0.\]
This shows that $D_2\mB_\vfcf$ is a contraction in $C^0$ norm. Using (\ref{ConvergentExpansionofInverse}) and (\ref{MatrixFormofDsP})
again, we conclude that $D\sP_\vfcf^{-1}$
is bounded with respect to $C^0$ norm.

Now Condition 4 of Lemma \ref{lem:MoserTheorem} is verified.

\subsection{Conclusion}\label{Conclusion}

Now given $\bvarphi\in C^k(V)\times C^k(V)$
with $|\bvarphi|_k$ small enough, we can apply Lemma \ref{lem:MoserTheorem} and find $\psi\in \RectFunc^{k-\jop}$, such that the Dirichlet Problem of HCMA equation on $\trR\times V$ with boundary data $\varphit+\Str(\psi)$ has solution $\Psi\in C^{k-J-2}(\overline\trR\times V)$ satisfying
\[\Psi|_{\tcC\times V}=\Str(\Psi).\]
So $\Psi$ can be extended to be a periodic function on $\sS\times V$, as argued in Section \ref{IterationFramework}. And since 
\[k-\jop>4+X,\]
according to (\ref{CrEstimateofPhiDisc}) and (\ref{NormEstimateinMoserTheoremOct10}), we have
\[|\Psi|_{k-\jop-2;\sS\times V}\leq |\psi|_{k-\jop}+|\varphi_0|_k+|\varphi_1|_k\leq C(V,\omega_0, k,\jop)\left(|\varphi_0|_k+|\varphi_1|_k\right),\]
so $|\Psi|_2$ small. Then we can apply the argument in Section \ref{IterationFramework}
and conclude  that $\Psi$ does not depend on $\theta$. 
So $\Psi=\psi+\varphi_t$, is a geodesic and 
\[\Psi\in C^{k-\jop}([0,1]\times V),\]
\[|\Psi|_{k-\jop;[0,1]\times V}\leq |\psi|_{k-\jop}+|\varphi_0|_k+|\varphi_1|_k\leq C(V,\omega_0, k,\jop)\left(|\varphi_0|_k+|\varphi_1|_k\right).\]


\appendix
\section{Families of Elliptic Problems}\label{FamilyofProblems}

\subsection{Family of Possion Equations}\label{FamilyofPossionEquations20180311jointHolder}
In this section, we prove the following lemma regarding the H\"older estimate of a family of Possion equations.

\begin{lemma}\label{lem:FamilyofPossionEquationwithBernsteinandExtension20180315}
Let $\ApU$ be a bounded region in $\EC$, with smooth boundary,
and let ${\ApBdelta}$ be the ball in $\ER^n$ with radius $\delta$ and center $0$.
We denote the complex coordinates on $\ApU$ by $\tau=u_1+\sqrt{-1}u_2$, and the
coordinates on $\ApB$ by $x^i$, for $i=1,...,n$.
Then for any $r\notin \EZ$, $r>0$, given
\begin{equation}\label{lem:FamilyofPossion-RegRhs}
f \in C^r(\overline\ApU\times \overline\ApB),\quad
\varphi\in C^r(\partial \ApU\times \overline\ApB)
\end{equation}
there exists a unique $h\in C^r(\overline\ApU\times \overline\ApB)$ satisfying:
\begin{align}
\triangle_\tau h=f, &\text{   in  \  } \ApU\times \overline\ApB;
\label{20180312equationofh}\\
h=\varphi,&\text{ on\  }\partial \ApU\times \overline\ApB.
\label{20180312equationofh-bc}
\end{align}
Here $\triangle_\tau$ stands for $\partial_{u_1}^2+\partial_{u_2}^2$.
Moreover, $h$ satisfies
\begin{equation}
|h|_{r; \ApU\times\ApB}\leq C(\ApU, r)(|\varphi|_{r;\ApU\times\ApB}+|f|_{r;\ApU\times\ApB}).
\label{20180311jointHolderEstimate}
\end{equation}
\end{lemma}
\begin{proof}
First we show the existence and uniqueness of solution. For each $x\in\ApB$,  we solve the
Dirichlet problem (\ref{20180312equationofh}), (\ref{20180312equationofh-bc})
for $h(\cdot, x)$ in the domain $\ApU$.
From (\ref{lem:FamilyofPossion-RegRhs}), we have $h(\cdot, x)
\in  C^{r}(\overline\ApU)\cap C^{r+2}(\ApU)$ for each $x\in\ApB$.
It remains to show (\ref{20180311jointHolderEstimate}).

We will show below that to prove the estimate
(\ref{20180311jointHolderEstimate}), we only need to estimate the directional
H\"older norms of $h$ in $\overline\ApU\times \overline\ApB$,
i.e. H\"older norms separately for variables on
$\ApU$ and on $\ApB$, and most importantly we need to prove the following estimates:
for some constant $C$ independent of $x\in \overline\ApB$,
and $\tau\in \overline \ApU$,
\begin{align}
\left|h(\cdot, x)\right|_r&\leq C(|\varphi|_r+|f|_r),
\text{\ \ \ \ \ \ \ \ \ \  \ \ \ \ \  for any } x\in \overline\ApB,
\label{20180311HolderonProductLeaf}\\
\left|h(\tau, \cdot)\right|_r&\leq C (|\varphi|_r+|f|_r),
\text{\ \ \ \ \ \ \ \ \ \  \ \ \ \ \  for any } \tau\in \overline \ApU.
\label{20180311HolderonManifold}
\end{align}

The estimate (\ref{20180311HolderonProductLeaf}) directly follows from
the Schauder estimates, using the fact that
$$
|\varphi(\cdot, x)|_r \le C |\varphi|_r, \quad
|f(\cdot, x)|_r \le C |f|_r \text{\ \ \ \ \ \ \ \ \ \  \ \ \ \ \  for any }
x\in \overline\ApB,
$$
where $C$ does not depend on $x$.

Next, we prove estimate (\ref{20180311HolderonManifold}). We will
 use the notation $r=m+\alpha$, for $m\in \EZ,\ 0<\alpha<1$.

 If $m\ge 1$,
 we first show the existence of $D_x^j h$ for $j=1,\dots, m$. Let $H_k$, $k=1,\dots, n$,
 be the solution of problem (\ref{20180312equationofh}), (\ref{20180312equationofh-bc})
 with the right-hand sides $D_{x_k} f$, $D_{x_k} \varphi$. Let
  $H=(H_1, \dots, H_k)$. Fix $x, \hat x \in \overline\ApB$.
Let
\begin{align*}
& q(\tau)= h(\tau, \hat x)- h(\tau, x)-H(\tau, x)\cdot(\hat x-x),\\
& F(\tau):= f(\tau, \hat x)- f(\tau,x)-D_xf(\tau, x)\cdot(\hat x-x),\quad
 \Phi(\tau):= \varphi(\tau, \hat x)- \varphi(\tau,x)-D_x\varphi(\tau, x)\cdot(\hat x-x).
\end{align*}
Then $q(\cdot)$ satisfies
$q\in C^{m-1,\alpha}(\overline\ApU)\cap C^{m+1,\alpha}(\ApU)$ with $m\ge 1$, and
\begin{equation}\label{20180311Syst_q}
\begin{split}
\triangle_\tau q=F, &\text{  \ \ \   in  \  } \ApU; 
\\
q=\Phi,&\text{ \ \ \  on\  }\partial \ApU. 
\end{split}
\end{equation}
From this, noting that $m\ge 1$, so $r\ge 1+\alpha$,
$$
|q|_{C^0(\overline\ApU)}\le C(\ApU)(|F|_{C^0(\overline\ApU)}+|\Phi|_{C^0(\partial\ApU)})
\le C(\ApU)(|f|_r+|\varphi|_r)|x-\hat x|^{1+\alpha}.
$$
It follows
 that $D_x h(\tau, x)$ exists for all
$(\tau, x)\in\overline\ApU\times \overline\ApB$, specifically $D_x h = H$. Then,
for each $x\in \overline\ApB$
\begin{align*}
|D_x h(\cdot,  x)|_{C^0(\overline\ApU)}&=|H(\cdot,  x)|_{C^0(\overline\ApU)}
\le C(\ApU)(|D_x f(\cdot,  x)|_{C^0(\overline\ApU)}+|D_x\varphi(\cdot,  x)|_{C^0(\partial\ApU)})\\
&\le
C(\ApU)(|D_x f|_{C^0(\overline\ApU\times \overline\ApB)}+
|D_x\varphi|_{C^0(\partial\ApU\times\overline\ApB)}),
\end{align*}
that is
$$|D_x h|_{L^\infty(\overline\ApU\times \overline\ApB)}
\le C(\ApU)(| f|_r+
|\varphi|_r).
$$
By a similar argument, $D_x^j h$ exist in $\overline\ApU\times \overline\ApB$
for each $j=1,\dots, m$, and
$$|D_x^j h|_{L^\infty(\overline\ApU\times \overline\ApB)}
\le  C(\ApU)(| f|_r+
|\varphi|_r).
$$

Next we consider the general case $m\ge 0$, and estimate the H\"older seminorm of $D_x^m h$. Denote
\begin{align*}
& q(\tau)=D_x^m h(\tau, \hat x)-D_x^m h(\tau, x),\\
& F(\tau):=D_x^m f(\tau, \hat x)-D_x^m f(\tau,x), \quad
 \Phi(\tau):=D_x^m \varphi(\tau, \hat x)-D_x^m \varphi(\tau,x).
\end{align*}
Then $q$, $F$, $\Phi$ satisfy (\ref{20180311Syst_q}), and we obtain
$$|q|_0\le C(\ApU)(|F|_{C^0(\overline\ApU)}+|\Phi|_{C^0(\partial\ApU)})
\le C(|f|_r+|\varphi|_r)|x-\hat x|^\alpha.$$
Now  (\ref{20180311HolderonManifold}) is proved.

Next we need to extend $h$ from $\overline \ApU\times \overline\ApB$ to
a larger open region, or, equivalently, to the
whole space $\EC\times \ER^n$, so that the extension satisfies estimates
(\ref{20180311HolderonProductLeaf})--(\ref{20180311HolderonManifold})
in the whole space.
We will do that in two steps, first extending from
$\overline \ApU\times \overline\ApB$ to $\EC\times \overline\ApB$,
and then to $\EC\times \ER^n$.

\newcommand{\Ext}{{\mathcal E}}
Let $\Ext$ be an extension operator, acting from $C^{\rho}(\overline\ApU)$ to
$C^{\rho}(\EC)$ for each $\rho\in [0, m+\alpha]$,
 such that $\Ext$ is a linear operator, satisfying
\begin{equation}\label{bddness-ext}
|\Ext[v]|_{C^{\rho}(\EC)} \le C(\rho)|v|_{C^{\rho}(\overline\ApU)}
\quad\mbox{for each }\;\rho\in [0, r],\;
 v\in C^{\rho}(\overline\ApU).
\end{equation}
For example, extension operator  defined in
section 6.9 of \cite{GT} satisfies these properties, where the linearity
follows from its explicit definition, and (\ref{bddness-ext})
follows from (6.94) of \cite{GT}. Define the function
$h_1$ on $\EC\times \overline\ApB$ by
$$
h_1(\tau, x)=\Ext[h(\cdot, x)](\tau)\quad\mbox{for }\; \tau\in \EC,\; x\in
\overline\ApB.
$$
Then, by (\ref{20180311HolderonProductLeaf}) and (\ref{bddness-ext}),
there exists constant $C$ such that
\begin{align}
\left|h_1(\cdot, x)\right|_r&\leq C(|\varphi|_r+|f|_r),
\text{\ \ \ \ \ \ \ \ \ \  \ \ \ \ \  for any } x\in \overline\ApB,
\label{20180311HolderonProductLeaf-ext1}
\end{align}
i.e. $h_1$ satisfies (\ref{20180311HolderonProductLeaf}) in $\EC$.

Next we show that $h_1$ satisfies (\ref{20180311HolderonManifold})
for each $\tau\in \EC$, i.e. that
\begin{align}
\left|h_1(\tau, \cdot)\right|_r&\leq C (|\varphi|_r+|f|_r),
\text{\ \ \ \ \ \ \ \ \ \  \ \ \ \ \  for any } \tau\in  \EC.
\label{20180311HolderonManifold-h1}
\end{align}
 The argument is similar to the proof of
(\ref{20180311HolderonManifold}) for $h$ above, with the use of
(\ref{bddness-ext}) instead of the estimates for Dirichlet problem in the
previous argument. We sketch this proof:

First, if $m\ge 1$,
we show the existence of $D_x^j h_1$ for $j=1,\dots, m$.
Fix $x, \hat x \in \overline\ApB$.
Denote $ H_1(\tau, x)=\Ext[D_x h(\cdot, x)](\tau)$.
Using the linearity of $\Ext$, we have:
\begin{align*}
 h_1(\cdot, \hat x)- h_1(\cdot, x)-H_1(\cdot, x)\cdot(\hat x-x)
 = \Ext\left[ h(\cdot, \hat x)- h(\cdot, x)-D_x h(\cdot, x)\cdot(\hat x-x)
 \right].
\end{align*}
Using (\ref{bddness-ext}) with $\rho=0$ to
estimate
the  $L^\infty(\EC)$-norm of the last expression,
we  obtain
\begin{align*}
|h_1(\cdot, \hat x)- h_1(\cdot, x)-H_1(\cdot, x)\cdot(\hat x-x)|_{L^\infty(\EC)}
&\le
 C| h(\cdot, \hat x)- h(\cdot, x)-D_x h(\cdot, x)\cdot
 (\hat x-x)|_{C^0(\overline\ApU)}.
\end{align*}
By (\ref{20180311HolderonManifold}), and noting that $r\ge m+\alpha\ge 1+\alpha$ , for any $\tau\in\overline\ApU$,
\begin{align*}
| h(\tau, \hat x)- h(\tau, x)-D_x h(\tau, x)\cdot
 (\hat x-x)|
& \le C|h(\tau, \cdot)|_r |x-\hat x|^{1+\alpha}
\le C(\ApU)(| f|_r+
|\varphi|_r)|x-\hat x|^{1+\alpha}.
 \end{align*}
 Combining the last two estimates,
 \begin{align*}
|h_1(\cdot, \hat x)- h_1(\cdot, x)-H_1(\cdot, x)\cdot(\hat x-x)|_{L^\infty(\EC)}
\le C(\ApU)(| f|_r+
|\varphi|_r)|x-\hat x|^{1+\alpha}.
 \end{align*}
We conclude that  $D_x h_1(\tau, x)$ exists for all
$(\tau, x)\in\overline\EC\times \overline\ApB$, specifically $D_x h_1 = H_1$.
Then, for each $ x\in \overline\ApB$, using (\ref{bddness-ext}) with $\rho=0$,
and then using (\ref{20180311HolderonProductLeaf}),
we have
\begin{align*}
|D_x h_1(\cdot,  x)|_{C^0(\EC)}&=|\Ext[D_x h(\cdot, x)]|_{C^0(\EC)}\\
&
\le C|D_x h(\cdot, x)|_{C^0(\overline\ApU)}\\
&\le
C(\ApU)(| f|_r+
|\varphi|_r).
\end{align*}
By a similar argument, we can show the existence and estimates
of $D_x^j h$ in $\EC\times \overline\ApB$
for each $j=1,\dots, m$, and moreover,
\begin{align}
\label{ExprExtDeriv-h1}
&D_x^j h_1(\cdot, x)=\Ext[D_x^j h(\cdot, x)],\\
\label{EstExtDeriv-h1}
&|D_x^j h_1|_{L^\infty(\EC\times \overline\ApB)}
\le  C(\ApU)(| f|_r+
|\varphi|_r)\quad\mbox{ for each }\; j=1,\dots, m.
\end{align}

Next, in the general case $m\ge 0$, from (\ref{ExprExtDeriv-h1}) and (\ref{bddness-ext}) with $\rho=0$,
 we  get that, for any $x, \hat x \in \overline\ApB$,
$$
|D_x^m h_1(\cdot, \hat x)-D_x^m h_1(\cdot, x)|_{L^\infty(\EC)}
=|\Ext[D_x^m h(\cdot, \hat x)-D_x^m h(\cdot, x)]|_{L^\infty(\EC)}
\le C|D_x^m h(\cdot, \hat x)-D_x^m h(\cdot, x)|_{C^0(\ApU)}.
$$
By (\ref{20180311HolderonManifold}), for any $\tau\in\overline\ApU$,
$$
|D_x^m h(\tau, \hat x)-D_x^m h(\tau, x)|
\le |h(\tau, \cdot)|_r |\hat x - x|^\alpha
\le C (|\varphi|_r+|f|_r) |\hat x - x|^\alpha.
$$
Combining the last two estimates, we obtain
$$
|D_x^m h_1(\cdot, \hat x)-D_x^m h_1(\cdot, x)|_{L^\infty(\EC)}
\le C(|\varphi|_r+|f|_r) |\hat x - x|^\alpha.
$$
From this and (\ref{EstExtDeriv-h1}) we obtain (\ref{20180311HolderonManifold-h1}).

Now we extend $h_1$ from $\EC\times \overline\ApB$ to $\EC\times \ER^n$.
We argue similarly as above:
\newcommand{\Extx}{\hat{\mathcal E}}
Let $\Extx$ be an extension operator, acting from $C^{\rho}(\overline\ApB)$ to
$C^{\rho}(\ER^n)$ for each $\rho\in [0, m+\alpha]$,
 such that $\Extx$ is a linear operator, satisfying
\begin{equation}\label{bddness-extxxx}
|\Extx[v]|_{C^{\rho}(\ER^n)} \le C(\rho)|v|_{C^{\rho}(\overline\ApB)}
\quad\mbox{for each }\;\rho\in [0, r],\;
 v\in C^{\rho}(\overline\ApB),
\end{equation}
and let
$$
h_2(\tau, x)=\Extx[h_1(\tau, \cdot)](x)\quad\mbox{for }\; \tau\in \EC,\; x\in
\ER^n.
$$
Then using (\ref{20180311HolderonProductLeaf-ext1}),
(\ref{20180311HolderonManifold-h1}) and (\ref{bddness-extxxx}),
and following the proof of (\ref{20180311HolderonProductLeaf-ext1}),
(\ref{20180311HolderonManifold-h1}), but reversing the roles
of $x$ and $\tau$ variables, we obtain the existence of $C$ such that
\begin{align}
\left|h_2(\cdot, x)\right|_r&\leq C(|\varphi|_r+|f|_r),
\text{\ \ \ \ \ \ \ \ \ \  \ \ \ \ \  for any } x\in \ER^n,
\label{20180311HolderonProductLeaf-ext2}
\\
\left|h_2(\tau, \cdot)\right|_r&\leq C (|\varphi|_r+|f|_r),
\text{\ \ \ \ \ \ \ \ \ \  \ \ \ \ \  for any } \tau\in  \EC.
\label{20180311HolderonManifold-h2}
\end{align}

Now we can apply   a theorem of Bernstein (Theorem 1 of \cite{KrantzOntology}) to
$h_2$,  and get the joint H\"older estimate (\ref{20180311jointHolderEstimate}).
\end{proof}

\begin{remark}
It is easy to see that in the statement of the lemma above we can replace $\ApB$
 by any compact smooth manifold, and get the corresponding results.
\end{remark}

\subsection{Family of  Riemann-Hilbert Problems with Constant Coefficients}\label{FamilyofConstantCoefficientRHProblem}
\begin{lemma}\label{lem:FamilyofConstantCoefficientRHProblem}
Let $B$ be the closed unit ball in $\EC^n$, $D$ be the unit disc in $\EC$.
We denote the complex coordinates on $D$ and $B$ by $\tau$ and $z^i$,
for $i=1,..., n$, respectively. Let $A=(A_{ij}), S=(S_{ij})\in
C^{\infty}(B\torange \EC^{n\times n})$
satisfy
\[\det A\neq 0. \]
Then for any $\rhb\in C^\rhrho(\partial D\times B\torange \EC^n)$, with $\rhrho>0$,
$\rhrho\notin \EZ$, there exists $(\rhf,\rhh)\in C^\rhrho(\cD\times B\torange \EC^n\times \EC^n)$ solving
\begin{align}
\partial_{\overline\tau}\rhf=\partial_{\overline\tau}\rhh=0,& \text{ in } D\times B;		 \label{HolomorphicEquationFamilyofRHProblemLemma}\\
A\overline\rhf+S\rhf-\rhh=\rhb, &\text{ on } \partial D\times B 	;		 \label{BoundaryConditionFamilyofRHProblemLemma}\\
\rhf(-\sqrt{-1},\cdot)=0.													\label{OnePtBoundaryConditionFamilyofRHProblemLemma}
\end{align}

Moreover,
\[|\rhf|_{\rhrho; \cD\times B}+|\rhh|_{\rhrho; \cD\times B}\leq C(\rhrho, A, S) |\rhb|_{\rhrho; \partial D\times B}.
\]
\end{lemma}
\begin{proof}
Combining (\ref{BoundaryConditionFamilyofRHProblemLemma}) with its complex conjugation, we get
\begin{equation}
\left(
\begin{array}{cc}
S&-I\\ \overline{A}& 0
\end{array}
\right)
\left(
\begin{array}{c}\rhf\\ \rhh
\end{array}
\right)+
\left(
\begin{array}{cc}
A&0\\ \overline{S}&-I
\end{array}
\right)
\left(
\begin{array}{c}\overline\rhf\\\overline \rhh
\end{array}
\right)=
\left(
\begin{array}{c}\rhb\\ \overline\rhb
\end{array}
\right)			.		\label{BoundaryConditionMatrixFormFamilyofRHProblemLemma}
\end{equation}
Note that
\[\left(
\begin{array}{cc}
A&0\\ \overline{S}&-I
\end{array}
\right)=
\overline{
\left(
\begin{array}{cc}
0&I\\ I& 0
\end{array}
\right)
\left(
\begin{array}{cc}
S&-I\\ \overline{A}& 0
\end{array}
\right)},
\]
so if $P\in \EC^{n\times n}$ is a constant matrix satisfying
\[{\overline P}^{-1}P=
\left(
\begin{array}{cc}
0&I\\ I& 0
\end{array}
\right),
\]
we can transform (\ref{BoundaryConditionMatrixFormFamilyofRHProblemLemma}) into
\[P \left(
\begin{array}{cc}
S&-I\\ \overline{A}& 0
\end{array}
\right)
\left(
\begin{array}{c}\rhf\\ \rhh
\end{array}
\right)+
\overline{ P
\left(
\begin{array}{cc}
S&-I\\ \overline{A}& 0
\end{array}
\right)
\left(
\begin{array}{c}\rhf\\ \rhh
\end{array}
\right)}=
P\left(
\begin{array}{c}\rhb\\ \overline\rhb
\end{array}
\right).\]
Such matrix $P$ obviously exists, for example
\[P=\left(\begin{array}{cc}I&I\\-\sqrt{-1} I& \sqrt{-1}I\end{array}\right).\]
Then, denoting
\[
\left(
\begin{array}{c}
\rhg_1\\ \rhg_2
\end{array}
\right)
=
\left(\begin{array}{cc}I&I\\-\sqrt{-1} I& \sqrt{-1}I\end{array}\right)
\left(
\begin{array}{cc}
S&-I\\ \overline{A}& 0
\end{array}
\right)
\left(\begin{array}{c}
\rhf\\
\rhh\end{array}
\right),
\]
we can reduce (\ref{HolomorphicEquationFamilyofRHProblemLemma}) (\ref{BoundaryConditionFamilyofRHProblemLemma}) (\ref{OnePtBoundaryConditionFamilyofRHProblemLemma}) to
\begin{align}
\partial_{\overline\tau}\rhg_1=\partial_{\overline\tau}\rhg_2=0,& \text{ in } D\times B;\\
\left(
\begin{array}{c}
\rhg_1\\ \rhg_2
\end{array}
\right)+
\overline
{\left(
\begin{array}{c}
\rhg_1\\ \rhg_2
\end{array}
\right)}=
\left(
\begin{array}{c}
\rhb+\overline\rhb\\ -\sqrt{-1}\;\rhb+\sqrt{-1}\overline\rhb
\end{array}
\right), &\text{ on } \partial D\times B;
\label{DirichletRiemannHilbert}\\
\rhg_1(-\sqrt{-1},z)=\rhb(-\sqrt{-1},z),\ \ \rhg_2(-\sqrt{-1},z)=-\sqrt{-1}\rhb(-\sqrt{-1},z), &\text {  for any } z\in B.
\end{align}

Then  functions $2\mbox{Re}\, \rhg_1$ and $2\mbox{Re}\, \rhg_2$ are harmonic with
respect to $\tau$-variables,
and have Dirichlet data given by the right-hand side of (\ref{DirichletRiemannHilbert}).
Now by Lemma \ref{lem:FamilyofPossionEquationwithBernsteinandExtension20180315}, we have
\[|\rhg_1,\rhg_2|_{\rhrho; \cD\times B}\leq C(\rhrho)|\rhb|_{\rhrho;\partial D\times B},\]
and so,
\[|\rhf,\rhh|_{\rhrho; \cD\times B}\leq C(\rhrho, A, S)|\rhb|_{\rhrho;\partial D\times B},\]
with $C(\rhrho, A, S)$ depending on $\sup\{\frac{1}{\det A}\}$ and $|A, S|_{\rhrho; B}$.\end{proof}

\subsection{Family of Harmonic Functions in a Long Strip}
\label{FamilyofDirichletProblemofHarmonicFunctiononLongStrip}
In this section, we adopt notations of Section \ref{Iteration}.

Given $r\geq0,$ we define $\lh: \RectFunc^r\rightarrow \RectFunc^r$ as following. For $F\in \RectFunc^r$, let $H$ be the function in $C^0(\ctrR\times V)$, satisfying
\[
\left\{
\begin{array}{cc}
\triangle_\tau H=0, & \text{ in } \trR\times V  ;\\
H=\Str(F),\ &\text{ on }{\partial \trR\times V}.
\end{array}
\right.
\]
Then
\[\lh(F):= H\big|_{\Rectangle\times V}.\]

\begin{lemma}\label{lem:onLongStripContractionMap}
For any $r>0$, $\Theta>2$, $\lh$ is a bounded map from $\RectFunc^r$ to $\RectFunc^r$, and there exists a $\delta(r,\Theta)$, such that for any $F\in \RectFunc^{r},$
\begin{equation}|\lh(F)|_r\leq \delta(r,\Theta)|F|_r,											\label{MainEstimateonNoodleSep28}
\end{equation}
and $\delta(r,\Theta)\rightarrow 0$, when we fix $r$ and let $\Theta\rightarrow \infty.$
\end{lemma}

Lemma \ref{lem:onLongStripContractionMap} is a generalization of the following Lemma \ref{lem:SingleStripContrctionSep28}. Before presenting Lemma \ref{lem:SingleStripContrctionSep28}, we need to introduce some notations, for the convenience of presentation.

For $0\leq r\leq\infty$, define
\[\Noodle^r=\{f\in C^r(\Rectangle\torange \ER)\big|\ f|_{\{t=0,1\}}=0\}.\]
given $f\in \Noodle^0$, we can find $h\in C^0(\ctrR\torange\ER)$, satisfying
\[\left\{
\begin{array}{cc}
\triangle h=0, &\text{ in }\trR;\\
h=\Str(f), &\text{ on }\partial\trR.
\end{array}
\right.\]
$\Str$ was defined at (\ref{DefinitionofStrSep28}) on $\RectFunc$, but we can also
consider it as an operator on $\Noodle^0 $ in the obvious way.
Then we define
\[\tlh(f)=h|_\Rectangle.\]
According to basic harmonic function theory, we know $\tlh(f)$ is in $\Noodle^\infty$, and we should have the following estimate

\begin{lemmasubsection}\label{lem:SingleStripContrctionSep28}
For any $r\geq 0$, $\Theta>2$, there exists a $\tilde\delta(r,\Theta)>0$, such that
\begin{equation}
|\tlh(f)|_r\leq\tilde\delta(r,\Theta)|f|_0,
\end{equation}
for any $f\in\Noodle^r$
and  $\tilde\delta(r,\Theta)\rightarrow 0$, as we fix $r$ and let $\Theta\rightarrow \infty.$
\end{lemmasubsection}
The proof is standard harmonic function theory, we can first control $|\tlh(f)|_0$ by using the following barrier function
\[w=2\sin(\frac{\pi}{4}+\frac{\pi}{2}t)\cdot \frac{e^{{\pi\over 2}\theta}+e^{-{\pi\over 2}\theta}}{e^{{\pi\over 2}\Theta}+e^{-{\pi\over 2}\Theta}}|f|_0,\]
then higher order estimates follow easily.

Now we can prove Lemma \ref{lem:onLongStripContractionMap} with Lemma \ref{lem:SingleStripContrctionSep28}.
The proof is similar to proof of Lemma \ref{lem:FamilyofPossionEquationwithBernsteinandExtension20180315}, but here we have better control on $\tau$-direction regularity, because here we only concern the regularity in $\Rectangle\times V$.

First, suppose that $\overline{\Support(f)}$ is contained in $B\subset V$, where there is a smooth map $\varphi: B\rightarrow  \EC^n$.
We denote the coordinates on $B$ by $\{z^\alpha\}_{\alpha=1}^{n}$.

For $j\in \EN$, $j\leq [r]$, we have that any $j$-th order manifold-direction-derivative of $H$, $D_z^jH$, satisfies:
\[\left\{
\begin{array}{cc}
\triangle_{\tau}(D_z^j H)=0, &\text{ in }\trR\times B;\\
D_z^j H=D_z^j F, &\text{ on }\partial \trR\times B.
\end{array}
\right.\]

So, we have, for any $k\geq 0$, there exists a $\tilde\delta(k,\Theta)$, such that
\begin{equation}|D_\trR^kD_z^j\lh(F)|_0\leq \tilde\delta(k,\Theta)|D_z^j F|_0\leq \tilde\delta(k,\Theta)|F|_{[r]},		\label{NoodleDirectionInfiniteControl}
\end{equation}
with $\tilde\delta(k,\Theta)\rightarrow 0$ as $\Theta\rightarrow \infty$.

And, for $r\notin \EZ, $ given $x, y\in B$, define
\[p=D_z^{[r]}H(x, \spacecdot)-D_z^{[r]}H(y, \spacecdot),\]
then $p$ satisfies
\[\left\{
\begin{array}{cc}
\triangle_\tau p=0, &\text{ in }\trR\times B;\\
|p|\leq |F|_{r}|x-y|^{r-[r]}, &\text{ on }\partial \trR\times B.
\end{array}
\right.\]

So, with Lemma \ref{lem:SingleStripContrctionSep28}, we can find  $\tilde\delta(r,\Theta)$, such that
\begin{equation}
\left|D_\trR^{[r]}\tilde\lh(F)(x,\spacecdot)-D_\trR^{[r]}\tilde\lh(F)(y,\spacecdot)\right|=\left|p|_{\Rectangle\times B}\right|\leq \tilde\delta(r,\Theta)|x-y|^{r-[r]}|F|_{r}.  							\label{HoldNoodleTogetherforHolder}
\end{equation}
with $\tilde \delta(r,\Theta)\rightarrow 0$ as $\Theta\rightarrow \infty$. Estimates of other
derivatives are easier, so for $F$ with support contained in a coordinate ball, (\ref{MainEstimateonNoodleSep28}) has been proved by combining (\ref{NoodleDirectionInfiniteControl}) (\ref{HoldNoodleTogetherforHolder}) together.
Now the global version follows easily.


\section{A Version of Moser's Inverse Function Theorem}\label{MoserTheorem}

Similar to \cite{HormanderonNashMoser}, our version of Moser's inverse function theorem will be presented in an abstract setting.

Let $\{\sF^\varrho\}_{\varrho\in[0,\infty)}$ be a family of Banach spaces,
with inclusion $\sF^{\varrho_1}\subset \sF^{\varrho_2}$, for $0\leq\varrho_1\leq \varrho_2$.
Furthermore, set $\sF^\infty=\cap_\varrho\sF^\varrho$
and assume that there exists a smoothing operator
 $S_Q:\sF^0\rightarrow\sF^\infty$, for any $Q\geq1$ such that for any $u\in \sF^\varrho$, with $\nu,  \varrho\leq U$, we have
\begin{align}
|S_Qu|_\nu\leq C(U)Q^{\nu-\varrho}|u|_\varrho, \ \text{ if } \nu\geq \varrho;			\label{MoserAbstractSettingBlowUpofSmoothing}\\
|S_Qu-u|_\nu\leq C(U)Q^{\nu-\varrho}|u|_\varrho, \ \text{ if }\nu\leq \varrho.			\label{MoserAbstractSettingConvergenceofSmoothing}
\end{align}
\begin{remark}\label{inrepIneqAbstrRmk}
(\ref{MoserAbstractSettingBlowUpofSmoothing}) and (\ref{MoserAbstractSettingConvergenceofSmoothing}) would imply for any $\kappa\leq\nu\leq \varrho\leq U$,
\begin{equation*}
|u|_\nu^{{\varrho-\kappa}}\leq C(U)|u|_\kappa^{{\varrho-\nu}}|u|_\varrho^{{\nu-\kappa}},
\end{equation*}
as explained in \cite{HormanderonNashMoser}.
\end{remark}

\begin{lemma}\label{lem:MoserTheorem}

Assume, that some fixed numbers $r,\ B,\ b,\ \chi,\ l,\ \alpha\in \ER^+$ satisfy:
\begin{equation}
{r}>B>B-\alpha>b>l\geq1,
	\label{IndexOrderMoserTheorem}
\end{equation}
\begin{equation}
\frac{B}{\chi}>\frac{\left(1+\frac{2l}{r}\right)^3(r-B)}{r-2l-\chi},		\label{BGreaterthanchi}
\end{equation}
\begin{equation}
\frac{B}{r-B}<1+\frac{2l}{r},        \label{MC_130_20180220_1752}
\end{equation}
\begin{equation}\label{labelB6-feldm}
\frac{B(r-b)}{b(r-B)}>\left(1+{2l\over r}\right)^2,
\end{equation}
\begin{equation}
\frac{r^3(r+l-B+\alpha)}{(r+2l)^3(r-B)}>\frac{B-\alpha}{B}\label{ConvergeWRTBalphaNorm}.
\end{equation}
And assume
 there is a map $\sP$ satisfies the following conditions, with
the parameters introduced above, and some fixed $C_0>1$:

\begin{description}
\item[Condition 1. ]
For some $\epsilon>0$, $\sP$ is defined in
\[\sN^b=\{f\in \sF^b\mid |f|_b<\epsilon\},\]
and  $\sP$ maps
$\sN^b$ into $\sF^{b-l},$ and maps
\[\sN^{r+l}=\sF^{r+l}\cap \sN^b,\]
into   $\sF^{r}$,
with
\begin{equation}
|\sP(f)|_r\leq C_0|f|_{r+l}. \label{CrEstimateofsPMoserTheorem}
\end{equation}
Note, in particular, that this implies
 \[\sP(0)=0.\]
\item[Condition 2. ]
$\sP$ is ``Lipschitz", in the sense that, for any $f_1, f_2\in\sN^b$,
\begin{equation}
|\sP(f_1)-\sP(f_2)|_0\leq C_0|f_1-f_2|_b.\label{LipEstimateofsPMoserTheorem}
\end{equation}

\item[Condition 3. ]
$\sP$ is  ``differentiable'', in the sense that there exists a map
\[D\sP\,:\,\sN^{r+l}\times\sF^{r+l}\rightarrow \sF^r,\]
which is linear with respect to second variable,
and for $f\in \sN^{r+l}, v\in \sF^{r+l}$, with $|v|_b$ small enough such that
$f+v\in \sN^{r+l}$, we have
\begin{equation}
|\sP(f+v)-\sP(f)-D\sP(f,v)|_0\leq C_0|v|_0^{2-{\chi\over r}}\left(|v|_{r+l}+(1+|f|_{r+l})|v|_{b}\right)^{\chi\over r}.
\label{SecondOrderEstimateMoserTheorem}
\end{equation}In the following, we denote  $D\sP_f(v)=D\sP(f, v)$.
\item[Condition 4. ]
For any $f\in \sN^{r+l}$, the map
 $D\sP_f$ can be extended to a linear map from $\sF^r$ to $\sF^r$  satisfying, for any $v\in \sF^r$
\begin{equation}
|D\sP_f(v)|_0\leq C_0|v|_0.\label{DsPC0EstimateMoserTheorem}
\end{equation}
Moreover,
 as a map from $\sF^r$ to $\sF^r$, $D\sP_f$ is invertible, with inverse satisfying:
\begin{equation}|D\sP^{-1}_f(v)|_r\leq C_0[(|f|_{r+l}+1)|v|_b+|v|_r],\label{InverseCrEstimateofDsPMoserTheorem}
\end{equation}
\begin{equation}
|D\sP_f^{-1}(v)|_0\leq C_0|v|_0.\label{InverseC0EstimateofDsPMoserTheorem}
\end{equation}
\end{description}

Then, there exist positive constants $\delta$ and $C$, depending only
on $r,B,b,\chi,l,\alpha,  C_0$, such that for any $h\in \sF^B$ with $|h|_B<\delta$,
there exists $f\in \sF^{B-\alpha}$ solving
\[\sP(f)=h.\]
Moreover, $f$ satisfies
\begin{equation}
|f|_{B-\alpha}\le C|h|_B.				
\label{NormEstimateinMoserTheoremOct10}
\end{equation}
\end{lemma}

\begin{remark}In (\ref{IndexOrderMoserTheorem}), ``$\geq 1$" is not essential, just for convenience of computation, actually ``$>0$" will be enough. And in our application we will let $l=2$.\end{remark}
\begin{remark} Continuity of $D\sP_f(v)$, with respect to first variable is not required, but weak continuity:
\begin{equation}
|D\sP_{f_1}(v)-D\sP_{f_2}(v)|_0\leq C \max\{|f_1-f_2|_b^{1\over 2}|v|_0^{1-{\chi\over 2r}}(|v|_{r+l}+(1+|f|_{r+l})|v|_{b})^{\frac{\chi}{2r}}, |f_1-f_2|_b\}
\end{equation}
can be derived from (\ref{SecondOrderEstimateMoserTheorem}) (\ref{DsPC0EstimateMoserTheorem}) and (\ref{LipEstimateofsPMoserTheorem}).
\end{remark}

{\flushleft\em Proof of Lemma \ref{lem:MoserTheorem}:}

Instead of directly targeting $h$, we choose a sequence of
smooth approximations of $h$ in $\sF^r$,
\[h_n=S_{N_n}h,\]
with $N_n$ to be determined.
As explained by J. Nash in \cite{NashImbedding}, the plan is:`` ... `feeding in' the smoother parts ... first, saving the rougher parts for later."
 By
(\ref{MoserAbstractSettingBlowUpofSmoothing}) and			
(\ref{MoserAbstractSettingConvergenceofSmoothing}),
these approximations satisfy
\begin{equation}
|h_n|_r\leq C|h|_BN_n^{r-B},			 \label{BlowuphwithNn}
\end{equation}
\begin{equation}
|h_n-h|_0\leq C|h|_BN_n^{-B},		\label{ApproximatehwithNn}
\end{equation}
where the constant $C$ is $C(r)$ from
(\ref{MoserAbstractSettingBlowUpofSmoothing}),		
(\ref{MoserAbstractSettingConvergenceofSmoothing}).

We will then construct a sequence $f_n\in \sF^{r+l}$,
in such way that, as $n\rightarrow \infty$,
\[|\sP(f_n)-h_n|_0\rightarrow 0.\]
At each step of iteration we correct $f_n$ by adding a smooth
approximation of $D\sP_{f_n}^{-1}(h_n-\sP(f_n))$.
Denote
\begin{equation}
F_n:=D\sP_{f_n}^{-1}(h_n-\sP(f_n)),\label{DefofFnSep8}
\end{equation}
and its smooth approximation
\begin{equation}
\label{DefofFnSep8-Approx}
v_n:=S_{M_n}(F_n),
\end{equation}
with $M_n$ to be determined.

Let
\begin{align}
f_1&=0,\\
f_{i+1}&=f_{i}+v_i, \text{ for } i\in \EZ^+,\label{inductiondefinitionoff_201803282347}
\end{align}
for $v_i$ defined by (\ref{DefofFnSep8-Approx}).
We show that, for some specifically chosen $0<\mu<1, K>1, A>>1, \lambda>0$, we can choose $M_i, N_i$ at each step of iteration, to make, for all $i\in \EZ^+$,
\begin{equation}
|h_i|_r\leq \mu e^{AK^i},											\label{Blowuph}
\end{equation}
\begin{equation}
|h_i-h|_0\leq{1\over 3}\mu e^{-\lambda AK^{i+1}},					\label{Approximateh}
\end{equation}
\begin{equation}
|f_i|_{r+l}\leq \mu e^{AK^i}	,									\label{CrlBdoffi}
\end{equation}
\begin{equation}
|\sP(f_i)-h_i|_0\leq \mu e^{-\lambda AK^i},							\label{C0EstimateofDistancetohi}
\end{equation}
\begin{equation}
|f_i|_b\leq\epsilon,										\label{fiTrappedinNeighborhood}
\end{equation}
and
\begin{equation}
\{f_n\}\text{ is a Cauchy sequence with respect to $C^{B-\alpha}$ norm.} 		 \label{Ckalphaconvergenceoffn}\end{equation}
	
We note that (\ref{fiTrappedinNeighborhood}) is equivalent to that $f_i$ stays in $\sN_{r+l}$, for all $i\in \EZ^+$.
To make (\ref{Blowuph})--(\ref{Ckalphaconvergenceoffn}) valid, we will specify
the sufficient conditions which involve only  $|h|_B$ and
the parameters
 $r$, $ B$, $ b$, $ l$, $ \chi$, $\alpha$, $\lambda$, $ \mu$, $ K$, $ A$,
 $ M_i$, $ N_i$,
 and $C_0$. Putting all requirements together, we find that if
 $r,\ B,\ b,\ l,\ \chi, \ \alpha$ satisfy
 (\ref{IndexOrderMoserTheorem})-(\ref{ConvergeWRTBalphaNorm}), then we can
 find $\lambda,\ \mu, \ K,\ A,$ such that at each step of iteration, we can
 choose $M_n$, $N_n$ so that (\ref{Blowuph})-(\ref{Ckalphaconvergenceoffn})
 are satisfied, and in order to have (\ref{fiTrappedinNeighborhood}) satisfied,
 we need $|h|_B$ small enough  in addition to the previous requirements.
 Also, constants $C$ below may depend on the parameters
  $r,\ B,\ b,\ \chi,\ l,\ \alpha$, and satisfy $C\ge 1$.

Our argument will be in 5 steps. From Step 1 to Step 5, after each requirement
is stated, we assume it is satisfied until the end of Step 5.
Then after Step 5, we show how to guarantee the validity of all requirements.

{\flushleft\bf  Step 1.} To make (\ref{Blowuph}) and (\ref{Approximateh}) valid,
we use (\ref{BlowuphwithNn})--(\ref{ApproximatehwithNn}), and
below $C$ is the constant from these estimates.
Then
 we require
\begin{equation}|h|_B\leq \mu,															\label{LowerBoundofmubyhB}
\end{equation}
and, for all $i\in \EZ^+$,
\begin{equation}
(3C)^\frac 1B e^{\frac{\lambda K}{B}AK^i}\leq N_i\leq
{1\over C^\frac 1{r-B}}e^{AK^i{1\over r-B}}.				\label{IntervalofNi}
\end{equation}

Then, by (\ref{IntervalofNi}), to make sure there is space left for $N_i$, we require
\begin{equation}
 e^{AK^i(\frac{1}{r-B}-\frac{\lambda K}{B})}\geq 3^\frac 1B C^{\frac 1B+\frac 1{r-B}}.								\label{LowerBoundofAfromh}
\end{equation}


{\flushleft\bf  Step 2.} Here we show how to guarantee (\ref{CrlBdoffi}) and
(\ref{C0EstimateofDistancetohi}) at $i=1$. When $i=1$,
\[f_1=\sP(f_1)=0,\]
so we only need to satisfy  (\ref{C0EstimateofDistancetohi}), which becomes now
\begin{equation}
|h_1|_0\leq \mu e^{-\lambda AK}.
\end{equation}
Using (\ref{Approximateh}), we obtain
\[|h_1|_0\leq |h|_0+|h_1-h|_0\leq |h|_0+\frac{1}{3}\mu e^{-\lambda AK^2},\]
so we require
\begin{equation}
2|h|_0e^{\lambda AK}\leq \mu.						\label{LowerBoundofmubyh0}
\end{equation}

{\flushleft\bf  Step 3.}
Now suppose at $i=n$, conditions
(\ref{CrlBdoffi}) and (\ref{C0EstimateofDistancetohi}) are satisfied, i.e.
\begin{equation}
|f_n|_{r+l}\leq \mu e^{AK^n}, 						\label{InductionLeftf}
\end{equation}
\begin{equation}
|\sP(f_n)-h_n|_0\leq \mu e^{-\lambda AK^n}.		\label{InductionLeftDistancetoh}
\end{equation}
Note that $f_{n+1}=f_{n}+v_n$ by (\ref{inductiondefinitionoff_201803282347}),
where $v_n$ is
given by (\ref{DefofFnSep8})--(\ref{DefofFnSep8-Approx}). We show how to choose $M_n$ such that
\begin{equation}
|f_{n+1}|_{r+l}\leq \mu e^{AK^{n+1}}, 						\label{InductionRightf}
\end{equation}
\begin{equation}
|\sP(f_{n+1})-h_{n+1}|_0\leq \mu e^{-\lambda AK^{n+1}}.	\label{InductionRightDistancetoh}
\end{equation}

We will first analyse (\ref{InductionRightf}) and then
(\ref{InductionRightDistancetoh}).
 To estimate  $|f_{n+1}|_{r+l}$,  we need to derive the estimates on $h_n-\sP(f_n)$ and $v_n$.

By (\ref{Blowuph}) (\ref{InductionLeftf}) and (\ref{CrEstimateofsPMoserTheorem}), we have
\begin{equation}
|h_n-\sP(f_n)|_r\leq\mu e^{AK^n}+C_0\mu e^{AK^n}\leq 2C_0\mu e^{AK^{n}}.
\end{equation}
Combining this estimate with interpolation inequality in Remark \ref{inrepIneqAbstrRmk}
and (\ref{InductionLeftDistancetoh}), we get
\begin{align}
|h_n-\sP(f_n)|_b&\leq C |h_n-\sP(f_n)|_r^{b\over r}\cdot|h_n-\sP(f_n)|_0^{1-\frac{b}{r}}\nonumber\\
				&\leq2CC_0\mu e^{AK^n\left(\frac{b}{r}-\lambda(1-\frac{b}{r})\right)},
\end{align}
where $C$ depends only on $r$. Thus
\begin{equation}
|h_n-\sP(f_n)|_b\leq 2CC_0\mu,
\end{equation}
provided
\begin{equation}
\lambda>\frac{b}{r-b}.  					\label{LowerBoundofLambdaSTAR1}
\end{equation}

Then with (\ref{InverseCrEstimateofDsPMoserTheorem}) and (\ref{DefofFnSep8}), we can estimate $|F_n|_r$ as
\begin{align}
|F_n|_r&=|D\sP^{-1}_{f_n}(h_n-\sP(f_n))|_r\nonumber\\
		&\leq C_0((|f_n|_{r+l}+1)|h_n-\sP(f_n)|_b+|h_n-\sP(f_n)|_r)\nonumber\\
		&\leq C_0((\mu e^{AK^n}+1)2CC_0\mu+2C_0\mu e^{AK^n})\nonumber\\
		&\leq 6C_0^2 C\mu e^{AK^n} .
\end{align}
Then $v_n=S_{M_n}(F_n)$ has the following estimates by
(\ref{MoserAbstractSettingBlowUpofSmoothing}) and			
(\ref{MoserAbstractSettingConvergenceofSmoothing}):
\begin{equation}
|v_n|_{r+l}\leq 6C_0^2C\mu e^{AK^n}M_n^l, 		\label{CrlBdofvnInitialForm}
\end{equation}
\begin{equation}
|v_n-F_n|_{0}\leq 6C_0^2C\mu e^{AK^n}M_n^{-r}.			\label{ApproximateFn}
\end{equation}

By (\ref{inductiondefinitionoff_201803282347}), the  induction hypothesis
(\ref{InductionLeftf}), and (\ref{CrlBdofvnInitialForm}), we have
\begin{align}
|f_{n+1}|_{r+l}&\leq |f_n|_{r+l}+|v_n|_{r+l}
					\leq \mu e^{AK^n}+6C^2_0C\mu e^{AK^n}M_n^l.			\label{CrlBdoffn+1InitialFormAugustSixth}
\end{align}
To satisfy (\ref{InductionRightf}), we need to have
\begin{equation}
\text{Right-Hand Side of } (\ref{CrlBdoffn+1InitialFormAugustSixth}){\leq} \mu e^{AK^{n+1}}.	\label{Requirementonfn1rl}	
\end{equation}	
This estimate is satisfied if we require
\begin{equation}
 e^{A(K-1)}\geq 2,											\label{WeakLowerBdofA}
\end{equation}
and
\begin{equation}
M_n^l\leq \frac{1}{12C_0^2C} e^{AK^n(K-1)}. 
\label{UpperBdofMn}
\end{equation}

To estimate the left hand side  of (\ref{InductionRightDistancetoh}), we need to
first  estimate $|v_n|_0$ using
(\ref{DefofFnSep8-Approx}).  From  (\ref{InverseC0EstimateofDsPMoserTheorem})
and (\ref{InductionLeftDistancetoh}) we get
\[|F_n|_0\leq C_0|h_n-\sP(f_n)|_0\leq C_0\mu e^{-\lambda AK^n},\]
then, with (\ref{ApproximateFn}), we obtain
\begin{align}
|v_n|_0\leq |F_n|_0+|v_n-F_n|_0\leq C_0\mu e^{-\lambda AK^n}+6C_0^2 C\mu e^{AK^n}M_n^{-r}.\label{C0BdofvnInitialForm}
\end{align}

We estimate $|h_{n+1}-\sP(f_{n+1})|_0$ as

\begin{align}
	|h_{n+1}-\sP(f_{n+1})|_0&\leq|h_{n+1}-h_{n}|_0+|\underbrace{h_n-\sP(f_n)-D\sP_{f_n}(F_n)}_{=0, \ \text{ by definition}}|_0+|D\sP_{f_n}(F_n-v_n)|_0\nonumber\\
&\ \ \ +|\sP(f_n)+D\sP_{f_n}(v_n)-\sP(f_{n+1})|_0.  \label{Link_2018_02_20_1737_MCTY}
\end{align}
We have the following estimates for each term on the right hand side of (\ref{Link_2018_02_20_1737_MCTY}):
\begin{itemize}
\item by (\ref{Approximateh})
\begin{equation}
|h_n-h_{n+1}|_0\leq \frac{2}{3}\mu e^{-\lambda AK^{n+1}},
\end{equation}
\item by (\ref{DsPC0EstimateMoserTheorem}) and (\ref{ApproximateFn})
\begin{equation}
|D\sP_{f_n}(F_n-v_n)|_0\leq C_0|F_n-v_n|_0\leq 6 C_0^3 C\mu e^{AK^n} M_n^{-r},\label{DsPApproximateFnInitialForm}
\end{equation}
\item by (\ref{SecondOrderEstimateMoserTheorem}),
\begin{align}
|\sP(f_n)+&D\sP_{f_n}(v_n)-\sP(f_{n+1})|_0\leq C_0|v_n|_0^{2-\frac{\chi}{r}}(|v_n|_{r+l}+(1+|f_n|_{r+l})|v_n|_{b})^{\chi\over r}
.\label{SecondOrderRequirementInitialForm}
\end{align}
\end{itemize}

So, to make (\ref{InductionRightDistancetoh}) valid, we require
\begin{equation}
\text{Right-Hand Side of } (\ref{DsPApproximateFnInitialForm})\leq \frac{1}{6}\mu e^{-\lambda AK^{n+1}},\label{RequirementonDsPApproximateFnInitialForm}
\end{equation}
\begin{equation}
\text{Right-Hand Side of } (\ref{SecondOrderRequirementInitialForm})\leq \frac{1}{6}\mu e^{-\lambda AK^{n+1}}.
\label{RequirementonSecondOrderRequirementInitialForm}
\end{equation}
To satisfy (\ref{RequirementonDsPApproximateFnInitialForm}), we require a stronger
inequality:
\begin{equation}
36C_0^3C e^{AK^n({1+\lambda K\over r})}\leq M_n.				\label{LowerBdofMn}
\end{equation}
To estimate the right hand side of (\ref{SecondOrderRequirementInitialForm}), and to guarantee the validity of (\ref{RequirementonSecondOrderRequirementInitialForm}), we need to provide some estimates of $v_n$.
Applying (\ref{UpperBdofMn}) to (\ref{CrlBdofvnInitialForm}), and
(\ref{LowerBdofMn}) to (\ref{C0BdofvnInitialForm}), we get 
\begin{equation}
|v_n|_{r+l}\leq \mu e^{AK^{n+1}}, 			\label{CrlBdofvnComplete}
\end{equation}
\begin{equation}
|v_n|_{0}\leq 2C_0 \mu e^{-\lambda AK^{n}}.\label{C0BdofvnComplete}
\end{equation}
These estimates imply, by interpolation formula on Remark \ref{inrepIneqAbstrRmk},
\begin{align}
|v_n|_b&\leq C|v_n|_0^{1-{b\over r+l}}|v_n|_{r+l}^{b\over r+l}\nonumber\\
							&\leq C(2C_0\mu e^{-\lambda AK^n})^{1-{b\over r+l}}(\mu e^{AK^{n+1}})^{b\over r+l}\nonumber\\
							&\leq 2C C_0\mu e^{{AK^n}(-\lambda+\frac{b\lambda}{r+l}+\frac{Kb}{r+l})}.
\label{CbInductionBoundofvn}
\end{align}
We require
\begin{equation}
\lambda>\frac{(K+\lambda)b}{r+l},			\label{lambdaWeakLowerBd}
\end{equation}
then from (\ref{CbInductionBoundofvn}) we have
\begin{equation}|v_n|_b\leq 2CC_0\mu,												\label{CbBoundofvn}
\end{equation}

For convenience of computation, we want to make
\begin{equation}
(|f_n|_{r+l}+1)|v_n|_b\leq \mu e^{AK^{n+1}}.						\label{AddBlowupBound}
\end{equation}
From (\ref{CbBoundofvn}) and  (\ref{InductionLeftf}),
we find that (\ref{AddBlowupBound}) is valid if we require
\begin{equation}
(\mu e^{AK^n}+1)2CC_0\mu\leq \mu e^{AK^{n+1}}.
\end{equation}
The last inequality can be derived from
\[4CC_0\leq  e^{A(K-1)}.\]

By (\ref{C0BdofvnComplete}), (\ref{CrlBdofvnComplete}) and (\ref{AddBlowupBound}),
we find that (\ref{RequirementonSecondOrderRequirementInitialForm}) is valid if
\begin{equation}
48C_0^3\leq e^{AK^n\left(\lambda(2-K)-\frac{(\lambda+K)\chi}{r}\right)}.   \label{LowerBdofB}
\end{equation}
Note that (\ref{LowerBdofB}) is satisfied for all $n\in \EZ^+$ if
\[\lambda(2-K)>\frac{(\lambda+K)\chi}{r},\]
and
\[48C_0^3\leq e^{AK\left(\lambda(2-K)-\frac{(\lambda+K)\chi}{r}\right)}.\]
{\flushleft\bf  Step 4.}
Now we give conditions which ensure that $f_n$ stays in $\sN_{r+l}$, i.e. we need
\begin{equation}|f_n|_b\leq\epsilon.			\label{ControllingfninNeighborhood}
\end{equation}
We use $\sum_{n=1}^{\infty}|v_n|_b$ as the bound of $|f_n|_b$. This sum can be estimated as
\begin{align}
\sum_{n=1}^\infty|v_n|_b
			\leq 2C C_0\mu\sum_{n=1}^\infty e^{{AK^n}(-\lambda+\frac{b\lambda}{r+l}+\frac{Kb}{r+l})},
\end{align}
where we have used (\ref{CbInductionBoundofvn}) and (\ref{lambdaWeakLowerBd}).
So to make (\ref{ControllingfninNeighborhood}) valid, we require
\begin{equation}
\mu\leq \frac{\epsilon(1-e^{A(K-1)\left(-\lambda+\frac{b\lambda}{r+l}+\frac{Kb}{r+l}\right)})}{2CC_0}.			\label{UpperBdofNormofh}
\end{equation}

{\flushleft\bf  Step 5.} Now we analyze the convergence of $f_n$
in  $C^{B-\alpha}$ norm.  Since, using again
(\ref{CrlBdofvnComplete}), (\ref{C0BdofvnComplete})
and interpolation, we have
\begin{align}
|v_n|_{B-\alpha}&=C|v_n|_0^{1-\frac{B-\alpha}{r+l}}|v_n|_{r+l}^{B-\alpha\over r+l}\nonumber\\
				&\leq C(2C_0\mu e^{-\lambda AK^n})^{1-{B-\alpha\over r+l}}(\mu e^{AK^{n+1}})^{B-\alpha\over r+l}\nonumber\\
				&\leq 2C C_0\mu e^{{AK^n}(-\lambda+\frac{\lambda(B-\alpha)}{r+l}+\frac{K(B-\alpha)}{r+l})},			\label{B-alphaNormEstimateOct10}
\end{align}
and $f_{n+1}=f_{n}+v_n$,
it follows that $f_n$ converges with respect to $C^{B-\alpha}$ norm if
\begin{equation}
\lambda>\frac{(B-\alpha)(\lambda+K)}{r+l}. \label{lambdaStrongerLowerBd}
\end{equation}
Note that  we need $B-\alpha>b$, so (\ref{lambdaStrongerLowerBd}) is stronger
than (\ref{lambdaWeakLowerBd}).

Below,  we collect the requirements stated in all steps, and show how to
satisfy all of them.
\begin{itemize}
\item {Step 1:}
\begin{equation}
(3C)^\frac 1B
e^{AK^n{\lambda K\over B}}\leq N_n\leq
\frac{1}{C^\frac 1{r-B}}e^{AK^n\frac{1}{r-B}},		\text{ for all } n\in \EZ^+,	\label{S11}
\end{equation}
\begin{equation}
|h|_B\leq \mu;																					\label{S12}
\end{equation}
\item {Step 2:}
\begin{equation}
2|h|_0e^{\lambda AK}\leq \mu;																\label{S2}
\end{equation}
\item {Step 3:}
\begin{equation}
\lambda>{b\over r-b},																			\label{S31}
\end{equation}
\begin{equation}
4CC_0\leq e^{A(K-1)},	        																		\label{S32}
\end{equation}
\begin{equation}
36C_0^3 C e^{AK^n\left(\frac{1+\lambda K}{r}\right)}\leq M_n\leq
\frac{1}{(12C_0^2C)^\frac 1l}e^{AK^n\left({K-1\over l}\right)}	,	\text{ for all } n\in \EZ^+	,	\label{S33}
\end{equation}
\begin{equation}
\lambda(2-K)>\frac{(\lambda+K)\chi}{r}														\label{S34-0.5}
\end{equation}
\begin{equation}
48C_0^3\leq e^{AK\left(\lambda(2-K)-\frac{(\lambda+K)\chi}{r}\right)};							\label{S34}
\end{equation}
\item {Step 4:}
\begin{equation}\mu\leq \frac{\epsilon(1-e^{A(K-1)\left(-\lambda+\frac{b\lambda}{r+l}+\frac{Kb}{r+l}\right)})}{2CC_0}; \label{S41}
\end{equation}
\item {Step 5:}
\begin{equation}\lambda>\frac{(B-\alpha)(\lambda+K)}{r+l}.							\label{S51}
\end{equation}
\end{itemize}
Now we discuss how to choose parameters to satisfy (\ref{S11})--(\ref{S51}).
Note that (\ref{IndexOrderMoserTheorem})-(\ref{ConvergeWRTBalphaNorm}) are
 satisfied. If we let
\[K=1+\frac{2l}{r}, \ \ \ \lambda=\frac{B}{K^2(r-B)},\]
then
(\ref{S31}) directly follows from (\ref{labelB6-feldm}), also
(\ref{S34-0.5}) follows from (\ref{BGreaterthanchi}), and
(\ref{S51})  follows from (\ref{ConvergeWRTBalphaNorm}).
Furthermore, we can choose $A$ big enough, such that
(\ref{S32}), (\ref{S34}) and
\begin{equation}
(3C)^\frac 1B e^{AK\frac{\lambda K}{B}}\leq
{1\over C^\frac 1{r-B}}e^{AK{1\over r-B}}.
\label{S11PuncturedStart}
\end{equation}
\begin{equation}
36C_0^3 C e^{AK\left(\frac{1+\lambda K}{r}\right)}\leq \frac{1}{12C_0^2C}e^{AK\left({K-1\over l}\right)},	\label{S33PuncturedStart}
\end{equation}
are satisfied. Note that
to get (\ref{S34}) we have used (\ref{S34-0.5}), and
to get (\ref{S33PuncturedStart}) we have used (\ref{MC_130_20180220_1752}).
Validity of (\ref{S11PuncturedStart}) and (\ref{S33PuncturedStart}) implies for all $n\in \EZ^+,$
\begin{equation}
(3C)^\frac 1B
e^{AK^n{\lambda K\over B}}\leq\frac{1}{C^\frac 1{r-B}}e^{AK^n\frac{1}{r-B}},			\label{S11Punctured}
\end{equation}
\begin{equation}
36C_0^3 C e^{AK^n\left(\frac{1+\lambda K}{r}\right)}\leq \frac{1}{12C_0^2C}e^{AK^n\left({K-1\over l}\right)}.	\label{S33Punctured}
\end{equation}
So, for each $n\in \EZ^+$, we can find $N_n, M_n$ satisfying (\ref{S11}) and (\ref{S33}).

Finally, if
\[|h|_B\leq\min\left\{ \frac{\epsilon(1-e^{A(K-1)
\left(-\lambda+\frac{b\lambda}{r+l}+\frac{Kb}{r+l}\right)})
e^{-\lambda AK}}{4CC_0}, \frac{1}{4}e^{-\lambda AK}\right\},\]
we can choose
\begin{align}\mu=2|h|_Be^{\lambda Ak}<1,   			\label{ValueofmuOct10}
\end{align} satisfying (\ref{S12}) (\ref{S2}) (\ref{S41}).

Now all parameters has been determined.

Being able to choose these parameters means that we can construct a sequence $f_n\in \sF^{r+l}$ (actually in $\sF^\infty$) and
$f\in \sF^{B-\alpha}$, such that as $n\rightarrow \infty$,
\[|f_n-f|_{B-\alpha}\rightarrow 0, \ \ |\sP(f_n)-h_n|_0\rightarrow 0, \ \ |h_n-h|_0\rightarrow 0,\]
which implies
\begin{align}
	&|\sP(f)-h|_0\nonumber\\
\leq&|\sP(f)-\sP(f_n)|_0+|\sP(f_n)-h_n|_0+|h_n-h|_0\nonumber\\
\leq&C_0|f_n-f|_b+\mu e^{-\lambda AK^n}+\frac{1}{3}\mu e^{-\lambda AK^{n+1}}\rightarrow 0,\nonumber
\end{align}
so,
\[\sP(f)=h.\]
Also, by (\ref{B-alphaNormEstimateOct10}) and (\ref{ValueofmuOct10}),
we have
\begin{equation*}
|f|_{B-\alpha}\leq C\cdot \mu<C|h|_B.				
\end{equation*}
This verifies (\ref{NormEstimateinMoserTheoremOct10}).
\begin{flushright}{$\blacksquare$}

\end{flushright}

\section*{Acknowledgements}
The research of Xiuxiong Chen was supported in part by the National Science Foundation under Grant DMS-1515795 and DMS-1603351.
The research of Mikhail Feldman was supported in part by the National Science Foundation under Grants DMS-1401490, DMS-1764278 and the Van Vleck Professorship Research Award by the University of Wisconsin-Madison.
The third named author Jingchen Hu was supported by the National Natural Science Foundation of China (grant no. 11571330 and 11271343) and the Fundamental Research Funds for the Central Universities, he wishes to thank his collaborators Jiyuan Han, Jingrui Cheng, Prof. Bing Wang of UW-Madison,  Guohuan Qiu, Bin Deng, Prof. Bin Xu of USTC and Long Li in Universit\'e de Grenoble Alpes  for very helpful discussion.

Xiuxiong Chen\\ School of Mathematical Sciences,  University of Science and Technology of China, Hefei, Anhui, China, 230026;\\
Institute of Mathematical Sciences, ShanghaiTech University, 393 Middle Huaxia Road, Pudong, Shanghai, China, 201210;\\
Math Department, Stony Brook University, Stony Brook, NY 11794-3660.\\Email address: xiu@math.sunysb.edu\\

Mikhail Feldman\\ Department of Mathematics, University of Wisconsion-Madison, Madison, WI, USA, 53705.\\
E-mail address: mfeldman2@wisc.edu\\

Jingchen Hu\\ Institute of Mathematical Sciences, ShanghaiTech University, 393 Middle Huaxia Road, Pudong, Shanghai, China, 201210.\\
E-mail address: jingchenhoo@gmail.com

\end{document}